\documentclass[11pt]{amsart}

\usepackage{amsmath,amsthm,amsfonts,amssymb,bm,graphicx }

\usepackage{graphicx}
\usepackage{enumitem}
\usepackage{xparse}
\usepackage{tikz}
\usepackage{comment}
\usepackage{bm}
\usepackage[colorinlistoftodos]{todonotes}
\usepackage{cite}
\usepackage{multirow}
\usepackage{pifont}
\usepackage{amsaddr}
\usepackage{orcidlink}
\makeatletter
\newcommand{\citecomment}[2][]{\citen{#2}#1\citevar}
\newcommand{\citeone}[1]{\citecomment{#1}}
\newcommand{\citetwo}[2][]{\citecomment[,~#1]{#2}}
\newcommand{\citevar}{\@ifnextchar\bgroup{;~\citeone}{\@ifnextchar[{;~\citetwo}{]}}}
\newcommand{\citefirst}{\@ifnextchar\bgroup{\citeone}{\@ifnextchar[{\citetwo}{]}}}

\makeatother

\makeatletter
\newcommand{\addresseshere}{%
  \enddoc@text\let\enddoc@text\relax
}
\makeatother

\newcommand{\cmark}{\ding{51}}
\newcommand{\xmark}{\ding{55}}

\usepackage
{hyperref}
\hypersetup{colorlinks=true,citecolor=blue,linkcolor=blue,urlcolor=blue}

\usepackage{geometry}
\theoremstyle{plain}
\newtheorem{theorem}{Theorem}
\newtheorem*{theorem*}{Theorem}
\newtheorem{lemma}[theorem]{Lemma}

\newtheorem{proposition}[theorem]{Proposition}

\theoremstyle{remark}

\newtheorem{example}{Example}

\numberwithin{theorem}{section}

\numberwithin{equation}{section}

\def\N{\mathbb N}
\def\Z{\mathbb Z}
\def\R{\mathbb R}
\def\Q{\mathbb Q}

\def\O{\mathcal O}
\def\P{\mathcal P}

\def\ve{\varepsilon}

\begin{document}

\author{V\'it\v{e}zslav Kala$^{1}$ \orcidlink{0000-0001-5515-6801} ,
Ester Sgallov\'a$^1$
\orcidlink{0009-0008-6985-9215},
Magdal\'ena Tinkov\'a$^{1,2,3,*}$ \orcidlink{0000-0003-0874-9705}}

\title[Arithmetic of cubic number fields: Jacobi--Perron, Pythagoras, and indecomposables]
{Arithmetic of cubic number fields: \\  Jacobi--Perron, Pythagoras, and indecomposables}

\email{vitezslav.kala@matfyz.cuni.cz, ester.sgallova@centrum.cz, tinkova.magdalena@gmail.com}

\address{$^1$Charles University, Faculty of Mathematics and Physics, Department of Algebra, Sokolovsk\'{a} 83, 18600 Praha 8, Czech Republic}

\address{$^2$Czech Technical University in Prague, Faculty of Information Technology, Th\'akurova 9, 160 00 Praha 6, Czech Republic}

\address{$^3$TU Graz, Institute of Analysis and Number Theory, Kopernikusgasse 24/II, 8010 Graz, Austria}

\thanks{$^*$ Corresponding author.}



\begin{abstract}
We study a new connection between multidimensional continued fractions, such as Jacobi--Perron algorithm, and additively indecomposable integers in totally real cubic number fields. First, we find the indecomposables of all signatures in Ennola's family of cubic fields, and use them to determine the Pythagoras numbers. Second, we compute a number of periodic JPA expansions, also in Shanks' family of simplest cubic fields. Finally, we compare these expansions with indecomposables to formulate our conclusions.
\end{abstract}

\setcounter{tocdepth}{1}  \maketitle 

\tableofcontents

\section{Introduction}

Motivated by Lagrange's 1770 result on the periodicity of the classical continued fraction for quadratic irrationalities, Hermite asked in 1849 whether there is a way of expressing cubic numbers with periodic expressions. Among the most influential and prominent approaches to this problem has been the \textit{Jacobi--Perron algorithm} (JPA, developed 1868 and 1907). Although it seems that JPA typically does not provide periodic expansions, recently Karpenkov \cite{Kar} used its modification to solve Hermite's problem in the totally real case.

The goal of this paper is to propose and explore a new connection between multidimensional continued fractions such as JPA and 
\textit{additively indecomposable integers} in totally real (cubic) number fields. Although these indecomposables are natural objects that were known and used already in 1945 by Siegel \cite{Si3}, they rose to a greater prominence only recently in connection with the study of the arithmetic of quadratic forms in number fields, e.g., the \textit{Pythagoras numbers}.
Let us now proceed by explaining the background  in more detail.

\medskip

The recent motivation for considering indecomposables comes from the study of universal quadratic forms over number fields. This topic has a long and beautiful history, starting with Lagrange's four square theorem which says that every positive integer can be written as the sum of four integer squares. In other words, the quadratic form $x^2+y^2+z^2+w^2$ is universal over $\Z$, i.e., it represents all elements of $\N(=\Z_{>0})$. This property was then studied for other (positive definite) quadratic forms and later extended to totally real number fields. In this case, coefficients of our quadratic form belong to  the ring $\O_K$ of algebraic integers of the field $K$. We say that such a form is \textit{universal} if it represents all elements of $\O_K^{+}$, the set of totally positive algebraic integers, i.e., of elements whose all conjugates are positive.

Every totally real number field possesses a universal quadratic form, which follows from the main theorem of \cite{HKK}, leading to the question of how many variables these universal forms must have. Blomer and Kala \cite{BK, Ka} showed that for every $M\in\N$, one can find infinitely many real quadratic fields which do not admit a universal quadratic form with less than $M$ variables. This result was then extended to cubic fields \cite{Ya}, multiquadratic fields \cite{KS}, and fields of degrees divisible by $2$ and $3$ \cite{Ka3}. In all these papers, indecomposable integers play an essential role.

We say that $\alpha\in\O_K^{+}$ is \textit{indecomposable} if one cannot express it as $\alpha=\beta+\gamma$ where $\beta,\gamma\in\O_K^{+}$. One of their first applications came in 1945 when Siegel \cite{Si3} proved that the sum of squares can be universal only over $\Q$ and $\Q(\sqrt{5})$ -- using  indecomposable integers (which he called {extremal elements}) in his proof. As we already mentioned, this connection was then crucial also in the aforementioned recent papers on more general universal forms (for more details, see the survey \cite{Ka survey}).

In the case of real quadratic fields $\Q(\sqrt{D})$, indecomposable integers are closely connected to the continued fraction of $\sqrt{D}$  (or $\frac{\sqrt{D}-1}{2}$ if $D\equiv 1\pmod 4$) \cite{Pe,DS}. Furthermore, we know their structure in several families of cubic fields \cite{KT,Ti2,GMT}, and some partial results for real biquadratic fields can be found in \cite{CLSTZ,KTZ, Ma}. Except for that, we do not have similar results about their structure for fields of higher degrees. However, Brunotte  \cite{Bru} proved that their norm is bounded, and one can take the discriminant of the field as this upper bound \cite[Theorem 5]{KY2}. However, this bound does not have to be sharp. Thus, it was further refined for quadratic fields in a series of articles \cite{DS,JK,Ka2,TV} and for several families of cubic fields \cite{Ti2,GMT}.

Before proceeding further, let us comment on our mention of \textit{families} in the last paragraph. In general, it is hard to study the properties of all number fields of given degree when we consider them ordered by, say, the discriminant. In particular, their class numbers or sizes of units behave quite erratically. However, the situation is much clearer in certain parametric families, such as Chowla's family of real quadratic fields $\Q(\sqrt{4a^2+1})$ \cite{CF}, Shanks' family of the simplest cubic fields \cite{Sh}  given as the rupture fields of polynomials $x^3-ax^2-(a+3)x-1$, or Ennola's family  of cubic fields corresponding to $x^3+(a-1)x^2-ax-1$ \cite{En}. In such a family, it is possible to express the (fundamental) units explicitly in terms of the parameter $a$ and to control the arithmetic behavior of these fields, including their indecomposables. 
Thus one gets an opportunity for proving results not just for a single number field, but rather for infinitely many of them at once -- yet in a setting that is simpler than dealing with all fields of given degree. However, one additional difficulty arises, as the nicest results may hold only for the order $\Z[\rho]$ generated by a root $\rho$ of the defining polynomial, rather than for the full ring of integers $\O_K$. Nevertheless, $\Z[\rho]$ typically equals the maximal order $\O_K$ for infinitely many values of the parameter $a$ (that include at least the squarefree values of some discriminant) -- and so, with this understanding, we will mostly work with the orders $\Z[\rho]$ in the present paper.

Besides totally positive indecomposables, it is also important to study indecomposables in other signatures $\mathfrak{s}$ (called \textit{$\mathfrak{s}$-indecomposables} in this paper), i.e., those integers which cannot be written as the sum of two elements with the same signature $\mathfrak{s}$. (For more details and precise definition, see Section \ref{sec:prelinf}.)
While these elements do not have a direct application to universal forms, they are key for determining the Pythagoras number (see below), as well as for our study of the Jacobi--Perron algorithm. 

In quadratic fields, $\mathfrak{s}$-indecomposables can be again constructed from usual continued fractions. The simplest cubic fields (which we also study in this paper) possess units of all signatures, and so their $\mathfrak{s}$-indecomposables can be obtained from the totally positive ones by multiplication by units. However, this is of course not true for all totally real number fields, and so representatives of $\mathfrak{s}$-indecomposables for the other signatures have to be dealt with separately. In this paper, we determine them for the family of {Ennola's cubic fields}.  
Besides the easy case of real quadratic fields, this is the first result on general  $\mathfrak{s}$-indecomposables.

\begin{theorem} \label{thm:m1}
Let $\rho$ be a root of the polynomial $x^3+(a-1)x^2-ax-1$ where $a\geq 3$. Then $1$, 
\begin{align*}
1+s\rho+\rho^2 \;&\text{ where } 1\leq s\leq a-1,\\
-v-(a(v-1)+w)\rho+(a(v-1)+w+1)\rho^2 \;&\text{ where } 1\leq v \leq a-1 \text{ and } \max\{1,v-1\}\leq w\leq a-1,\\
-1-u\rho+(u+2)\rho^2 \;&\text{ where } 0\leq u\leq a-2
\end{align*} 
are, up to multiplication by totally positive units, all the $\mathfrak{s}$-indecomposable integers in $\Z[\rho]$.
\end{theorem} 

We already hinted that $\mathfrak{s}$-indecomposable integers are connected to the Pythagoras number. As the sum of squares is almost never universal over $\O_K$, one can instead focus only on those algebraic integers which actually can be written as the sum of squares. To be more precise, let $\O$ be a commutative ring, and let $\sum \O^2$ be the set of elements in $\O$ which can be written as the sum of squares of elements from $\O$. Moreover, let $\sum^{m}\O^2\subseteq \sum\O^2$ contain all the elements of $\O$ which can be expressed as the sum of at most $m$ squares. Then
\[
\P(\O)=\inf\bigg\{m\in\N\cup\{\infty\}\mid \sum\O^2={\sum}^{m}\O^2\bigg\}
\] 
is called the \textit{Pythagoras number} of the ring $\O$.

While there are many results on Pythagoras numbers of fields \cite{CEP, CDLR, Ho, Pf, Pr},  
in this paper, we focus on orders $\O\subseteq\O_K$ for which only quite little is known. In this case, $\P(\O)<\infty$, but Scharlau \cite{Sc2} used a nice construction in  multiquadratic fields to prove that  $\P(\O)$ can attain arbitrarily large values. However, Kala and Yatsyna \cite{KY} found an upper bound: they showed the existence of the function $f$ depending only on the degree $d$ of the field such that $\P(\O)\leq f(d)$. Furthermore, for $2\leq d\leq 5$, we have $\P(\O)\leq d+3$.

The Pythagoras number for all orders in real quadratic fields was determined by Peters \cite{Pet}. In particular, $\P(\O)=5$ for all but a few specific cases. 
However, we have such a complete characterization only in the quadratic case. Kr\'asensk\'y, Ra\v{s}ka, and Sgallov\'a \cite{KRS} obtained some partial results for orders in real biquadratic fields, mostly concerning the maximal order $\O_K$. 
Some recent results on the Pythagoras number can be also found in \cite{Kr,HH, GMT}. 

We determine the Pythagoras number for Ennola's cubic fields by following the result of Tinkov\'a \cite{Ti1} who proved that $\P(\Z[\rho])=6$ where $\rho$ generates a simplest cubic field, i.e.,  $\rho$ is a root of the polynomial $x^3-ax^2-(a+3)x-1$ with $a\geq 3$. Her proof was heavily based on the knowledge of (totally positive) indecomposable integers. Specifically, using her method and Theorem \ref{thm:m1}, we show the following:

\begin{theorem} \label{thm:m2}
Let $\rho$ be a root of the polynomial $x^3+(a-1)x^2-ax-1$ where $a\geq 4$. Then $\P(\Z[\rho])=6$.
\end{theorem}

The first part of this paper is structured as follows. In Section \ref{sec:prelinf}, we review some (mostly) standard facts on totally real number fields, $\mathfrak{s}$-indecomposables, the method for finding them, and information about the simplest cubic fields and Ennola's cubic fields that are studied in this paper. Section \ref{Sec:indeennola} contains the proof of Theorem \ref{thm:m1} as well as several new results on $\mathfrak{s}$-indecomposables in Ennola's cubic fields. In Section \ref{sec:Pyth}, we proceed with the proof of Theorem \ref{thm:m2}. 

\medskip

In the second part, we turn our attention to the Jacobi--Perron algorithm and its connections with $\mathfrak{s}$-indecomposables. As we have mentioned before, for real quadratic fields $\Q(\sqrt{D})$ one can obtain all $\mathfrak{s}$-indecomposables using the continued fraction. 
Let us describe this relation in detail. For simplicity, let $D\equiv 2,3 \pmod 4$ be a squarefree positive integer (the case $D\equiv 1 \pmod4$ being similar), and let $\sqrt{D}=[u_0,\overline{u_1,u_2,\ldots, u_{s-1}u_{s}}]$ be the periodic continued fraction of $\sqrt{D}$. Moreover,  set
\[
\frac{p_i}{q_i}=[u_0,\ldots,u_i]
\] 
where $p_i$ and $q_i$ are coprime positive integers for each $i\in\N_0(=\Z_{\geq 0})$, and, moreover, let $p_{-1}=1$ and $q_{-1}=0$. Then the algebraic integers $\alpha_i=p_i+q_i\sqrt{D}$ are called \textit{convergents} (note that $\alpha_{s-1}$ is the fundamental unit). Moreover, we can define the elements $\alpha_{i,r}=\alpha_i+r\alpha_{i+1}$ where $0\leq r< u_{i+2}$ which we call \textit{semiconvergents}. Up to multiplication by $-1$ and conjugation, these semiconvergents cover all $\mathfrak{s}$-indecomposables in $\Q(\sqrt{D})$ \cite{Pe,DS}. For $i$ odd, we get the totally positive indecomposables, the case of even $i$ produces $\mathfrak{s}$-indecomposables in the signature $(+,-)$. 

To reach a similar result for fields of higher degrees, we cannot rely on the classical continued fraction as it is periodic only for quadratic irrationalities. And this motivation -- to find an algorithm that produces periodic expansions for all irrationalities of degrees greater than $2$, as proposed by Hermite -- started the interest in multidimensional continued fractions. However, the question of their periodicity is still open for most of them. Only for the algorithm recently introduced by Karpenkov \cite{Kar}, we know that it always gives the required periodic expansions in totally real cubic fields. This situation is partly caused by the fact that it is extremely hard to prove that some expansion is \textit{not} periodic. There is also some effort to define the usual continued fraction on algebraic integers -- see, e.g., \cite{MVV}.

In this paper, we study perhaps the most well-known of these algorithms, the \textit{Jacobi--Perron algorithm} (JPA; described in detail in Section \ref{sec:prelijpa}). This algorithm first appeared in the work of Jacobi \cite{Jac}, who was trying to find an algorithm that would solve Hermite's problem. He considered only the case of dimension $2$; his result was later extended to any dimension by Perron \cite{Pe2}. One of our motivations for choosing this algorithm (besides the pragmatic reason that JPA seems to work the best for our purposes) is that it gives a unit in the corresponding field, the so-called Hasse--Berstein unit \cite{BHunit}. However, in contrast to the classical continued fraction, this unit does not have to be fundamental \cite{Vo}, but in some cases, it is \cite{Ad,TB}.

Since the first appearance of JPA, many works discussed different families of periodic expansions. For example, Bernstein \cite{Be1,Be2,Be3} intensely studied JPA expansions of vectors $(\alpha,\alpha^2,\ldots,\alpha^{n-1})$ where $\alpha=\sqrt[n]{D}$ and $D $ satisfies some particular requirements. Moreover, we know that in every real  number field $K$ of degree $d$, we can find a $(d-1)$-tuple of elements in $K$ with periodic Jacobi--Perron expansion \cite{DPpern}. This expansion can even attain arbitrary period length \cite{LR,Vo}. For more information, we also refer to \cite{Bebook,Sch}.

In this paper, we study JPA expansions of specific vectors in the simplest \cite{Sh} and Ennola's cubic fields \cite{En,En1}. In Sections \ref{sec:jpasimplest} and \ref{sec:jpaennola} we show that these expansions are periodic. For example, we prove there the following theorem:

\begin{theorem} \label{thm:m3}
The Jacobi--Perron expansion of the vector $(1, |\tau|,\tau^2)$ is periodic whenever
\begin{enumerate}
    \item $\tau$ is a root of the polynomial $x^3+(a-1)x^2-ax-1$ where $a\geq 3$, or
    \item $\tau$ is a root of the polynomial $x^3-(a-1)x^2-ax-1$ where $a\geq 5$.
\end{enumerate}
\end{theorem}

Consequently, in Section \ref{sec:semi}, we propose the definitions \textit{convergents} and \textit{semiconvergents} which would be analogies of the same elements originating from the classical continued fraction. We discuss whether these elements are $\mathfrak{s}$-indecomposable; in some cases, they indeed are, but this is not always the case. For example, we show the following result as Theorem \ref{th.8.2}.

\begin{theorem}\label{th.1.4}
Let $\rho$, $\rho'$, $\rho''$, $\psi$, $\psi'$ and $\psi''$ be as in Subsections $\ref{subsec:ennola1polynomial}$ and $\ref{subsec:ennola2polynomial}$. Then
\begin{enumerate}
    \item all semiconvergents of JPA expansions of vectors $(1,-\rho',\rho'^2)$, $(1,-\psi',\psi'^2)$ and $(1,-\psi'',\psi''^2)$ are \emph{$\mathfrak{s}$-indecomposable} in their signatures,
    \item some semiconvergent of JPA expansions of vectors $(1,\rho,\rho^2)$, $(1,-\rho'',\rho''^2)$ and $(1,\psi,\psi^2)$ is \emph{$\mathfrak{s}$-decomposable} in its signature.
\end{enumerate}
\end{theorem}

To help resolve this ambiguity, in Section \ref{sec:exper}, we provide similar data for concrete examples of cubic fields not belonging to the above families. In Section \ref{sec:gensemi}, we generalize our definitions and study other linear combinations of convergents.

To conclude, it seems that the Jacobi--Perron algorithm sometimes gives nice results in relation to $\mathfrak{s}$-indecomposables. 
However, so far, we have not found such a close connection as in the case of the classical continued fraction. That was a reason why we also started to study other multidimensional continued fraction algorithms; we discuss some of our experiments in Section \ref{sec:othermulti}. 
We can also mention an alternative approach to these general algorithms in the paper of \v Rada, Starosta, and Kala \cite{RSK}.  


Finally, note that many of the arguments in Sections 6--9 are primarily lengthy, straightforward computations. In such cases, we often only explain the method and give an example in the main body of the paper, leaving most of the actual calculations to the Appendices A6--A9 \cite{KST}.
The computer programs that we used are available online at \url{https://sites.google.com/view/tinkovamagdalena/codes}.
This article is partly based on the theses \cite{Sg,Ti3}.

\section{Preliminaries on number fields} \label{sec:prelinf}

Let $K$ be a totally real number field of degree $d$, and let us denote by $\O_K$ its ring of algebraic integers. For an element $\alpha\in K$, we will define the \textit{trace} of $\alpha$ and the \textit{norm} of $\alpha$ as
\[
\text{Tr}(\alpha)=\sum_{i=1}^d\sigma_i(\alpha) \qquad\text{ and }\qquad N(\alpha)= \prod_{i=1}^d\sigma_i(\alpha) 
\]
where $\sigma_i$ are all the embedding of $K$ into $\R$.
The \textit{signature} of $\alpha\in K$ is the $d$-tuple of signs of the form
\[
\mathfrak{s}(\alpha)=\left(\text{sgn}(\sigma_1(\alpha)), \text{sgn}(\sigma_2(\alpha)),\ldots,\text{sgn}(\sigma_d(\alpha))\right)
\] 
where $\text{sgn}$ is the signum function. Throughout this paper, we will replace $\pm 1$ in $\mathfrak{s}(\alpha)$ by the symbols $+$ and $-$. Moreover, we will slightly abuse this notation and use the symbol $\mathfrak{s}$ also for a $d$-tuple of given signs. 

Let $\O\subseteq \O_K$ be an order in a number field $K$. Then $\O^\times$ denotes its group of units. We will use the symbol $\O^{\mathfrak{s}}$ to denote the subset of those elements in $\O$ which have the signature $\mathfrak{s}$. The element $\alpha\in\O^{\mathfrak{s}}$ is \textit{$\mathfrak s$-indecomposable} in $\O$ if it cannot be written as $\alpha=\beta+\gamma$ for any $\beta,\gamma\in\O^{\mathfrak{s}}$. Otherwise, we say that $\alpha$ is \textit{$\mathfrak{s}$-decomposable} in $\O$. The particular case of totally positive indecomposable integers, i.e., for which $\mathfrak{s}=(+,+,\ldots,+)$, was intensively studied in the past. Moreover, we say that elements $\alpha,\beta\in\O$ are \textit{associated} if there exists a unit $\varepsilon\in\O^\times$ such that $\alpha=\varepsilon\beta$. Similarly, we will also use the symbol $X^{\mathfrak{s}}$ to denote the set of all elements in $X$ with the signature $\mathfrak{s}$ for $X$ being an arbitrary subset of $K$. 

For $\O$, we define the \textit{codifferent} of $\O$ as
\[
\O^{\vee}=\{\delta\in K; \text{Tr}(\alpha\delta)\in\Z \text{ for all }\alpha\in \O\}.
\]
Moreover, if $\O=\Z[\eta]$ for some $\eta\in \O$ with the minimal polynomial $f$, we know that
\[
\O^{\vee}=\frac{1}{f'(\eta)}\Z[\eta]
\]
where $f'$ is the derivative of $f$ \cite[Proposition 4.17]{N}. It is easy to see that if for $\alpha\in\O^{\mathfrak s}$, there exists $\delta\in\O^{\vee,\mathfrak{s}}$ such that $\text{Tr}(\alpha\delta)=1$, then $\alpha$ is $\mathfrak{s}$-indecomposable in $\O$. However, the opposite implication is not true in general.  

\subsection{Method for the determination of $\mathfrak{s}$-indecomposables} \label{subsec:method}
Now we will describe a method for the determination of candidates on totally positive indecomposables which we developed in \cite[Section 4]{KT}, and show how we can modify it to get $\mathfrak{s}$-indecomposables for some other signatures $\mathfrak{s}$.

Let $\sigma:K\rightarrow\R^d$ be the diagonal embedding defined by $\sigma(\alpha)=(\sigma_1(\alpha),\sigma_2(\alpha),\ldots,\sigma_d(\alpha))$. Moreover, let us denote by $\R^{d,+}=\{(v_1,\dots,v_d)\in\R^d, v_i>0\text{ for all }i\}$ the totally positive octant of $\R^d$. In this octant, we can find images of all totally positive elements in $K$.

When we consider the fundamental domain for the action of multiplication by totally positive units in $K$, we know that this domain can be chosen as a polyhedric cone $\mathcal{Q}$ by Shintani's unit theorem \cite[Thm (9.3)]{Ne}. Recall that a polyhedric cone is a finite disjoint union of simplicial cones. It means that for every element $\alpha\in \O^{+}$, there exists a totally positive unit $\varepsilon\in\O^{\times,+}$ such that the image $\sigma(\alpha\varepsilon)$ belongs to one of these simplicial cones. We will use the symbol $\mathcal{C}(\alpha_1,\ldots,\alpha_e)=\R^{+}\alpha_1+\cdots+\R^{+}\alpha_e$ for these cones, and $\overline{\mathcal{C}}(\alpha_1,\ldots,\alpha_e)$ for their closures given by $\overline{\mathcal{C}}(\alpha_1,\ldots,\alpha_e)=\R_0^{+}\alpha_1+\cdots+\R_0^{+}\alpha_e$.

For the case on the totally real cubic fields, one possible choice of these simplicial cones was described by Thomas and Vasquez \cite[Theorem 1]{ThV}. Let $\varepsilon_1,\varepsilon_2\in\O^{\times,+}$ be two units satisfying some additional properties (in particular, they are \textit{proper} as defined in \cite{ThV}). Then
\begin{align*}
\mathcal Q&= \mathcal C(1,\ve_1,\ve_2)\sqcup \mathcal C(1,\ve_1,\ve_1\ve_2^{-1})\sqcup\mathcal C(1,\ve_1)\sqcup\mathcal C(1,\ve_2)\sqcup\mathcal C(1,\ve_1\ve_2^{-1})\sqcup\mathcal C(1)\\
&\subset
\overline{\mathcal C}(1,\ve_1,\ve_2)\cup \overline{\mathcal C}(1,\ve_1,\ve_1\ve_2^{-1}),
\end{align*}
where $\sqcup$ is disjoint union. Moreover, to get totally positive indecomposables, we can restrict to some bounded subsets of these two sets. More concretely, let $\mathcal{D}(\alpha_1,\alpha_2,\alpha_3)=[0,1]\alpha_1+[0,1]\alpha_2+[0,1]\alpha_3$; we will call such sets parallelepipeds. Then, if $\alpha\in\O^{+}$, there exists a totally positive unit $\varepsilon\in\O^{\times,+}$ such that $\alpha\varepsilon$ lies in $\overline{\mathcal C}(1,\ve_1,\ve_2)$ or $\overline{\mathcal C}(1,\ve_1,\ve_1\ve_2^{-1})$. If $\alpha\varepsilon\in\overline{\mathcal C}(1,\ve_1,\ve_2)$, then we can express it as $\alpha\varepsilon=t_1+t_2\ve_1+t_3\ve_2$ for some $t_1,t_2,t_3\in\R_0^+$. However, if $t_1>1$, then $\alpha\varepsilon$ is totally greater than $1$, which implies that it is decomposable. We can draw a similar conclusion for coefficients $t_2$ and $t_3$, and the same is also true for $\alpha\varepsilon$ lying in $\overline{\mathcal C}(1,\ve_1,\ve_1\ve_2^{-1})$. Therefore, all candidates on the indecomposable integers are contained either in $\mathcal{D}(1,\varepsilon_1,\varepsilon_2)$, or in $\mathcal{D}(1,\ve_1,\ve_1\ve_2^{-1})$. Thus, to get all totally positive indecomposables in $\O$ (up to multiplication by totally positive units), it is enough to determine all elements from $\O$ belonging to these two sets and decide which of them are indeed indecomposable.

This way, we can get all totally positive indecomposables in $\O$. Moreover, if the set $\O^{\times,\mathfrak{s}}$ is nonempty for a given signature $\mathfrak{s}$, then all $\mathfrak{s}$-indecomposables can be expressed as $\alpha\varepsilon$ where $\varepsilon\in\O^{\times,\mathfrak{s}}$ and $\alpha$ is a totally positive indecomposable. However, in our fields, it can occur that $\O^{\times,\mathfrak{s}}=\emptyset$ for some $\mathfrak{s}$. In that case, $\mathfrak{s}$-indecomposables are not associated with totally positive indecomposables. For them, we have to slightly modify our method. Let us take some element $\gamma\in\O^{\mathfrak{s}}$, and let $\varepsilon_1$ and $\varepsilon_2$ be as above. Then $\overline{\mathcal{C}} (\gamma, \gamma \ve_1, \gamma \ve_2) \cup \overline{\mathcal{C}} (\gamma, \gamma \ve_1, \gamma \ve_1 \ve_2^{-1})$
contains a fundamental domain for $\O^{\mathfrak{s}}$ modulo multiplication by $\O^{\times, +}$, i.e., for each element $\alpha \in \O^{\mathfrak{s}}$ there is (not necessarily unique) $\ve \in \O^{\times, +}$ such that $\alpha \ve \in \overline{\mathcal{C}} (\gamma, \gamma \ve_1, \gamma \ve_2) \cup \overline{\mathcal{C}} (\gamma, \gamma \ve_1, \gamma \ve_1 \ve_2^{-1})$. As before, if $\alpha \ve$ does not lie in $\mathcal{D}(\gamma, \gamma \ve_1, \gamma \ve_2) \cup \mathcal{D}(\gamma, \gamma \ve_1, \gamma \ve_1 \ve_2^{-1})$, then $\alpha$ is $\mathfrak{s}$-decomposable. For the finite set of elements from $\O^{\mathfrak{s}}$ obtained this way, we can decide similarly as before (using, e.g., the codifferent of $\O$) which of them are $\mathfrak{s}$-indecomposable.

\subsection{The simplest cubic fields} 
In this paper, we will study in detail two families of totally real cubic number fields. One of them are the so-called simplest cubic fields $\Q(\rho)$ which are generated by a root $\rho$ of the polynomial $f(x)=x^3-ax^2-(a+3)x-1$ with $a\geq -1$ \cite{Sh}. In this paper, we denote the roots of polynomial $f$ as $a+1<\rho$, $-2<\rho'<-1$, and $-1<\rho''<0$.
Moreover, if $a\geq 7$, then  
\begin{equation} \label{eq:estimates}
a+1<\rho<a+1+\frac{2}{a},\ \ -1-\frac{1}{a+1}<\rho'<-1-\frac{1}{a+2},\text{ and } -\frac{1}{a+2}<\rho''<-\frac{1}{a+3}.  
\end{equation}  
These bounds mostly comes from \cite{LP}, only the original estimate for $\rho'$ is slightly refined (see, e.g., \cite{Ti1}).

We know that $\O_K=\Z[\rho]$ in infinitely many cases of $a$, e.g., whenever the square root $a^2+3a+9$ of the discriminant of $K$ is squarefree. Moreover, the fields $\Q(\rho)$ are Galois. A system of fundamental units in $\Z[\rho]$ is formed by conjugates $\rho$ and $\rho'$ \cite{Th}. It follows that in these fields, we have units of all signatures, and, moreover, every totally positive unit is a square. 

Applying the method described in Subsection \ref{subsec:method}, we determined the structure of totally positive indecomposables as follows:

\begin{theorem}[{\cite[Theorem 1.2]{KT}}] \label{thm:indesimplest}
Let $K$ be the simplest cubic field with the parameter
$a\in\Z_{\geq -1}$. 
The elements $1$, $1+\rho+\rho^2$, and $-v-w\rho+(v+1)\rho^2$ where $0\leq v\leq a$ and $v(a+2)+1\leq w\leq (v+1)(a+1)$ are, up to multiplication by totally positive units, all the totally positive indecomposable elements in $\Z[\rho]$.
\end{theorem}

Note that Theorem 1.2 in \cite{KT} also contained the assumption that $\O_K=\Z[\rho]$. However, this assumption was not necessary anywhere in the proofs, and so Theorem 2.1 above holds as stated.
Moreover, this theorem also provides us all $\mathfrak{s}$-indecomposables $\beta$ for every signature $\mathfrak{s}$. Since $\Q(\rho)$ contains units of all signatures, every such $\beta$ can be expressed as $\beta=\alpha\varepsilon$ where $\alpha$ is one of the elements from Theorem \ref{thm:indesimplest}, and $\varepsilon$ is a unit with the signature $\mathfrak{s}$.

Moreover, in this article, we use the following notation:
\[
\theta_{v,W}=-v-(v(a+2)+1+W)\rho+(v+1)\rho^2
\]
where $0\leq v\leq a$ and $0\leq W\leq a-v$.

\subsection{Ennola's cubic fields}

The second family are Ennola's cubic fields \cite{En,En1}. They are generated by a root $\rho>1$ of the polynomial $x^3+(a-1)x^2-ax-1$ where $a\geq 3$. In this case, we denote the conjugates of $\rho$ as 
\begin{equation} \label{eq:estimatesennola}
1+\frac{1}{a+3}<\rho<1+\frac{1}{a+2}, \ \
-\frac{1}{a}<\rho'<-\frac{1}{a+1}, \text{ and }
-a+\frac{1}{a^2+a}<\rho''<-a+\frac{1}{a^2}.
\end{equation}
In Ennola's cubic fields, a system of fundamental units of $\Z[\rho]$ is formed by $\rho$ and $\rho-1$ \cite{Th}. Note that in some cases these units also give a system of fundamental units in the maximal order $\O_K$ \cite{En1}. They have the signature $(+,-,-)$. Thus, the order $\Z[\rho]$ contains only units of signatures $(+,-,-)$, $(-,+,+)$, $(+,+,+)$ and $(-,-,-)$. It also implies that in $\Z[\rho]$, we have two separate sets of $\mathfrak{s}$-indecomposables which are not associated by multiplying by units (of any signatures). 

The set of totally positive indecomposables was determined in \cite{KT} in the following way:

\begin{proposition}[{\cite[Proposition 8.1]{KT}}] \label{prop:indeennola}
Let $\rho$ be a root of the polynomial $x^3+(a-1)x^2-ax-1$ where $a\geq 3$. Then $1$ and $1+s\rho+\rho^2$ where $1\leq s\leq a-1$ are, up to multiplication by totally positive units, all the totally positive indecomposable elements in $\Z[\rho]$. 
\end{proposition}

Similarly as in the case of the simplest cubic fields, Proposition \ref{prop:indeennola} covers all $\mathfrak{s}$-indecomposables with the signatures $\mathfrak{s}=(+,+,+),(-,-,-),(+,-,-),(-,+,+)$. The remaining $\mathfrak{s}$-indecomposables in $\Z[\rho]$ are derived in Section \ref{Sec:indeennola}. Moreover, in this paper, we will adopt the following notation:
\begin{enumerate}
\item $\kappa_s= 1+s \rho + \rho^2 \text{ where } 1\leq s \leq a-1 $,
\item $\lambda_{v,w}=-v-(a(v-1)+w)\rho+(a(v-1)+w+1)\rho^2 \text{ where } 1\leq v \leq a-1 \text{ and } \max\{1,v-1\}\leq w \leq a-1$,
\item $\mu_u=-1-u\rho+(u+2)\rho^2 \text{ where } 0\leq u \leq a-2$.
\end{enumerate}
Note that these elements are representatives of all $\mathfrak{s}$-indecomposables introduced in Theorem \ref{thm:m1}. 

\section{Indecomposable integers in Ennola's cubic fields} \label{Sec:indeennola}

Following the method described in Subsection \ref{subsec:method}, we will find all the $\mathfrak{s}$-indecomposables for $\mathfrak{s}=(+,-,+)$ in Ennola's cubic fields. We have chosen this concrete signature since the element $\frac{1}{f'(\rho)}$ which takes a role in the codifferent of order $\Z[\rho]$, has this signature.
Note that $\mathfrak{s}$-indecomposables with this signature are associated to $\mathfrak{s}$-indecomposables with signatures $\mathfrak{s}=(-,+,-),(-,-,+),(+,+,-)$.

Note that in \cite{KT}, we used totally positive units $1$, $\rho^2$ and $\rho(\rho-1)$ to determine the structure of totally positive indecomposables in $\Z[\rho]$. Moreover, as we will discuss below in Subsection \ref{subsec:smallestnorm}, the element $-1+\rho^2$ of the signature $(+,-,+)$ has the smallest norm $2a-3$ (in absolute value) which is not attained by an element associated with a rational integer. Note that one can easily check that $-1+\rho^2$ has minimal polynomial $x^3-(a^2-2)x^2-(a-1)^2x+2a-3$.

We will now use the method for finding all $\mathfrak{s}$-indecomposables as described at the end of Subsection~\ref{subsec:method}; we take $\gamma=-1+\rho^2$, $\varepsilon_1=\rho^2$ and $\varepsilon_2=-1+\rho^2$ to define the corresponding two parallelepipeds.

\subsection{The first parallelepiped}
First of these parallelepipeds is generated by the triple $-1+\rho^2$, $(-1+\rho^2)\rho^2=-a+1-(a^2-a-1)\rho+(a^2-a)\rho^2$ and $(-1+\rho^2)\rho(\rho-1)=-a-(a^2-2)\rho+(a^2-1)\rho^2$. 

The first parallelepiped contains elements of the form 
\begin{align*} 
-1-w-w(a+1)\rho+(1+w(a+1))\rho^2 &\text{ where } 1\leq w\leq a-2,\\
-2-w-w(a+1)\rho+(2+w(a+1))\rho^2 &\text{ where }a-1\leq w\leq 2a-4.
\end{align*}
However, the elements in the second row can be rewritten as
\begin{multline*}
-2-w-w(a+1)\rho+(2+w(a+1))\rho^2=(-1-w_1-w_1(a+1)\rho+(1+w_1(a+1))\rho^2)\\+(-1-w_2-w_2(a+1)\rho+(1+w_2(a+1))\rho^2)
\end{multline*} 
where $w_1=a-2$ and $w_2=w-(a-2)$. Easily, $w-(a-2)\geq 1$ and $w-(a-2)\leq a-2$. Thus, these elements are $\mathfrak{s}$-decomposable, and it suffices to consider only the integers in the first row. 

\subsection{The second parallelepiped}

The second parallelepiped is generated by the triple $-1+\rho^2$, $(-1+\rho^2)\rho^2=-a+1-(a^2-a-1)\rho+(a^2-a)\rho^2$ and $(-1+\rho^2)\rho(\rho-1)^{-1}=\rho+\rho^2$. Thus we get the following inside elements:
\[
m+n\rho+o\rho^2=-t_1-(a-1)t_2+(-(a^2-a-1)t_2+t_3)\rho+(t_1+(a^2-a)t_2+t_3)\rho^2
\]
where $0\leq t_1,t_2,t_3<1$ and $m,n,o\in\Z$. As $-1+\rho^2$ has the norm $-(2a-3)$ and $\rho(\rho-1)^{-1}=1+a\rho+\rho^2$, we must have $t_1=\frac{t'_1}{2a^2-3a}$, $t_2=\frac{t'_2}{2a^2-3a}$ and $t_3=\frac{t'_3}{2a^2-3a}$ where $0\leq t'_1,t'_2,t'_3<2a^2-3a$, $0<t_2'$ and $t'_1,t'_2,t'_3\in\Z$. Moreover, we easily get the following congruences
\begin{align*}
t'_1&\equiv -(a-1)t'_2 \pmod{2a^2-3a},\\
t'_3&\equiv (a^2-a-1)t'_2 \pmod{2a^2-3a}.
\end{align*} 
Put $t=t'_2$. Moreover, let $t'_1=-(a-1)t+k(2a^2-3a)$ where $1\leq k\leq a-1$ is such that $(k-1)(2a-1)\leq t\leq k(2a-1)-1$. Similarly, let $t'_3=(a^2-a-1)t-l(2a^2-3a)$ for some $l\in\N_0$. Supposing this, we have
\[
(m,n,o)=(-k,-l,t+k-l).
\]

Moreover, we have $l+B=t+k-l$ for some $B\in\{1,2,3\}$, and that follows from the fact that
\[
B=-l+t+k-l=n+o=-(a^2-a-1)t_2+t_3+t_1+(a^2-a)t_2+t_3=t_1+t_2+2t_3
\]
where $0\leq t_1,t_2,t_3<1$, supposing that $B$ is an integer.
 
If $t\equiv k\pmod2$, then obviously $B=2$. Otherwise, we get the following two cases. If $t\leq (k-1)(2a-1)+2(k-1)$, then $B=3$. If $t\geq (k-1)(2a-1)+2k$, then $B=1$. Here, it is important to realize that $k$ and $(k-1)(2a-1)$ have different parities, so odd coefficients $B$ occur for those $t$ which are of the form $(k-1)(2a-1)+2s$.

Now we will discuss under which conditions the above elements are $\mathfrak{s}$-decomposable.

\begin{lemma} \label{lem:decom_B2}
If $k\geq 2$ and $B=2$, then elements $-k-l\rho+(t+k-l)\rho^2$ are $\mathfrak{s}$-decomposable.
\end{lemma} 

\begin{proof}
In this case, our elements can be decomposed as
\begin{multline*}
-k-l\rho+(t+k-l)\rho^2=(-k_1-l_1\rho+(s_1+k_1-l_1)\rho^2)+(-k_2-l_2\rho+(s_2+k_2-l_2)\rho^2)
\end{multline*}
where
\begin{align*}
k_1 &= k-1,
&l_1 &= (k-1)a-1,
&s_1 &= (k-1)(2a-1)-1,\\ 
k_2 &= 1,
&l_2 &= \frac{t - k(2a-1)+2a}{2},
&s_2 &= t - k(2a-1)+2a.    
\end{align*}
Note that both of these two elements lie in the second parallelepiped and have $B=1$. Moreover, it is important to realize that $1\leq s_2\leq 2a-2$ where the upper bound follows from the fact that $k$ and $k(2a-1)-1$ have different parities.   
\end{proof}

Note that in the previous lemma, we have omitted the case when $k=1$ which we will discuss later, and for which we obtain $\mathfrak{s}$-indecomposables.

We will proceed with $B=3$.

\begin{lemma}
If $B=3$, then elements $-k-l\rho+(t+k-l)\rho^2$ are $\mathfrak{s}$-decomposable. 
\end{lemma}

\begin{proof}
For $B=3$, our element can be decomposed as
\begin{multline*}
-k-l\rho+(t+k-l)\rho^2=(-k_1-l_1\rho+(s_1+k_1-l_1)\rho^2)+(-k_2-l_2\rho+(s_2+k_2-l_2)\rho^2)
\end{multline*}
where $-k_1-l_1\rho+(s_1+k_1-l_1)\rho^2$ is as in the proof of Lemma \ref{lem:decom_B2} and   
\begin{align*}
k_2 &= 1,
&l_2 &= \frac{t - k(2a-1)+2a-1}{2},
&s_2 &= t - k(2a-1)+2a.    
\end{align*}
Note that the second element has $B=2$.
In this case, we have $1\leq s_2\leq 2a-3$ as desired.
\end{proof}

Thus, we are left with the elements with $B=1$ which can be expressed as 
\[
-v-(a(v-1)+w)\rho+(a(v-1)+w+1)\rho^2
\]
where $1\leq v\leq a-1$ and $v\leq w\leq a-1$,
and the elements with $B=2$ and $k=1$ which can be rewritten as
\[
-1-u\rho+(u+2)\rho^2
\]
where $0\leq u \leq a-2$.

\subsection{Properties of indecomposables}

It is not difficult to see that the remaining elements from the first parallelepiped can be joined with the elements with $B=1$ from the second parallelepiped, and together form the set of elements of the form 
$-v-(a(v-1)+w)\rho+(a(v-1)+w+1)\rho^2$ where $1\leq v \leq a-1$ and $\max\{1,v-1\}\leq w\leq a-1$. Note that we exclude the element with $w=0$ which is $-1+\rho^2$ and was used in the determination of $\mathfrak{s}$-indecomposables, since its unit multiple is already contained in this set.

Now we will prove that all elements which we have not excluded so far, are $\mathfrak{s}$-indecomposable.   

\begin{proposition} \label{prop:indeennolaneg}
Let $\rho$ be a root of the polynomial $x^3+(a-1)x^2-ax-1$ where $a\geq 3$. Then 
\begin{align*}
-v-(a(v-1)+w)\rho+(a(v-1)+w+1)\rho^2 \;&\text{ where } 1\leq v \leq a-1 \text{ and } \max\{1,v-1\}\leq w\leq a-1,\\
-1-u\rho+(u+2)\rho^2 \;&\text{ where } 0\leq u\leq a-2
\end{align*} 
are, up to multiplication by totally positive units, all the $\mathfrak{s}$-indecomposable integers in $\Z[\rho]$ for $\mathfrak{s}=(+,-,+)$.
\end{proposition}

\begin{proof}
To prove this statement, we will use the element $\delta=\frac{1}{f'(\rho)}(\rho+a)\in\Z[\rho]^{\vee,\mathfrak{s}}$ from the codifferent. Easy computations show that if $\alpha=v_1+v_2\rho+v_3\rho^2\in\Z[\rho]$, then $\text{Tr}(\alpha\delta)=v_2+v_3$. Thus, for the elements $-v-(a(v-1)+w)\rho+(a(v-1)+w+1)\rho^2$, this trace is $1$ which immediately gives $\mathfrak{s}$-indecomposability of these elements.
  
Therefore, it remains to discuss elements $\mu_u=-1-u\rho+(u+2)\rho^2$, for which we have $\text{Tr}(\mu_u\delta)=2$. Note that these elements could decompose only as the sum of two elements $\alpha_1,\alpha_2\in\Z[\rho]^{\mathfrak{s}}$ such that $\text{Tr}(\alpha_i\delta)=1$ for $i=1,2$. Moreover, recall that $\frac{1}{f'(\rho)}$ has the same signature as our elements. For it, we get 
\[
\text{Tr}\left(\frac{1}{f'(\rho)}(v_1+v_2\rho+v_3\rho^2)\right)=v_3.
\] 
This implies that, if $\mu_u$ were $\mathfrak{s}$-decomposable, the only possible decompositions would be of the form
\[
-1-u\rho+(u+2)\rho^2=\underbrace{v_1-v_2\rho+(v_2+1)\rho^2}_{\alpha_1}+\underbrace{w_1-w_2\rho+(w_2+1)\rho^2}_{\alpha_2}
\]
where $0\leq v_2,w_2\leq u$.
Moreover, we can see that
\[
-v_2\rho'+(v_2+1)\rho'^2, -w_2\rho'+(w_2+1)\rho'^2>0  
\]
since $v_2,w_2\geq 0$ and $\rho'<0$. Thus, the elements $\alpha_1$ and $\alpha_2$ have the signature $(+,-,+)$ only if $v_1,w_1\leq -1$. However, under this condition, we can never obtain $v_1+w_1=-1$. Thus, our elements are $\mathfrak{s}$-indecomposable.    
\end{proof}

Thus, we have obtained the (almost triangular) set of $\mathfrak{s}$-indecomposables with the minimal trace $1$, and the line of remaining elements, for which we have found an element of the codifferent giving the trace $2$. Now we will show that this trace is minimal.

\begin{proposition}
Let $\mu_u=-1-u\rho+(u+2)\rho^2$ where $0\leq u\leq a-2$, and $\mathfrak{s}=(+,-,+)$. Then \[\min_{\delta\in\Z[\rho]^{\vee,\mathfrak{s}}}\textup{Tr}(\mu_u\delta)=2.\]
\end{proposition}

\begin{proof}
To prove this statement, we will use a similar method as used in \cite[Subsection 5.2]{Ti2}. Assume that $\textup{Tr}(\mu_u\delta)=1$ for some element $\delta=\frac{\delta_1}{f'(\rho)}\in\Z[\rho]^{\vee,\mathfrak{s}}$. It implies that $\delta_1\in\Z[\rho]$ is of the form
\[
\delta_1=((u+2)a-2)\frac{k}{2}-a\frac{l}{2}+\left(a+(u+2)\frac{k}{2}+(2a-1)\frac{l}{2}\right)\rho+(1+l)\rho^2
\]
for some $k,l\in\Z$. Note that $\delta_1$ must be totally positive since our $\mathfrak{s}$-indecomposables and $f'(\rho)$ have the same signature. Now we will discuss several traces which would be positive if $\delta_1$ were totally positive.

First of all, it is easy to compute that
\[
\text{Tr}\left(\frac{\delta_1}{f'(\rho)}(-1-\overline{u}\rho+(\overline{u}+2)\rho^2)\right)=(u-\overline{u})k+1.
\]  
If $0\leq \overline{u}\leq a-2$, then we know $-1-\overline{u}\rho+(\overline{u}+2)\rho^2$ is of the signature $(+,-,+)$ as $f'(\rho)$. It is also not difficult to verify (using estimates on $\rho,\rho'$ and $\rho''$, where we only need more precise lower bound on $\rho'$ in the spirit of Section \ref{sec:Pyth}) that this is also true for $\overline{u}=-1,a-1$. For the choices $\overline{u}=u-1$ and $\overline{u}=u+1$, we get the traces $k+1$ and $-k+1$. These two traces must be both positive which is possible only if $k=0$.

For $k=0$, the element $\delta_1$ lies in $\Z[\rho]$ only if $l=2l'$ for some $l'\in\Z$. If we consider the element $-1+\rho^2$ with the signature $(+,-,+)$, we get
\[
\text{Tr}\left(\frac{\delta_1}{f'(\rho)}(-1+\rho^2)\right)=-l'. 
\]
Thus, $l'<0$. On the other hand, if we consider the product of $-1+\rho^2$ and the totally positive unit $\rho(\rho-1)^{-1}=1+a\rho+\rho^2$, we obtain
\[
\text{Tr}\left(\frac{\delta_1}{f'(\rho)}(-1-\rho^2)(1+a\rho+\rho^2)\right)=2l'+2, 
\] 
which implies $l'>-1$. However, this is a contradiction with $l'<0$. Therefore, there is no totally positive $\delta_1$ satisfying $\textup{Tr}\left(\mu_u\frac{\delta_1}{f'(\rho)}\right)=1$ which gives the statement of the proposition.     
\end{proof}

\subsection{Sizes of units}

Every unit is $\mathfrak{s}$-indecomposable in its signature $\mathfrak{s}$. In this subsection, we will derive some properties of units in Ennola's cubic fields which we use in Section \ref{sec:Pyth} to find the Pythagoras number of $\Z[\rho]$, or to derive the smallest norm in Subsection \ref{subsec:smallestnorm}.

\begin{lemma} \label{lem:sizeunits}
Let $\varepsilon\in\Z[\rho]$ be a unit such that $|\varepsilon|,|\varepsilon'|,|\varepsilon''|<a$. Then $\varepsilon\in\{\pm 1,\pm \rho\}$.
\end{lemma}

\begin{proof}
Recall that a system of fundamental units in $\Z[\rho]$ is generated by the pair $\rho$ and $\rho-1$. Moreover, we can estimate conjugates of $\rho$ by
\begin{align*}
1+\frac{1}{a+3}&<\rho<1+\frac{1}{a+2},\\
-\frac{1}{a}&<\rho'<-\frac{1}{a+1},\\
-a+\frac{1}{a^2+a}&<\rho''<-a+\frac{1}{a^2}.
\end{align*}
Let $\varepsilon=\pm \rho^k(\rho-1)^l$ for some $k,l\in\Z$. If $k\geq 0$ and $l<0$, then clearly $|\varepsilon|>a$. If $k<0$ and $l\geq 0$, then $|\varepsilon'|>a$. Similarly, if $k\geq 0$ and $l>0$, then $|\varepsilon''|>a$. 

If $l = 0$ and $k\geq 2$, we obtain $|\varepsilon''|>a$. Thus, it remains to solve the case when $k<0$ and $l<0$. If $k<l$, it is not difficult to see that $|\varepsilon'|>a$. Likewise, $l<k$ leads to $|\varepsilon|>a$. Therefore, we are left with $k=l$. Moreover, it suffices to show that the minimal polynomial $(\rho(\rho-1))^{-1}$ has a root greater than $a$; in that case, every power of this element has a root $>a$. It is easy to check that the minimal polynomial of $\rho(\rho-1)$ is $x^3-(a^2+a)x^2+(2a+1)x-1$. Furthermore, this polynomial has a root in the interval $\big(0,\frac{1}{a}\big)$, from which the previous statement follows. Thus $k=0$ and $l=0$, and $k=1$ and $l=0$ are the only cases whose conjugates are all smaller than $a$ in absolute value.
\end{proof}

Moreover, Lemma \ref{lem:sizeunits} implies that if $\varepsilon$ is the square of some unit, then some of its conjugates is greater than $a^2$ unless $\varepsilon=1,\rho^2$; we will use this lemma in the determination of the Pythagoras number of $\Z[\rho]$.

\begin{lemma} \label{lem:sizesofunits2}
	Let $\ve\in\Z[\rho]$ be the square of a unit such that $\ve'>a^2$. If $\ve\neq \rho^{-2},\rho^{-2}(\rho-1)^2$, then at least one of the following holds:
	\begin{enumerate}
		\item $\ve'>a^4$, or
		\item $\ve>a^2$, or
		\item $\ve''>a^2$.
	\end{enumerate} 
\end{lemma}

\begin{proof}
Let $\ve=\rho^k(\rho-1)^l$ for some $k,l\in 2\Z$. If $k\geq 0$ and $l>0$, then we get
\[
\ve''=\rho''^k(\rho''-1)^l>\left(a-\frac{1}{a^2}\right)^k\left(a+1-\frac{1}{a^2}\right)^l>a^2.
\]
On the other hand, if $k\geq 0$ and $l<0$, then
\[
\ve=\rho^k(\rho-1)^l>\left(1+\frac{1}{a+3}\right)^k(a+2)^{-l}>a^2.
\]
In the case of $l=0$ and $k\geq 0$, we see that $\ve'<a^2$. Therefore, we have completely resolved the cases when $k\geq 0$.

If $k\leq -4$ and $k+2<l$, we can conclude that
\[ 
\ve'=\rho'^k(\rho'-1)^l\geq \rho'^{-4}>a^4.
\]
On the other hand, $l-2\leq k<0$ with $l<0$ leads to
\[
\ve=\rho^k(\rho-1)^l>\frac{1}{\left(1+\frac{1}{a+2}\right)^{-k}}(a+2)^{-l}\geq \frac{(a+2)^2}{\left(1+\frac{1}{a+2}\right)^4}>a^2.
\]
Thus, it remains to discuss the case when $k=-2$ and $l\geq 0$.
If $l\geq 4$, then clearly $\ve''>a^2$. Therefore, we are left with the cases $k=-2$ and $l=0$, and $k=-2$ and $l=2$ which are excluded in the statement of the lemma.
\end{proof}

\subsection{The smallest norm} \label{subsec:smallestnorm}

Using properties of $\mathfrak{s}$-indecomposables, we will find the minimal norm of elements in $\Z[\rho]$ that are not associated with a rational integer. Note that a similar result for the simplest cubic fields was provided by Lemmermeyer and Peth\"o \cite{LP}, who applied different methods.

\begin{proposition} \label{prop:smallestnorm}
Let $\alpha\in\Z[\rho]$. Then $|N(\alpha)|\geq 2a-3$ or $\alpha$ is associated to a rational integer. 
\end{proposition}

\begin{proof}  
Consider the sum $\beta+\gamma$ of two elements of the same signature $\mathfrak s$.
Then $$|N(\beta+\gamma)|=\left|\prod_i \sigma_i(\beta+\gamma)\right|=\left|\prod_i \sigma_i(\beta)\right|+\text{ other positive summands}>|N(\beta)|,$$
which follows from the fact that if we split the product $\prod_i \sigma_i(\beta+\gamma)$ into a sum of elements, then these summands have the same sign since $\beta$ and $\gamma$ have the same signature. Therefore, the absolute value of $\prod_i \sigma_i(\beta+\gamma)$ equals the sum of absolute values of these summands.      

Thus the minimal norm in $\Z[\rho]$ is attained by a non-unit $\mathfrak{s}$-indecomposable integer or by the sum of two units of the same signature. In the latter case, without loss of generality, we can consider only units of the signature $(+,+,+)$, where every such sum is associated with an element of the form $1+\varepsilon$ where $\varepsilon\neq 1$ is a totally positive unit. Since $\varepsilon\neq \rho$ as $\rho$ is not totally positive, some conjugate of $\varepsilon$ is greater than $a$ by Lemma \ref{lem:sizeunits}. The same is also true for $\varepsilon^{-1}$. Thus, we have
\[
N(1+\varepsilon)=N(1)+N(\varepsilon)+\text{Tr}(\varepsilon)+\text{Tr}(\varepsilon^{-1})\geq 2a+2>2a-3. 
\]
Since we know that the element $-1+\rho^2$ has absolute value of its norm $2a-3$, we can exclude sums of units from the consideration.

For the rest of the proof, let us first consider totally positive indecomposables of the form $1+s\rho+\rho^2$ where $1\leq s\leq a-1$. These elements have norms
\[
2a^2+2a+1+(a^2-3a-2)s-(2a-1)s^2+s^3,
\] 
see \cite[Section 4]{Ti2}.
The minimal norm among these elements can be attained only for some $s$ in the border of our interval, i.e., for $s=1$ or $s=a-1$, and for them, we get norms $3a^2+3a-1$ and $2a+3$ which are both clearly greater than $2a-3$ for $a\geq 3$.

Let us now focus on the elements $-1-u\rho+(u+2)\rho^2$. The absolute value of their norm is equal to
\[
2a^2+8a-17+(a^2+5a-14)u-7u^2-u^3.
\]
As in the previous case, the minimal norm can be obtained only at one of the border points. Here we get norms $2a^2+8a-17$ and $4a^2-9$ which are greater than $2a-3$ for $a\geq 3$.

Thus, it remains to discuss norms of elements $-v-(a(v-1)+w)\rho+(a(v-1)+w+1)\rho^2$ where $1\leq v \leq a-1$ and $\max\{1,v-1\}\leq w\leq a-1$. The absolute value of their norm is equal to
\begin{multline*}
 -w^3 - ((a - 1) v - 2 a + 4) w^2 + (av^2 + (a^2 - 5 a - 1) v - (a^2 - 7 a + 3)) w \\+ v^3 - (a^2 + a + 1) v^2 + (4 a^2 - 2) v - 3 a^2 + 3 a - 1.
\end{multline*}
It can be shown that for $v\geq 2$ and $w=v-3$, and for $v=1$ and $w=-1$, this expression is negative. Thus, for fixed $v$, it takes its smallest value in $w=v-1$ or $w=a-1$. Moreover, in both of these cases of $w$, the coefficient before $v^3$ is positive and, moreover, this norm is negative for $v=a$. Thus, our minimum can be again attained only in the border points $v=1$ or $v=a-1$. We obtain norms $2a-3$, $a^2+3a-9$ and $a^2-a-1$, from which $2a-3$ is the smallest.
\end{proof}

\section{Pythagoras number in Ennola's cubic fields} \label{sec:Pyth}

Now we will focus on the determination of the Pythagoras number of the order $\Z[\rho]$ where $\rho$ is a root of the polynomial $x^3+(a-1)x^2-ax-1$ with $a\geq 3$. Before that, let us recall the definition of the Pythagoras number. Let $\O$ be a commutative ring, and let
\[
\sum\O^2=\Big\{\sum_{i=1}^n \alpha_i^2;\; \alpha_i\in\O, n\in\N\Big\}.
\]
The set $\sum\O^2$ contains all elements from $\O$ which can be written as the sum of squares. Similarly, let
\[
\sum^m\O^2=\Big\{\sum_{i=1}^m \alpha_i^2;\; \alpha_i\in\O\Big\},
\]
i.e., the set $\sum^m\O^2$ contains all elements from $\O$ which are expressible as the sum of at most $m$ squares of elements from $\O$. Then the Pythagoras number of $\O$ is 
\[
\P(\O)=\inf\Big\{m\in\N\cup\{\infty\};\; \sum \O^2=\sum^m \O^2\Big\}.
\]
Recall that in our case of $\O\subseteq\O_K$, the value of $\P(\O)$ is always finite. 

From \cite[Corollary 3.3]{KY}, we know that $\P(\Z[\rho])\leq 6$. In this section, we show that $\P(\Z[\rho])= 6$ for $a\geq 4$. To do that, we use the method introduced in \cite[Section 3]{Ti1}, where Tinková determined the Pythagoras number for an analogous order in the simplest cubic fields. To apply this method, it is enough to know the structure of $\mathfrak{s}$-indecomposables in $\Z[\rho]$. Recall that the method consists of the following steps:
\begin{enumerate}
\item Suitably choose an element $\gamma\in\Z[\rho]$ which can be written as the sum of six non-zero squares.
\item Find all representatives of $\mathfrak{s}$-indecomposables $\alpha$ such that $N(\alpha)^2\leq N(\gamma)$. However, this step is not necessary, and we will skip it in some parts.
\item For all $\mathfrak{s}$-indecomposables $\alpha\in\Z[\rho]$ from Step (2), find all units $\varepsilon\in\Z[\rho]$ such that $\gamma\succeq (\varepsilon\alpha)^2$.
\item From $\mathfrak{s}$-indecomposables obtained in Step (3), construct all $\mathfrak{s}$-decomposables $\omega\in\Z[\rho]$ such that $\gamma\succeq \omega^2$. Note that every decomposition of $\gamma$ to the sum of squares can consist only of these squares. 
\item Discuss all possible decomposition of $\gamma$ as the sum of squares of elements derived in Step (4). 
\end{enumerate}

However, in the simplest cubic fields, we have units of all signatures, so all $\mathfrak{s}$-indecomposables are associated with totally positive indecomposables. Ennola's cubic fields do not have this useful property. Thus, to apply this method, we can benefit from the determination of $\mathfrak{s}$-indecomposables which we have done in Section \ref{Sec:indeennola}. Moreover, we will sometimes use the fact that if $\gamma\succeq \omega^2$, then $\text{Tr}(\gamma)\geq \text{Tr}(\omega^2)$ and $N(\gamma-\omega^2)\geq 0$.

To find the lower bound on the Pythagoras number of $\Z[\rho]$, we will use the element
\[
\gamma=a^2-3a+11-(a^2-5a+1)\rho-(a-5)\rho^2=1+1+1+4+(1+\rho)^2+(a-2-(a-1)\rho-\rho^2)^2.
\]
Using estimates on conjugates of $\rho$, we see that for $a\geq 6$, we have
\begin{align*}
\gamma &< 20+\frac{a^2-4a-13}{(a+3)^2}<21 <a^2,\\
\gamma' &< a^2-2a+6+\frac{7a+1}{a(a+1)^2}<a^2,\\
\gamma''& <a^2-2a+13-\frac{a^5+5a^4+11a^3+a^2-4a-5}{a^4(a+1)}<a^2.   
\end{align*}
It immediately implies that if $\gamma\succeq \omega^2$ where $\omega\in\Z$, then $\omega^2=1,4,9,16$.

First of all, we will find all squares of units which are totally smaller than $\gamma$.

\begin{lemma} \label{lem:pythunits}
Let $a\geq 6$. If $\gamma\succeq \varepsilon^2$ where $\varepsilon\in\Z[\rho]$ is a unit, then $\varepsilon^2=1$.
\end{lemma}

\begin{proof}
Since $\gamma,\gamma',\gamma''<a^2$ for $a\geq 6$, Lemma \ref{lem:sizeunits} implies that the only possible candidates on $\varepsilon^2$ are $1$ and $\rho^2$. However, it can be checked that 
\[
\gamma'' <a^2-2a+13-\frac{a^5+5a^4+11a^3+a^2-4a-5}{a^4(a+1)}<\left(a-\frac{1}{a^2}\right)^2<\rho''^2.
\]
Thus, we are left only with $\varepsilon^2=1$. 
\end{proof}

In the following lemmas, we will find all $\mathfrak{s}$-indecomposables whose squares are totally smaller than $\gamma$. For that, we will use sets of representatives of $\mathfrak{s}$-indecomposables stated in Theorem \ref{thm:m1}. We will start with those which are associated with totally positive indecomposables from Proposition \ref{prop:indeennola}.

\begin{lemma} \label{lem:pythtotallypos}
Let $a\geq 6$. If $\gamma\succeq (\varepsilon\alpha)^2$ for some unit $\varepsilon\in\Z[\rho]$ and non-unit totally positive indecomposable integers $\alpha$, then 
\[
(\varepsilon\alpha)^2=a^2-3a+5-(a^2-5a-1)\rho-(a-5)\rho^2.
\]
\end{lemma}

\begin{proof}
Let $\kappa_s=1+s\rho+\rho^2$ where $1\leq s\leq a-1$. Using estimates on conjugates of $\rho$, we see that
\begin{align*}
(1+s\rho+\rho^2)^2&>(2+s)^2,\\
(1+s\rho'+\rho'^2)^2&>\left(1-\frac{s}{a}\right)^2,\\
(1+s\rho''+\rho''^2)^2&>\left(1-sa+a^2+\frac{s}{a^2+a}-\frac{2}{a}+\frac{1}{a^4}\right)^2>a^2>\gamma''.
\end{align*}
Thus, we can immediately exclude $\varepsilon^2=1$.
Let now $\varepsilon^2=\rho^k(\rho-1)^l$ for some $k,l\in 2\Z$. By Lemma~\ref{lem:sizeunits} we know that except for two cases, one of conjugates of $\varepsilon^2$ is greater than $a^2$. If $\varepsilon^2>a^2$, then $(\kappa_s\varepsilon)^2>\gamma$, and if $\varepsilon''^2>a^2$, then $(\kappa_s''\varepsilon'')^2>\gamma''$. Thus, the only possible case is when $\varepsilon'^2>a^2$. By Lemma~\ref{lem:sizesofunits2}, if $\varepsilon^2\neq \rho^{-2},\rho^{-2}(\rho-1)^2$, then $\varepsilon'^2>a^4$, or $\varepsilon^2>a^2$, or $\varepsilon''^2>a^2$. However, the last two cases can be excluded by the previous part. Regarding the fist one, we can see that
\[
\varepsilon'^2(1+s\rho'+\rho'^2)^2>a^4\left(1-\frac{a-1}{a}\right)^2=a^2>\gamma',
\]
which is again impossible.    

Therefore, we are left with the cases $\varepsilon^2= \rho^2, \rho^{-2},\rho^{-2}(\rho-1)^2$. However, for $\varepsilon^2=\rho^2, \rho^{-2}(\rho-1)^2$, we have $\varepsilon''^2>1$ which implies $\varepsilon''^2(1+s\rho''+\rho''^2)^2>a^2>\gamma''$.

It remains to discuss the case when $\varepsilon^2=\rho^{-2}$. First of all, let us assume $s\geq 4$. In that case,
\[
\rho^{-2}(1+s\rho+\rho^2)^2>\frac{1}{\left(1+\frac{1}{a+2}\right)^2}(2+s)^2\geq \frac{64}{81} 36>28>21>\gamma.
\]
If $s=3$ and $a\geq 9$, we similarly get
\[
\rho^{-2}(1+3\rho+\rho^2)^2>\left(\frac{11}{12}\right)^2 25>21>\gamma.
\]
Using an approximate value of $\rho$, we can easily verify that the same is true for $s=3$ and $6\leq a \leq 8$.  

For $s=1$, we can compute that
\[
N(\gamma-\rho^{-2}(1+\rho+\rho^2)^2)=-198a+773<0
\]  
for $a\geq 6$. However, for $s=2$, we obtain an element totally smaller that $\gamma$, namely
\[
\rho^{-2}(1+2\rho+\rho^2)^2=a^2-3a+5-(a^2-5a-1)\rho-(a-5)\rho^2.\qedhere
\]
\end{proof}

Now we turn our attention to $\mathfrak{s}$-indecomposables with $\mathfrak{s}=(+,-,+)$ and start with the line of elements of the form $-1-u\rho+(u+2)\rho^2$.

\begin{lemma} \label{lem:negusecka}
Let $a\geq 6$. Then $(\varepsilon(-1-u\rho+(u+2)\rho^2))^2\not\preceq \gamma$ for every unit $\varepsilon\in\Z[\rho]$ and $0\leq u\leq a-2$.
\end{lemma}

\begin{proof}
For $\mu_u=-1-u\rho+(u+2)\rho^2$, we can deduce the following lower bounds:
\begin{align*}
(-1-u\rho+(u+2)\rho^2)^2&>\left(-1-u\left(1+\frac{1}{a+2}\right)+(u+2)\right)^2=\left(1-\frac{u}{a+2}\right)^2\\
&\hspace{7cm}\geq \left(1-\frac{a-2}{a+2}\right)^2=\frac{16}{(a+2)^2},\\
(-1-u\rho'+(u+2)\rho'^2)^2&>\left(-1+\frac{u}{a}+\frac{u+2}{a^2}\right)^2\geq \left(-1+\frac{a-2}{a}+\frac{a}{a^2}\right)^2=\frac{1}{a^2},\\
(-1-u\rho''+(u+2)\rho''^2)^2&>(-1+u(a-1)+(u+2)(a-1)^2)^2\geq (-1+2(a-1)^2)^2\\&\hspace{5cm}=4a^4-16a^3+20a^2-8a+1>a^2>\gamma''
\end{align*}
for $a\geq 6$. Thus, we can immediately exclude $\varepsilon^2=1$. If $\varepsilon^2\neq 1,\rho^2$, then $\varepsilon^2$ has one conjugate greater than $a^2$ by Lemma~\ref{lem:sizeunits}. However, if $\varepsilon''^2>a^2$, then $(\varepsilon''\mu_u'')^2>\gamma''$. Moreover, if $\varepsilon^2>a^4$, then
\[
(\varepsilon\mu_u)^2>a^4\frac{16}{(a+2)^2}>21>\gamma
\]
for $a\geq 6$. Similarly, if $\varepsilon'^2>a^4$, then
\[
(\varepsilon'\mu_u')^2>\frac{a^4}{a^2}=a^2>\gamma'.
\]
Thus, we must have $\varepsilon^2,\varepsilon'^2<a^4$ and $\varepsilon''^2<1$. Note that $\varepsilon^2=\rho^2$ does not satisfy $\varepsilon''^2<1$, so we can directly exclude this unit. Let $\varepsilon^2=\rho^k(\rho-1)^l$ for some $k,l\in 2\Z$. If $k\geq 0$ and $l\geq 0$, then clearly $\varepsilon''^2\geq 1$. If $l\leq -4$ and $k\geq 0$, then $\varepsilon^2>a^4$, and if $k\leq -4$ and $l\geq 0$, then $\varepsilon'^2>a^4$. Furthermore, if $l=-2$ and $k\geq 0$, then
\[
(\varepsilon''\mu_u'')^2>\frac{1}{(a+1)^2}(4a^4-16a^3+20a^2-8a+1)>a^2>\gamma'',
\] 
and the same is true if $k=-2$ and $l\geq 0$. Therefore, we must have $k,l<0$. 

In the case of $k\leq -4$ and $k+2<l<0$, we see that
\[
(\varepsilon'\mu_u')^2>a^4\frac{a^{-k-4}}{\left(1+\frac{1}{a}\right)^{-l}}\frac{1}{a^2}\geq a^4 \frac{1}{a^2}=a^2>\gamma'.
\] 
Likewise, if $l\leq -4$ and $l-2\leq k<0$, then
\[
(\varepsilon\mu_u)^2>\frac{(a+2)^4}{(1+\frac{1}{a+2})^6}\frac{16}{(a+2)^2}>21>\gamma
\]
for $a\geq 6$. Thus, we are left with only two units, namely $\rho^{-2}(\rho-1)^{-2}$ and $\rho^{-4}(\rho-1)^{-2}$. For $\rho^{-2}(\rho-1)^{-2}$, we can compute that
\[
\text{Tr}(\rho^{-2}(\rho-1)^{-2}\mu_u^2)=2a^2+10a+21+3u^2+14u\geq 2a^2+10a+21>2a^2-4a+37=\text{Tr}(\gamma) 
\]
for $a\geq 6$. For $\rho^{-4}(\rho-1)^{-2}$, we get
\[
\text{Tr}(\rho^{-4}(\rho-1)^{-2}\mu_u^2)=a^4+2a^3-8a^2+8a+19-(2a^3+4a^2-14a-4)u+(a^2+2a-2)u^2.
\]
Moreover, this trace decreases in $u$ in the interval $[0,a-2]$ which implies 
\[
\text{Tr}(\rho^{-4}(\rho-1)^{-2}\mu_u^2)>\text{Tr}(\rho^{-4}(\rho-1)^{-2}(-1-(a-2)\rho+a\rho^2)^2)=8a^2+3>2a^2-4a+37=\text{Tr}(\gamma)
\] 
for $a\geq 6$. By this, the proof is completed.  
\end{proof}

Now we will use the norm of $\gamma$ to restrict the remaining set of $\mathfrak{s}$-indecomposable integers whose squares can be totally smaller than $\gamma$. It can be derived that
\[
N(\gamma)=20a^4-82a^3+440a^2-778a+1559.
\]
Moreover, recall that
\begin{multline*}
|N(-v-(a(v-1)+w)\rho+(a(v-1)+w+1)\rho^2)|= -w^3 - ((a - 1) v - 2 a + 4) w^2 + (av^2 + (a^2 - 5 a - 1) v \\- (a^2 - 7 a + 3)) w + v^3 - (a^2 + a + 1) v^2 + (4 a^2 - 2) v - 3 a^2 + 3 a - 1.
\end{multline*}

\begin{lemma} \label{lem:normpythtriangle}
Let $a\geq 12$.
Let $\lambda_{v,w}=-v-(a(v-1)+w)\rho+(a(v-1)+w+1)\rho^2$ for some $1\leq v \leq a-1$ and $\max\{1,v-1\}\leq w\leq a-1$, and let $N(\gamma)\geq N(\lambda_{v,w}^2)$. Then $v$ and $w$ satisfy one of the following conditions:
\begin{enumerate}
\item $v=1$ and $1\leq w\leq a-2$,
\item $v=2$ and $w=2,3,a-2$,
\item $w=v-1$ and $v=2,3,a-4,a-3,a-2,a-1$,
\item $1\leq v\leq a-1$ and $w=a-1$,
\item $(v,w)\in\{(a-4,a-2),(a-3,a-3),(a-3,a-2),(a-2,a-2)\}$.
\end{enumerate}
\end{lemma}

\begin{proof}
We will proceed similarly as in the proof of Proposition \ref{prop:smallestnorm}. First of all, let us focus on the subset given by $3\leq v\leq a-5$ and $v\leq w\leq a-2$. As before, we can see that the minimal norm is attained at some border point. We can compute that
\begin{align*}
N(\lambda_{3,3}^2)&=36a^4-114a^3-372a^2+1032a+1849>N(\gamma),\\
N(\lambda_{3,a-2}^2)&=36a^4-312a^3+1000a^2-1404a+729>N(\gamma),\\
N(\lambda_{a-5,a-5}^2)&=64a^4-32a^3-4012a^2+1004a+63001>N(\gamma),\\
N(\lambda_{a-5,a-2}^2)&=25a^4+70a^3-1681a^2-2422a+29929>N(\gamma)
\end{align*}
for $a\geq 12$. Moreover, for $a-4\leq v\leq a-2$ and $v\leq w\leq a-2$, we can compute that \[N(\lambda_{a-4,a-4}^2),N(\lambda_{a-4,a-3}^2) >N(\gamma);\] the remaining elements are listed in the case (5).

Thus, we are left with the following cases:
\begin{itemize}
\item $v=1$, i.e., the case (1), where the case $w=a-1$ is covered by (4). 
\item $v=2$, for which $N(\lambda_{2,4}^2),N(\lambda_{2,a-3}^2)>N(\gamma)$; thus, it is enough to discuss $w=2,3,a-2$ which gives (2),
\item $w=v-1$, for which $N(\lambda_{4,3}^2),N(\lambda_{a-5,a-6}^2)>N(\gamma)$; thus, it is enough to discuss $v=2,3,a-4,a-3,a-2,a-1$ which gives (3),
\item $w=a-1$, i.e., the case (4).\qedhere
\end{itemize}    
\end{proof}

However, most of the integers from Lemma \ref{lem:normpythtriangle} do not produce any suitable elements.

\begin{lemma}
Let $a\geq 12$. There does not exist a unit $\varepsilon\in\Z[\rho]$ such that $(\varepsilon\lambda_{v,w})^2\preceq \gamma$ if $v$ and $w$ are as in (2), (3) and (5) of Lemma $\ref{lem:normpythtriangle}$. 
\end{lemma}

\begin{proof}
Let us start with $\lambda_{2,2}$. We can deduce that
\begin{align*}
(-2-(a+2)\rho+(a+3)\rho^2)^2&>\frac{1}{(a+3)^2},\\
(-2-(a+2)\rho'+(a+3)\rho'^2)^2&>\frac{(a^2-3a-3)^2}{a^4},\\
(-2-(a+2)\rho''+(a+3)\rho''^2)^2&>a^6+4a^5-4a^4-18a^3+12a^2+8a+1>a^2>\gamma''
\end{align*}
for $a\geq 12$. Thus, we can immediately exclude $\varepsilon^2=1,\rho^2$ as $(\varepsilon''\lambda_{2,2}'')^2>a^2>\gamma''$. Since $a^4\lambda_{2,2}^2>\gamma$ and $a^4\lambda_{2,2}'^2>\gamma'$, we must have $\varepsilon^2,\varepsilon'^2<a^4$ and $\varepsilon''^2<1$. We can proceed similarly as in the proof of Lemma~\ref{lem:negusecka} to conclude that $\varepsilon^2=\rho^{-2}(\rho-1)^{-2},\rho^{-4}(\rho-1)^{-2}$. For these units, we see that
\begin{align*}
\text{Tr}(\rho^{-2}(\rho-1)^{-2}\lambda_{2,2}^2)&=2a^2+2a+14>\text{Tr}(\gamma),\\
\text{Tr}(\rho^{-4}(\rho-1)^{-2}\lambda_{2,2}^2)&=a^4-2a^3-13a^2+30a+9>\text{Tr}(\gamma)
\end{align*} 
for $a\geq 12$.

Let us proceed with $\lambda_{2,3}$. Here we have
\begin{align*}
(-2-(a+3)\rho+(a+4)\rho^2)^2&>\frac{(2a^2+10a+11)^2}{(a+2)^2(a+3)^4},\\
(-2-(a+3)\rho'+(a+4)\rho'^2)^2&>\frac{(a^2-4a-4)^2}{a^4},\\
(-2-(a+3)\rho''+(a+4)\rho''^2)^2&>a^6+6a^5-a^4-32a^3+19a^2+10a+1>a^2>\gamma''
\end{align*}
for $a\geq 12$. Similarly as in the previous case, we are left with the units $\varepsilon^2=\rho^{-2}(\rho-1)^{-2},\rho^{-4}(\rho-1)^{-2}$ which leads to
\begin{align*}
\text{Tr}(\rho^{-2}(\rho-1)^{-2}\lambda_{2,3}^2)&=2a^2+2a+33>\text{Tr}(\gamma),\\
\text{Tr}(\rho^{-4}(\rho-1)^{-2}\lambda_{2,3}^2)&=a^4-4a^3-12a^2+52a+1>\text{Tr}(\gamma)
\end{align*} 
for $a\geq 12$. 

For $v=2$, it remains to discuss the element with $w=a-2$. For it, we see that
\begin{align*}
(-2-(2a-2)\rho+(2a-1)\rho^2)^2&>\frac{(a^3+2a^2-10a-14)^2}{(a+2)^2(a+3)^4},\\
(-2-(2a-2)\rho'+(2a-1)\rho'^2)^2&>\frac{1}{a^4},\\
(-2-(2a-2)\rho''+(2a-1)\rho''^2)^2&>4a^6-12a^5+9a^4-4a^3+6a^2+1>a^2>\gamma''.
\end{align*}
In this case, we must have $\varepsilon^2<a^2$, $a^2<\varepsilon'^2<a^6$ and $\varepsilon''^2<1$. This (together with the lower bound on $\lambda_{2,a-2}^2$) leads to the units $\rho^{-2}$, $\rho^{-4}$ and $\rho^{-4}(\rho-1)^2$. For them, we compute that
\begin{align*}
\text{Tr}(\rho^{-2}\lambda_{2,a-2}^2)&=4a^4+4a^3-7a^2-20a+13>\text{Tr}(\gamma),\\
\text{Tr}(\rho^{-4}\lambda_{2,a-2}^2)&=8a^2+8a-29>\text{Tr}(\gamma),\\
\text{Tr}(\rho^{-4}(\rho-1)^{2}\lambda_{2,a-2}^2)&=4a^4+12a^3+9a^2-18a-50>\text{Tr}(\gamma)
\end{align*}
for $a\geq 12$. Thus, we have completed the proof for (2).

Now, we will turn our attention to (3), i.e., $w=v-1$ and $v=2,3,a-4,a-3,a-2,a-1$. Here, we have to use more precise estimates on $\rho$ and $\rho'$, namely $1+\frac{1}{a+2+\frac{1}{a^2}}<\rho$ and $-\frac{1}{a+1-\frac{2}{a}+\frac{2}{a^2}}<\rho'$ for $a\geq 7$ (which can be checked using the minimal polynomial of $\rho$). Using them, we can conclude that
\begin{align*}
\lambda_{v-1,v}^2&>\frac{(a^6+7a^5+10a^4+10a^3+18a^2+a+6)^2}{(a+2)^2(a^3+2a^2+1)^4},\\
\lambda_{v-1,v}'^2&>\frac{(a^5+2a^4+6a^3-12a^2+12a-4)^2}{(a^3+a^2-2a+2)^4},\\
\lambda_{v-1,v}''^2&>a^6+2a^5-5a^4-8a^3+7a^2+6a+1.
\end{align*}
Again, we must have $\varepsilon^2,\varepsilon'^2<a^4$ and $\varepsilon''^2<1$ which using the same procedure as in the proof of Lemma \ref{lem:negusecka}, gives $\varepsilon^2=\rho^{-2}(\rho-1)^{-2},\rho^{-4}(\rho-1)^{-2}$. For these units, we have
\begin{align*}
\text{Tr}(\rho^{-2}(\rho-1)^{-2}\lambda_{v-1,v}^2)&=(a^2+2a+2)v^2-(2a^2+4a+6)v+2a^2+2a+5\\
& \hspace{7cm}\geq 2a^2+2a+1>\text{Tr}(\gamma),\\
\text{Tr}(\rho^{-4}(\rho-1)^{-2}\lambda_{v-1,v}^2)&=(a^2+4a+1)v^2-(2a^3+8a^2-4)v+a^4+4a^3-4a+1\\
& \hspace{7cm}\geq 2a^2+2a-1>\text{Tr}(\gamma).
\end{align*} 

We will proceed with with (5), i.e., with $(v,w)\in\{(a-4,a-2),(a-3,a-3),(a-3,a-2),(a-2,a-2)\}$. In this case, similarly as before, we can deduce that
\begin{align*}
\lambda_{v,w}^2&>\frac{(2a^6+13a^5+20a^4+18a^3+26a^2+3a+8)^2}{(a+2)^2(a^3+2a^2+1)^4},\\
\lambda_{v,w}'^2&>\frac{(a^5+3a^4+8a^3-20a^2+20a-8)^2}{(a^3+a^2-2a+2)^4},\\
\lambda_{v,w}''^2&>a^8-10a^7+31a^6-32a^5+29a^4-56a^3+31a^2-10a+25.
\end{align*}  
Therefore, we must have $\varepsilon^2,\varepsilon'^2<a^4$, and $\varepsilon''^2<1$. This again leads to the unit $\varepsilon^2=\rho^{-2}(\rho-1)^{-2},\rho^{-4}(\rho-1)^{-2}$. However, $\rho''^{-2}(\rho''-1)^{-2}\lambda_{v,w}''^2>a^2$, and
\begin{align*}
\text{Tr}(\rho^{-4}(\rho-1)^{-2}\lambda_{v,w}^2)&\geq 2a^2+2a+9>\text{Tr}(\gamma)
\end{align*} 
for $a\geq 12$. By this, the proof is completed.   
\end{proof}

Now, we will discuss the remaining elements found in Lemma \ref{lem:normpythtriangle}.

\begin{lemma} \label{lem:pythtriangle}
Let $a\geq 12$ and let $v$ and $w$ be as in (1) or (4) of Lemma $\ref{lem:normpythtriangle}$. Moreover, let $(\varepsilon\lambda_{v,w})^2\preceq \gamma$ for some unit $\varepsilon\in\Z[\rho]$. Then $(\varepsilon\lambda_{v,w})^2$ is one of the following elements:   
\begin{itemize}
\item $\rho^{-2}(\rho-1)^{-2}(-1-\rho+2\rho^2)^2=a^2-3a+3-(a^2-5a+3)\rho-(a-4)\rho^2$,
\item $\rho^{-2}(\rho-1)^{-2}(-1-2\rho+3\rho^2)^2=a^2-5a+8-(a^2-7a+5)\rho-(a-6)\rho^2$,
\item $\rho^{-4}(\rho-1)^{-2}(-(a-1)-(a(a-2)+a-1)\rho+(a(a-2)+a)\rho^2)^2=1+2\rho+\rho^2$,
\item $\rho^{-4}(\rho-1)^{-2}(-(a-2)-(a(a-3)+a-1)\rho+(a(a-3)+a)\rho^2)^2=4+4\rho+\rho^2$,
\item $\rho^{-4}(\rho-1)^{-2}(-(a-3)-(a(a-4)+a-1)\rho+(a(a-4)+a)\rho^2)^2=9+6\rho+\rho^2$. 
\end{itemize} 
\end{lemma}

\begin{proof}
We will start with (1), i.e., with elements of the form $\lambda_{v,w}=-1-w\rho+(w+1)\rho^2$ where $1\leq w\leq a-2$, for which we have
\begin{align*}
(-1-w\rho+(w+1)\rho^2)^2&>\frac{(3a^2+16a+19)^2}{(a+2)^2(a+3)^4},\\
(-1-w\rho'+(w+1)\rho'^2)^2&>\frac{(a+1)^2}{a^4},\\
(-1-w\rho''+(w+1)\rho''^2)^2&>4a^4-12a^3+9a^2.
\end{align*}
As in some of the previous cases, our unit has to satisfy $\varepsilon^2,\varepsilon'^2<a^4$ and $\varepsilon''^2<1$ which gives the units $\rho^{-2}(\rho-1)^{-2}$ and $\rho^{-4}(\rho-1)^{-2}$. Moreover, if $w\geq 4$, then $a^2\lambda_{v,w}^2>21$ for $a\geq 12$, and $\varepsilon^2<a^2$ is not fulfilled by any of these two units. Thus, we are left with $w=1,2,3$ and $\varepsilon^2=\rho^{-2}(\rho-1)^{-2},\rho^{-4}(\rho-1)^{-2}$.

For $\rho^{-4}(\rho-1)^{-2}$, we obtain
\[
\text{Tr}(\rho^{-4}(\rho-1)^{-2}(-1-w\rho+(w+1)\rho^2)^2)\geq a^4-4a^3-10a^2+48a-6>\text{Tr}(\gamma).
\]
For $w=3$ and $\rho^{-2}(\rho-1)^{-2}$, we get
\[
N(\gamma-\rho^{-2}(\rho-1)^{-2}(-1-3\rho+4\rho^2)^2)=-20a^3+72a^2-10a-79<0
\] 
for $a\geq 12$. On the other hand, for $w=1,2$, we indeed obtain elements totally smaller than $\gamma$, namely
\begin{align*}
\rho^{-2}(\rho-1)^{-2}(-1-\rho+2\rho^2)^2&=a^2-3a+3-(a^2-5a+3)\rho-(a-4)\rho^2,\\
\rho^{-2}(\rho-1)^{-2}(-1-2\rho+3\rho^2)^2&=a^2-5a+8-(a^2-7a+5)\rho-(a-6)\rho^2.
\end{align*}

We will proceed with the elements from (4), i.e., with $\lambda_{v,a-1}=-v-(a(v-1)+a-1)\rho+(a(v-1)+a)\rho^2$ where $1\leq v \leq a-1$. For them, we get the following lower bounds:
\begin{align*}
\lambda_{v,a-1}^2&>\frac{(2a^6+12a^5+16a^4+10a^3+18a^2+a+5)^2}{(a+2)^2(a^3+2a^2+1)^4},\\
\lambda_{v,a-1}'^2&>\frac{4(4a^3-7a^2+6a-2)^2}{(a^3+a^2-2a+2)^4},\\
\lambda_{v,a-1}''^2&>a^6-2a^5-a^4+2a^3+a^2.
\end{align*}
It follows that our unit $\varepsilon$ has to satisfy $\varepsilon^2<a^4$, $\varepsilon'^2<a^8$ and $\varepsilon''^2<1$. Using a similar procedure as before, we are left with the units $\rho^{-6}(\rho-1)^2$, $\rho^{-6}$, $\rho^{-6}(\rho-1)^{-2}$, $\rho^{-4}$, $\rho^{-4}(\rho-1)^{-2}$ and $\rho^{-2}(\rho-1)^{-2}$. Moreover, if $1\leq v\leq a-4$, then the above lower bounds are larger, and using these better bounds or, as before, $\text{Tr}(\varepsilon^2\lambda_{v,a-1}^2)$ to exclude all of the above units.
Thus, it remains to discuss the cases when $v=a-3, a-2, a-1$.

For $\rho^{-6}(\rho-1)^2$, $\rho^{-6}$, $\rho^{-4}$ and $\rho^{-2}(\rho-1)^{-2}$, we can deduce the following:
\begin{align*}
\text{Tr}(\rho^{-6}(\rho-1)^2\lambda_{v,a-1}^2)&\geq \text{Tr}(\rho^{-6}(\rho-1)^2\lambda_{a-3, a-1}^2)=a^4-2a^3+48a+75>\text{Tr}(\gamma),\\
\text{Tr}(\rho^{-6}\lambda_{v,a-1}^2)&\geq \text{Tr}(\rho^{-6}\lambda_{a-1, a-1}^2)=2a^2+2a-7>\text{Tr}(\gamma),\\
\text{Tr}(\rho^{-4}\lambda_{v,a-1}^2)&\geq \text{Tr}(\rho^{-4}\lambda_{a-3, a-1}^2)=a^4-4a^3-2a^2+8a+34>\text{Tr}(\gamma),\\ 
\text{Tr}(\rho^{-2}(\rho-1)^{-2}\lambda_{v,a-1}^2)&\geq \text{Tr}(\rho^{-2}(\rho-1)^{-2}\lambda_{a-3, a-1}^2)=a^4-6a^3+9a^2-4a+38>\text{Tr}(\gamma)
\end{align*}
for $a\geq 12$. Let us now focus on $\rho^{-6}(\rho-1)^{-2}$. For this unit, we obtain 
\begin{align*}
\text{Tr}(\rho^{-6}(\rho-1)^{-2}\lambda_{a-3, a-1}^2)&=9a^2+12a-15>\text{Tr}(\gamma),\\ 
\text{Tr}(\rho^{-6}(\rho-1)^{-2}\lambda_{a-2, a-1}^2)&=4a^2+4a-5>\text{Tr}(\gamma),\\
N(\gamma-\rho^{-6}(\rho-1)^{-2}\lambda_{a-1, a-1}^2)&=-32a^3+204a^2-580a+1523<0
\end{align*}  
for $a\geq 12$. Thus, we can exclude this unit. On the other hand, considering the unit $\rho^{-4}(\rho-1)^{-2}$, we get three squares totally smaller than $\gamma$, namely
\begin{align*}
\rho^{-4}(\rho-1)^{-2}\lambda_{a-3, a-1}^2&=9+6\rho+\rho^2,\\
\rho^{-4}(\rho-1)^{-2}\lambda_{a-2, a-1}^2&=4+4\rho+\rho^2,\\
\rho^{-4}(\rho-1)^{-2}\lambda_{a-1, a-1}^2&=1+2\rho+\rho^2.\qedhere
\end{align*}
\end{proof}

We will now focus on squares of $\mathfrak{s}$-decomposables totally smaller than $\gamma$. We will construct them using $\mathfrak{s}$-indecomposables which we derived in previous lemmas and whose signatures can be found in Table \ref{tab:signat}. Recall that the signature of both $\rho$ and $\rho-1$ is $(+,-,-)$.

More concretely, we will use the following fact. If $\omega\in\Z[\rho]$ is a $\mathfrak{s}$-decomposable integer, it can be written as $\omega=\sum_{i=1}^n \beta_i$ where elements $\beta_i\in\Z[\rho]$ are $\mathfrak{s}$-indecomposables. Then, if $\gamma\succeq\omega^2$, we have
\[
\gamma\succeq \omega^2=\Big(\sum_{i=1}^n \beta_i\Big)^2\succeq (\beta_i+\beta_j)^2
\]
for all $1\leq i,j\leq n$. Therefore, if we fix the signature $\mathfrak{s}$ and show that $(\beta_i+\beta_j)^2\succ \gamma$ for all pair $\beta_i,\beta_j$ such that $\beta_i$ and $\beta_j$ have the signature $\mathfrak{s}$, and their squares were found in Lemmas \ref{lem:pythunits}, \ref{lem:pythtotallypos}, \ref{lem:negusecka} and \ref{lem:pythtriangle}, then there does not exist a $\mathfrak{s}$-decomposable integer $\omega\in\Z[\rho]$ such that $\gamma\succeq\omega^2$. As we will see, this happens for almost all signatures $\mathfrak{s}$.

\begin{table}[ht]
\begin{tabular}{|c|c|c|}
\hline
$\mathfrak{s}$-indecomposables $\beta$ & Signature of $\beta$ & Signature of $-\beta$\\
\hline
$1$ & $(+,+,+)$ & $(-,-,-)$\\
\hline
$\rho^{-1}(1+2\rho+\rho^2)$ & $(+,-,-)$ & $(-,+,+)$\\
\hline
$\rho^{-1}(\rho-1)^{-1}\lambda_{1,1}$ & $(+,-,+)$ & $(-,+,-)$\\
\hline
$\rho^{-1}(\rho-1)^{-1}\lambda_{1,2}$ & $(+,-,+)$ & $(-,+,-)$\\
\hline
$\rho^{-2}(\rho-1)^{-1}\lambda_{a-1,a-1}$ & $(+,+,-)$ & $(-,-,+)$\\
\hline
$\rho^{-2}(\rho-1)^{-1}\lambda_{a-2,a-1}$ & $(+,+,-)$ & $(-,-,+)$\\
\hline
$\rho^{-2}(\rho-1)^{-1}\lambda_{a-3,a-1}$ & $(+,+,-)$ & $(-,-,+)$\\
\hline
\end{tabular}
\caption{Signatures of $\mathfrak{s}$-indecomposables whose squares are totally smaller than $\gamma$.} \label{tab:signat}
\end{table}

\begin{lemma}
Let $a\geq 12$ and $\gamma\succeq \omega^2$ for some $\mathfrak{s}$-decomposable integer $\omega\in\Z[\rho]$. Then $\omega^2\in\{4,9,16\}$.
\end{lemma}

\begin{proof}
As indicated before, we will discuss each of the possible signatures separately:
\begin{enumerate}
\item Signature $(+,+,+)$ (or $(-,-,-)$): In this case, we have only one $\mathfrak{s}$-indecomposable integer $1$, from which we can get additional elements $4$, $9$, and $16$.
\item Signature $(+,-,-)$ (or $(-,+,+)$): In this part with one $\mathfrak{s}$-indecomposable integer $\rho^{-1}(1+2\rho+\rho^2)$, we can deduce that
\begin{align*}
(2\rho^{-1}(1+2\rho+\rho^2))^2&=4(a^2-3a+5-(a^2-5a-1)\rho-(a-5)\rho^2)\\&\hspace{4cm}>4\left(11+\frac{5a^2+20a+27}{(a+2)^2}\right)>21>\gamma.
\end{align*} 
Thus, we do not get any suitable $\mathfrak{s}$-decomposables.
\item Signature $(+,-,+)$ (or $(-,+,-)$): In this signature, we work with two $\mathfrak{s}$-indecomposables $\rho^{-1}(\rho-1)^{-1}\lambda_{1,1}$ and $\rho^{-1}(\rho-1)^{-1}\lambda_{1,2}$. For them, we can conclude that 
\begin{enumerate}
\item $(2\rho^{-1}(\rho-1)^{-1}\lambda_{1,1})^2>21+\frac{3(5a^2+12a+12)}{(a+2)^2}>\gamma$,
\item $(\rho^{-1}(\rho-1)^{-1}(\lambda_{1,1}+\lambda_{1,2}))^2>25+\frac{4(6a^2+17a+17)}{(a+2)^2}>\gamma$,
\item $(2\rho^{-1}(\rho-1)^{-1}\lambda_{1,2})^2>36+\frac{4(7a^2+20a+20)}{(a+2)^2}>\gamma$. 
\end{enumerate} 
\item Signature $(+,+,-)$ (or $(-,-,+)$): For three $\mathfrak{s}$-indecomposables in this signature, we can deduce the following results:
\begin{enumerate}
\item $(2\rho^{-2}(\rho-1)^{-1}\lambda_{a-1,a-1})''^2>4a^2-16a+8>a^2>\gamma''$,
\item $(\rho^{-2}(\rho-1)^{-1}(\lambda_{a-1,a-1}+\lambda_{a-2,a-1}))^2>25+\frac{20}{a+3}+\frac{4}{(a+3)^2}>\gamma$,
\item $(\rho^{-2}(\rho-1)^{-1}(\lambda_{a-1,a-1}+\lambda_{a-3,a-1}))^2>36+\frac{24}{a+3}+\frac{4}{(a+3)^2}>\gamma$, 
\item $(2\rho^{-2}(\rho-1)^{-1}\lambda_{a-2,a-1})^2>36+\frac{24}{a+3}+\frac{4}{(a+3)^2}>\gamma$, 
\item $(\rho^{-2}(\rho-1)^{-1}(\lambda_{a-2,a-1}+\lambda_{a-3,a-1}))^2>49+\frac{28}{a+3}+\frac{4}{(a+3)^2}>\gamma$,
\item $(2\rho^{-2}(\rho-1)^{-1}\lambda_{a-3,a-1})^2>64+\frac{32}{a+3}+\frac{4}{(a+3)^2}>\gamma$ 
\end{enumerate}
for $a\geq 12$.    
\end{enumerate}
Therefore, in this lemma, we obtain only the squares $4$, $9$ and $16$.
\end{proof}

In previous lemmas, we have found all squares totally smaller than $\gamma$ for $a\geq 12$. Since every decomposition of $\gamma$ can consists only of these squares, we are now able to prove that we need at least six squares to express $\gamma$. 

\begin{proposition}
Let $\rho$ be a root of the polynomial $x^3+(a-1)x^2-ax-1$ where $a\geq 4$. Then $\P(\Z[\rho])=6$.
\end{proposition}

\begin{proof}
In the previous parts, we have obtained all squares which are totally smaller than our given element $\gamma$. In particular, for $a\geq 12$, we get the squares
\begin{enumerate}
\item rational integers $1$, $4$, $9$ and $16$,
\item $a^2-3a+5-(a^2-5a-1)\rho-(a-5)\rho^2$,
\item $a^2-3a+3-(a^2-5a+3)\rho-(a-4)\rho^2$,
\item $a^2-5a+8-(a^2-7a+5)\rho-(a-6)\rho^2$,
\item $1+2\rho+\rho^2$,
\item $4+4\rho+\rho^2$,
\item $9+6\rho+\rho^2$.
\end{enumerate}
Moreover, as described in \cite{Ti1}, we can verify by a computer program (we used Mathematica) that the same is true also for $6\leq a\leq 11$. The program which we used is available at
\url{https://sites.google.com/view/tinkovamagdalena/codes}.

Recall that $\gamma=a^2-3a+11-(a^2-5a+1)\rho-(a-5)\rho^2$. The coefficient $-(a^2-5a+1)$ before $\rho$ is clearly negative. Thus, every decomposition of $\gamma$ to the sum of squares must contain at least one of the elements $a^2-3a+5-(a^2-5a-1)\rho-(a-5)\rho^2$, $a^2-3a+3-(a^2-5a+3)\rho-(a-4)\rho^2$ and $a^2-5a+8-(a^2-7a+5)\rho-(a-6)\rho^2$. However, the coefficient before $\rho$ is greater than $-(a^2-5a+1)$ for $a^2-3a+5-(a^2-5a-1)\rho-(a-5)\rho^2$ and $a^2-5a+8-(a^2-7a+5)\rho-(a-6)\rho^2$. Therefore, if one of these two elements were a summand in a decomposition of $\gamma$, then some other summand would have to be one of these three elements. In that case, the coefficient before $1$ (in the base $1$, $\rho$ and $\rho^2$) would be at least $2a^2-10a+16>a^2-3a+11$ for $a\geq 7$. This is impossible as all squares totally smaller than $\gamma$ have the first coefficient positive. Note that almost the same is valid also for $a=6$ since $-(a^2-7a+5)=1$ and $-(a^2-5a+1)=-7$, and, thus, the first coefficient would be at least $a^2-5a+8+a^2-3a+3=35>29=a^2-3a+11$. 

Therefore, we must have
\[
\gamma=(a^2-3a+3-(a^2-5a+3)\rho-(a-4)\rho^2)+(8+2\rho+\rho^2).
\]
In every decomposition of $8+2\rho+\rho^2$ to the sum of squares, we can use only the elements $1$, $4$, $1+2\rho+\rho^2$ and $4+4\rho+\rho^2$. Nevertheless, $4+4\rho+\rho^2$ cannot appear in this decomposition since it has a too large coefficient before $\rho$. It follows that 
\[
8+2\rho+\rho^2=(1+2\rho+\rho^2)+7.
\] 
Moreover, we see that every decomposition of $7$ can consist only of the elements $1$ and $4$, and for that, it is necessary to have at least $4$ squares. Therefore, every decomposition of $\gamma$ to the sum of squares requires at least $6$ squares which implies $\P(\Z[\rho])=6$.   

For $a=5$, the element $\gamma$ is equal to $21-\rho$. By a computer program, we can derive that the squares totally smaller than $\gamma$ are $1$, $4$, $9$, $16$, $1+2\rho+\rho^2$, $4+4\rho+\rho^2$, $9+6\rho+\rho^2$, $8+5\rho+\rho^2$, $13-3\rho-\rho^2$, $15+\rho$ and $\rho^2$. However, easily, we can conclude that every decomposition of $\gamma$ must be of the form
\[
21-\rho=(13-3\rho-\rho^2)+(1+2\rho+\rho^2)+7,
\]
which implies that we need at least $6$ squares.

Similarly, for $a=4$ and $\gamma=15+3\rho+\rho^2$, we obtain the squares $1$, $4$, $9$, $1+2\rho+\rho^2$, $4+4\rho+\rho^2$, $9+6\rho+\rho^2$, $7+\rho$, $\rho^2$, $4+7\rho+2\rho^2$, $12-5\rho-2\rho^2$ and $9+5\rho+\rho^2$. This leads to the decomposition
\[
15+3\rho+\rho^2=(7+\rho)+(1+2\rho+\rho^2)+7.
\]
Thus, even in this case, it is necessary to have at least $6$ squares.
\end{proof}

For $a=3$, the situation is different, and the element $\gamma$ can be written as the sum of $4$ squares. In this case, we can provide the lower bound $5$ on $\P(\Z[\rho])$ which is, for example, attained by the element $9-\rho$.

\section{Preliminaries on the Jacobi--Perron algorithm} \label{sec:prelijpa}

In this section, we will summarize several basic facts about the Jacobi--Perron algorithm. In the following, we mostly use the notation from \cite{Vo}. We will be concerned with both its inhomogeneous and homogeneous form, for which we also show their mutual relation.

\subsection{Inhomogeneous Jacobi--Perron Algorithm}

As described in the introduction, the inhomogeneous Jacobi--Perron algorithm (iJPA) is a generalization of the classical continued fraction algorithm to higher dimensions.
To define it precisely, let us consider a vector $ \bm{\alpha} ^{(0)}=\big(\alpha_1^{(0)}, \alpha_2^{(0)}, \dots, \alpha _{n-1}^{(0)}\big) \in \mathbb{R}^{n-1} $ for $ n\geq 2 $. For $\bm{\alpha} ^{(0)}$, we define its \textit{inhomogeneous Jacobi--Perron Algorithm expansion} as the sequence $ \left\langle \bm{\alpha} ^{(k)}\right\rangle _{k \geq 0 }$ where    
\[ \bm{\alpha} ^{(k +1)} = \left(\alpha _1^{(k +1)}, \dots, \alpha _{n-2}^{(k +1)}, \alpha_{n-1}^{(k +1)}\right) = \left( \frac{\alpha_2^{(k)}-a_2^{(k)}}{\alpha_1^{(k)}-a_1^{(k)}}, \dots, \frac{\alpha_{n-1}^{(k)}-a_{n-1}^{(k)}}{\alpha_1^{(k)}-a_1^{(k)}}, \frac{1}{\alpha_1^{(k)}-a_1^{(k)}} \right) \hspace{-0.15cm} , \]
where $\alpha_1^{(k)}\neq a_1^{(k)}$ and $a_i^{(k)}=\big\lfloor \alpha_i^{(k)} \big\rfloor$. We will also use the notation $\bm{a}^{(k)}=\big(a_1^{(k)},a_2^{(k)},\ldots,a_{n-1}^{(k)}\big)$. We can easily see that for $n=2$, this expansion is exactly the continued fraction expansion for real number $\alpha_1^{(0)}=\alpha \in \mathbb{R}$. 

The iJPA expansion of $\bm{\alpha} ^{(0)}$ is called \textit{periodic} if there exist two integers $l_0, l_1\in\Z$ such that $l_0\geq 0$, $l_1 \geq 1$ and 
\[\bm{\alpha} ^{(k + l_1)}=\bm{\alpha} ^{(k)} \]
for every $k \geq l_0$. If $l_0$ and $l_1$ are the smallest integers fulfilling  these conditions, then the vectors $\bm{\alpha}^{(0)}, \bm{\alpha}^{(1)}, \dots, \bm{\alpha}^{(l_0-1)}$ and $\bm{\alpha}^{(l_0)}, \bm{\alpha}^{(l_0+1)}, \dots, \bm{\alpha}^{(l_0+l_1-1)}$ are called the \textit{preperiod} and the \textit{period} of the iJPA expansion of $\bm{\alpha}^{(0)}$, and $l_0$ and $l_1$ are their \textit{lengths}. If $l_0=0$, then the iJPA expansion of $\bm{\alpha} ^{(0)}$ is called \textit{purely periodic}.

Moreover, for the periodic iJPA expansion of some vector of algebraic integers $\bm{\alpha}^{(0)}$, we can study the element 
\[ \varepsilon = \displaystyle\prod_{k=l_0}^{l_0+l_1-1} \alpha_{n-1}^{(k)}. \]
It was proved in \cite{BHunit} that $\varepsilon$ is a unit in the ring of algebraic integers of $\mathbb{Q}\big( \alpha_1^{(0)},\alpha_2^{(0)}, \dots, \alpha_{n-1}^{(0)} \big)$. This unit is called the Hasse--Bernstein unit.

\subsection{The homogenous Jacobi--Perron Algorithm}
For us, it is more convenient to use the homogenous Jacobi--Perron Algorithm (JPA) which is defined in the following way:
Let us consider $ \bm{\beta} ^{(0)}=\left(\beta_1^{(0)}, \beta_2^{(0)}, \dots, \beta _{n}^{(0)}\right) \in \mathbb{R}^{n} $ for $ n\geq 2 $. The \textit{Jacobi--Perron Algorithm expansion} of $ \bm{\beta}  ^{(0)}$ is the sequence of vectors $ \langle \bm{\beta} ^{(k)}\rangle _{k \geq 0 }$ defined by
\[ \bm{\beta}^{(k+1)}=\left(\beta_2^{(k)}-b_2^{(k)} \beta_1^{(k)}, \beta_3^{(k)}-b_3^{(k)} \beta_1^{(k)}, \dots , \beta_n^{(k)}-b_n^{(k)} \beta_1^{(k)}, \beta_1^{(k)} \right),
\]
where $b_i^{(k)}= \Big \lfloor \frac{\beta_i^{(k)}}{\beta_1^{(k)}} \Big \rfloor$. We will also use the notation $\bm{b}^{(k)}=\big(b_{2}^{(k)},\ldots,b_{n}^{(k)}\big)$. 
For this expansion, in comparison to the iJPA expansion, if the starting elements are algebraic integers, then the elements in the expansion are always algebraic integers. We call the elements $\beta^{(k)}_i$ where $i,k \in \mathbb{N}$ the \textit{convergents of the expansion} or simply \textit{convergents}.

Moreover, in the article, we define the \textit{semiconvergents} of the JPA expansion of $\bm{\beta}=\bm{\beta} ^{(0)}$ as
\begin{equation} \label{eq:semi}
\delta_{i,j}^{(k)}=\beta_i^{(k)}-j\beta_1^{(k)}
\end{equation}
where $2\leq i \leq n$, $0\leq j\leq \Big\lfloor\frac{\beta_i^{(k)}}{\beta_1^{(k)}}\Big\rfloor-1$ and $k\in\N_{0}$. Note that for $j=0$, we obtain convergents of the JPA expansion of $\bm{\beta} ^{(0)}$.

The following proposition gives us a relationship between iJPA and JPA expansions of similar vectors. In the next sections, we will compute the JPA expansion for some specific vectors. However, for the iJPA expansions, we can easily see if they are periodic. So firstly, we will derive the iJPA expansion, and then we use the following well-known relation to determine the JPA expansion. The proof (as well as the proof of Proposition \ref{lm:iJPA-JPA,per}) can be found in Appendix A5 of the arxiv version of this paper \cite{KST}.

\begin{proposition}\label{lm:iJPA-JPA}
Let $\langle\bm{\alpha} ^{(k)}\rangle$ be the iJPA expansion of the vector $(\theta_1, \theta_2, \dots, \theta_{n-1})\in\R^{n-1}$. Then the JPA expansion of $\bm{\beta}^{(0)}= (1, \theta_1, \theta_2, \dots, \theta_{n-1})$ is equal to
\[ \bm{\beta}^{(k)}=\left(\delta_{k}, \alpha_1^{(k)} \delta_{k}, \dots, \alpha_{n-2}^{(k)} \delta_{k},\alpha_{n-1}^{(k)} \delta_{k}\right),\]
where $\delta_k=\frac{1}{\alpha_{n-1}^{(1)} \cdot \dots \cdot \alpha_{n-1}^{(k)}}$ for $k \geq 0$, and $\delta_{0}=1$. Moreover for every $i,k$, we have
\[  \Bigg\lfloor \frac{\beta_{i+1}^{(k)}}{\beta_1^{(k)}} \Bigg\rfloor = \big\lfloor \alpha_i^{(k)}\big\rfloor.\]
\end{proposition}

The JPA expansion of $\bm{\beta}^{(0)}$ is called \textit{periodic} if there exist two integers $l_0, l_1$ with $l_0\geq 0$ and $l_1 \geq 1$, and a unit $\varepsilon$ such that 
\[\beta_i^{(k + l_1)}=\beta_i^{(k)} \varepsilon\]
for all $1\leq i\leq n$ and $k \geq l_0$. If $l_0$ and $l_1$ are the smallest integers satisfying these conditions, then $\bm{\beta}^{(0)},\bm{\beta}^{(1)}, \dots, \bm{\beta}^{(l_0-1)}$ and $\bm{\beta}^{(l_0)}, \bm{\beta}^{(l_0+1)}, \dots, \bm{\beta}^{(l_0+l_1-1)}$ are called the \textit{preperiod} and the \textit{period} of the JPA expansion of $\bm{\beta}^{(0)}$, and $l_0$ and $l_1$ are their \textit{lengths}. If $l_0=0$, then the JPA expansion of $\bm{\beta}^{(0)}$ is \textit{purely periodic}. Moreover, Proposition \ref{lm:iJPA-JPA} implies the following statement:

\begin{proposition} \label{lm:iJPA-JPA,per}
If the iJPA expansion of $(\theta_1, \theta_2, \dots, \theta_{n-1})\in\R^{n-1}$ is periodic with preperiod length $l_0$ and period length $l_1$, then the JPA expansion of $(1, \theta_1, \theta_2, \dots, \theta_{n-1})$ is also periodic with the same length of preperiod and period, and the unit $\varepsilon$ in the period of iJPA expansion is the inverse of the Hasse--Bernstein unit.
\end{proposition}

Note that from the imhomogeneous JPA expansion $\langle\bm{\alpha}^{(k)}\rangle$ of $(\theta_1, \theta_2, \dots, \theta_{n-1})$, it is almost trivial to compute the JPA expansion $\langle\bm{\beta}^{(k)}\rangle$ of $(1, \theta_1, \theta_2, \dots, \theta_{n-1})$. It is enough to take $\beta_i^{k+1}=\beta_{i+1}^{(k)}-a_{i}^{(k)}\beta_1^{(k)}$ if $i\neq n$ and $\beta_n^{(k+1)}=\beta_1^{(k)}$ where $a_{i}^{(k)}=\lfloor \alpha_{i}^{(k)}\rfloor$, $\beta_{1}^{(0)}=1$ and $\beta_{i}^{(0)}=\alpha_{i-1}^{(0)}$ if $i\neq 1$. Thus, in the following, we mostly omit the computation of the JPA expansion if we know the iJPA expansion of the vector of the above form.

\section{Jacobi--Perron algorithm in the simplest cubic fields} \label{sec:jpasimplest}

In the case of quadratic fields, $\mathfrak{s}$-indecomposables can be obtained using the continued fraction of some concrete numbers. This expansion is periodic for these number and thus produces only finitely many elements up to multiplication by units. Therefore, to study elements determined by the Jacobi--Perron algorithm, it is natural to focus only on periodic JPA expansions. As described in the introduction, there exist numerous results on periodic expansions in many families of (not only cubic) fields. However, for our particular families of fields, we need to find  JPA expansions of some concrete initial vectors. To do that, we firstly discuss inhomogeneous JPA expansions, for which some of calculations are easier. Then we transform these results to the determination of homogeneous JPA expansion using Proposition \ref{lm:iJPA-JPA}. Since most of the computations are easy, we provide only sketches of proofs or, consequently, only statements without proofs. 

\subsection{Inhomogenous Jacobi--Perron expansion in the simplest cubic fields}

We will start with the determination of the inhomogeneous JPA expansion of the vector $(-\rho',\rho'^2)$ where $\rho'$ is the root of the polynomial $x^3-ax^2-(a+3)x-1$ such that $-2<\rho'<-1$. As we will see, this expansion is periodic for all $a\geq -1$. 
For $-1\leq a\leq 3$, this iJPA expansion has a different form, in particular it has a different period length than the expansion of $(-\rho',\rho'^2)$ for $a\geq 4$. We show the precise form of this expansion in Appendix A6 of the arxiv version of this paper \cite{KST}.  

For $a\geq 4$, the iJPA expansion of $(-\rho',\rho'^2)$ has the same length of preperiod and period. Moreover, there are only slight differences in its form for odd and even values of the coefficient $a$ which occurs in the middle of the period.

\begin{proposition} \label{prop:ijpasimplestperiodic4}
Let $a\geq 4$.
Then the Jacobi--Perron expansion of the vector $(-\rho',\rho'^2)$ is periodic. In particular, let $\lfloor\frac{a}{2}\rfloor=A_0$. Then the preperiod of the iJPA expansion of $(-\rho',\rho'^2)$ is 
\[
\begin{array}{ll}
\bm{\alpha}^{(0)}=(-\rho',\rho'^2),&\quad\bm{a}^{(0)}=(1,1),\\
\bm{\alpha}^{(1)}=(1-\rho',-2-(a+1)\rho'+\rho'^2), &\quad\bm{a}^{(1)}=(2,a+1).
\end{array}
\] 
For $a$ even, the period is
\begin{align*}
&\bm{\alpha}^{(2)}=(a+4+a\rho'-\rho'^2,-2-(a+1)\rho'+\rho'^2),\\
&\bm{\alpha}^{(3)}=(-1+\rho'^2,-\rho'),\\
&\bm{\alpha}^{(4)}=\left(\frac{2a+2+(a-1)\rho'-\rho'^2}{2a+3},\frac{-3a-4-(a^2+3a+1)\rho'+(a+2)\rho'^2}{2a+3}\right),\\
&\bm{\alpha}^{(5)}= \left(-\left(A_0+2\right)-\left(A_0+1\right)\rho'+\rho'^2,1-\rho'\right),\\
&\bm{\alpha}^{(6)}=\Bigg(\frac{A_0^2+2A_0-1+(A_0-2)\rho'-\rho'^2}{A_0^3+4A_0^2+3A_0-1},\\&\hspace{3cm}\frac{-2A_0^2-3A_0+1-(2A_0^3+3A_0^2+A_0-1)\rho'+(A_0^2+A_0)\rho'^2}{A_0^3+4A_0^2+3A_0-1}\Bigg),\\
& \bm{\alpha}^{(7)}= (-(A_0+4)-2A_0\rho'+\rho'^2,A_0+2-\rho'),\\
& \bm{\alpha}^{(8)}=(-\rho',-1-(a+1)\rho'+\rho'^2),
\end{align*}
and the corresponding vectors of integers parts are $\bm{a}^{(2)}=(1,a+1)$, $\bm{a}^{(3)}=(0,1)$, $\bm{a}^{(4)}=(0,A_0)$, $\bm{a}^{(5)}=(0,2)$, $\bm{a}^{(6)}=(0,1)$, $\bm{a}^{(7)}=(A_0-2,A_0+3)$ and $\bm{a}^{(8)}=(1,a+2)$.

For $a$ odd, the period is
\begin{align*}
&\bm{\alpha}^{(2)}=(a+4+a\rho'-\rho'^2,-2-(a+1)\rho'+\rho'^2),\\
&\bm{\alpha}^{(3)}=(-1+\rho'^2,-\rho'),\\
&\bm{\alpha}^{(4)}=\left(\frac{2a+2+(a-1)\rho'-\rho'^2}{2a+3},\frac{-3a-4-(a^2+3a+1)\rho'+(a+2)\rho'^2}{2a+3}\right),\\
&\bm{\alpha}^{(5)}=\left(-\left(A_0+2\right)-\left(A_0+2\right)\rho'+\rho'^2,1-\rho'\right),\\
& \bm{\alpha}^{(6)}=\Bigg(\frac{A_0^2+3A_0-2+(A_0-2)\rho'-\rho'^2}{A_0^3+6A_0^2+7A_0-5},\\&\hspace{2cm}
\frac{-2A_0^2-5A_0+3-(2A_0^3+6A_0^2+2A_0-3)\rho'+(A_0^2+2A_0-1)\rho'^2}{A_0^3+6A_0^2+7A_0-5}\Bigg),\\
&\bm{\alpha}^{(7)}=(-(A_0+5)-(2A_0+1)\rho'+\rho'^2,A_0+3-\rho'),\\
& \bm{\alpha}^{(8)}=(-\rho',-1-(a+1)\rho'+\rho'^2),
\end{align*}
and the corresponding vectors of integers parts are $\bm{a}^{(2)}=(1,a+1)$, $\bm{a}^{(3)}=(0,1)$, $\bm{a}^{(4)}=(0,A_0)$, $\bm{a}^{(5)}=(1,2)$, $\bm{a}^{(6)}=(0,1)$, $\bm{a}^{(7)}=(A_0-2,A_0+4)$ and $\bm{a}^{(8)}=(1,a+2)$.
\end{proposition}

\begin{proof}
We will provide only a sketch of the proof since the calculations are similar for each iteration. The most challenging part of this proof is to determine the values of $\bm{a}^{(k)}$, i.e., the integer parts of the included elements. However, the simplest cubic fields have many useful properties which can significantly facilitate our computations.
Recall that we have nice estimates on conjugates of $\rho$, namely for $a\geq 7$, we have
\begin{equation} \label{eq:estimates2}
 -1-\frac{1}{a+1}<\rho'<-1-\frac{1}{a+2}.  
\end{equation} 

In the initial iteration, we have the vector $(-\rho',\rho'^2)$. We know that $\lfloor -\rho'\rfloor=1$, i.e., $a_1^{(0)}=1$. Using (\ref{eq:estimates2}), we see that
\[
1<\rho'^2<\Big(1+\frac{1}{a+1}\Big)^2<1+\frac{2}{a+1}+\frac{1}{(a+1)^2}<2
\] 
for $a\geq 7$. It implies that $a_2^{(0)}=1$ in these cases of $a$. The same is also true for $4\leq a\leq 6$ which can be easily checked by a computer program (we used Mathematica). Then we have
\[
\alpha_1^{(1)}=\frac{\rho'^2-1}{-\rho'-1}=1-\rho'.
\]
The fact that
\[
\alpha_2^{(1)}=\frac{1}{-\rho'-1}=-2-(a+1)\rho'+\rho'^2
\] 
can be verified using $\rho'^3-a\rho'^2-(a+3)\rho'-1=0$.

We will proceed similarly with $\bm{\alpha}^{(1)}$. We can immediately conclude that $a_1^{(1)}=\lfloor 1-\rho'\rfloor=2$. On the other hand, 
the element $\alpha_2^{(1)}=-2-(a+1)\rho'+\rho'^2$ is a conjugate of $\rho'$, namely $\rho$, for which we have $a_2^{(1)}=\lfloor\rho\rfloor=a+1$.

In the second iteration, we have $a_2^{(2)}=\lfloor \rho\rfloor=a+1$. For $a_1^{(2)}$, we can conclude that
\[
\alpha_1^{(2)}=a+4+a\rho'-\rho'^2>a+4-a\Big(1+\frac{1}{a+1}\Big)-\Big(1+\frac{1}{a+1}\Big)^2=1+\frac{a^2+a-1}{(a+1)^2}>1
\]
and
\[
\alpha_1^{(2)}=a+4+a\rho'-\rho'^2<a+4-a\Big(1+\frac{1}{a+2}\Big)-\Big(1+\frac{1}{a+2}\Big)^2=2-\frac{1}{(a+2)^2}<2.
\]
Therefore, $a_1^{(2)}=1$. We can proceed similarly for other iteration.

We will only discuss in detail the fourth iteration, where the expansions start to differ for odd and even coefficients $a$. In particular, we will focus on the determination of $a_2^{(4)}$. One can easily check that $\alpha_2^{(4)}$ is a root of the polynomial
\[
h(x)=(2a+3)x^3-a^2x^2-(a^2+2a+3)x+1,
\]
which has two positive and one negative root. In particular $h(0)=1>0$, $h(1)=1-2a^2<0$, $8h\big(\frac{a}{2}\big)=- a^3-8a^2-12a+8<0$ and $8h\big(\frac{a+1}{2}\big)=a^3+a^2-9a-1>0$. The element $\alpha_2^{(4)}$ is one of these two positive roots. Moreover, 
\[
-3a-4-(a^2+3a+1)\rho'+(a+2)\rho'^2>-3a-4+(a^2+3a+1)+(a+2)=a^2+a-1>2a+3,
\]
which implies $\alpha_2^{(4)}>1$. It follows that $\lfloor\alpha_2^{(4)}\rfloor=\lfloor\frac{a}{2}\rfloor=A_0$.   
\end{proof} 

Moreover, we can state the following proposition:

\begin{proposition} 
The Jacobi--Perron expansion of the vector $(-\rho',\rho'^2)$ is periodic.
\end{proposition}

\begin{proof}
The verification for $-1\leq a\leq 3$ is trivial. However, in these few cases, the iJPA expansion of $(-\rho',\rho'^2)$ differs from the expansion for $a\geq 4$. For $a=-1$, it has the preperiod length $1$ and the period length $2$. For $a=0$, it is $2$ and $3$, and for $1\leq a\leq 3$, it is equal to $2$ and $5$. The statement for $a\geq 4$ follows from Proposition \ref{prop:ijpasimplestperiodic4}.   
\end{proof}

Note that a similar result can be also derived for the iJPA expansion of the vector $(-\rho'',\rho'')$ where $-1<\rho''<0$. We also studied the iJPA expansion of $(\rho,\rho^2)$, but except for a few values of $a$ (e.g., $a=-1,0,2$), it does not seem to be periodic.

\subsection{Homogenous Jacobi--Perron expansion in the simplest cubic fields} \label{subsec:homJPAsimplest}

Now, from the inhomogeneous JPA expansion of $(-\rho',\rho'^2)$, we obtain the homogeneous JPA expansion of the vector $(1,-\rho',\rho'^2)$.

\begin{theorem} \label{thm:hjpasimplest}
Let $a\geq 4$. Then the homogenous Jacobi--Perron expansion of the vector $(1,-\rho',\rho'^2)$ is periodic. The preperiod is equal to
\begin{align*}
\bm{\beta}^{(0)}&=(1,-\rho',\rho'^2),\\
\bm{\beta}^{(1)}&=(-1-\rho',-1+\rho'^2,1).
\end{align*}
The period length is $7$. The iterations (2)-(4) are equal to
\begin{align*}
\bm{\beta}^{(2)}&=(1+2\rho'+\rho'^2,a+2+(a+1)\rho',-1-\rho'),\\
\bm{\beta}^{(3)}&=(a+1+(a-1)\rho'-\rho'^2,-(a+2)-(2a+3)\rho'-(a+1)\rho'^2,1+2\rho'+\rho'^2),\\
\bm{\beta}^{(4)}&=(-(a+2)-(2a+3)\rho'-(a+1)\rho'^2,-a-(a-3)\rho'+2\rho'^2,a+1+(a-1)\rho'-\rho'^2).
\end{align*}
If $a$ is even, the iterations (5)-(7) are 
\begin{align*}
\bm{\beta}^{(5)}&=(-a-(a-3)\rho'+2\rho'^2,2A_0^2+4A_0+1+(4A_0^2+5A_0-1)\rho'+(2A_0^2+A_0-1)\rho'^2,\\&\hspace{7cm}-(a+2)-(2a+3)\rho'-(a+1)\rho'^2),\\
\bm{\beta}^{(6)}&=(2A_0^2+4A_0+1+(4A_0^2+5A_0-1)\rho'+(2A_0^2+A_0-1)\rho'^2,a-2-9\rho'-(a+5)\rho'^2,\\&\hspace{7cm}-a-(a-3)\rho'+2\rho'^2),\\
\bm{\beta}^{(7)}&=(a-2-9\rho'-(a+5)\rho'^2,-2A_0^2-6A_0-1-(4A_0^2+7A_0-4)\rho'-(2A_0^2+A_0-3)\rho'^2,\\&\hspace{6cm}2A_0^2+4A_0+1+(4A_0^2+5A_0-1)\rho'+(2A_0^2+A_0-1)\rho'^2).
\end{align*}
If $a$ is odd, the iterations (5)-(7) are
\begin{align*}
\bm{\beta}^{(5)}&=(-a-(a-3)\rho'+2\rho'^2,2A_0^2+5A_0+2+(4A_0^2+7A_0)\rho'+(2A_0^2+2A_0-1)\rho'^2,\\&\hspace{7cm}-(a+2)-(2a+3)\rho'-(a+1)\rho'^2),\\
\bm{\beta}^{(6)}&=(2A_0^2+7A_0+3+(4A_0^2+9A_0-2)\rho'+(2A_0^2+2A_0-3)\rho'^2,a-2-9\rho'-(a+5)\rho'^2,\\&\hspace{7cm}-a-(a-3)\rho'+2\rho'^2),\\
\bm{\beta}^{(7)}&=(a-2-9\rho'-(a+5)\rho'^2,-2A_0^2-9A_0-4-(4A_0^2+11A_0-4)\rho'-(2A_0^2+2A_0-5)\rho'^2,\\&\hspace{6cm}2A_0^2+7A_0+3+(4A_0^2+9A_0-2)\rho'+(2A_0^2+2A_0-3)\rho'^2).
\end{align*}
And finally, the iteration (8) is
\begin{align*}
\bm{\beta}^{(8)}&=(-a^2-5-(a^2-a+14)\rho'-7\rho'^2,7+(a^2+7a+26)\rho'+(a^2+6a+14)\rho'^2,\\&\hspace{10cm}a-2-9\rho'-(a+5)\rho'^2).
\end{align*} 
\end{theorem}
 
\begin{proof}
To prove this theorem, it is enough to combine results of Propositions \ref{prop:ijpasimplestperiodic4}, \ref{lm:iJPA-JPA} and \ref{lm:iJPA-JPA,per}. As suggested at the end of Section \ref{sec:prelijpa}, in each iteration, it suffices to compute a concrete linear combination of some elements. Thus, in the first iteration, we get
\begin{align*}
    \beta_1^{(1)}&=\beta_2^{(0)}-a_1^{(0)}\beta_{1}^{(0)}=-\rho'-1\cdot 1=-1-\rho',\\
    \beta_2^{(1)}&=\beta_3^{(0)}-a_2^{(0)}\beta_{1}^{(0)}=\rho'^2-1\cdot 1=-1+\rho'^2,\\
    \beta_3^{(1)}&=\beta_1^{(0)}=1
\end{align*}
where $\bm{a}^{(0)}=(a_1^{(0)},a_2^{(0)})$ is the vector of integer parts from Proposition \ref{prop:ijpasimplestperiodic4}. 

Similarly, in the next step, we get
\begin{align*}
    \beta_1^{(2)}&=\beta_2^{(1)}-a_1^{(1)}\beta_{1}^{(1)}=-1+\rho'^2-2\cdot (-1-\rho')=1+2\rho'+\rho'^2,\\
    \beta_2^{(1)}&=\beta_3^{(1)}-a_2^{(1)}\beta_{1}^{(1)}=1-(a+1)\cdot (-1-\rho')=a+2+(a+1)\rho',\\
    \beta_3^{(2)}&=\beta_1^{(1)}=-1-\rho'.
\end{align*}
Since these computation are trivial, we omit the remaining steps.
\end{proof}

\section{Jacobi--Perron algorithm in Ennola's cubic fields} \label{sec:jpaennola}

Let us recall that Ennola's cubic fields are fields generated by a root $\rho$ of the polynomial $x^3+(a-1)x^2-ax-1$ where $a \in \N, a\geq 3$. Here, we denote the conjugates of $\rho$ as 
\begin{align*}
1+\frac{1}{a+3}&<\rho<1+\frac{1}{a+2},\\
-\frac{1}{a}&<\rho'<-\frac{1}{a+1},\\
-a+\frac{1}{a^2+a}&<\rho''<-a+\frac{1}{a^2}.
\end{align*}
Similarly, like for the simplest cubics fields, we will consider JPA expansions of vectors of the form $(1,|\theta|,\theta^2)$ where $\theta\in\{\rho,\rho',\rho''\}$, prove their periodicity and determine their form.

Moreover, in Subsection \ref{subsec:ennola2polynomial}, we will discuss the same question for roots of the polynomial $x^3-(a-1)x^2-ax-1$ where $a \in \N, a\geq 5$. Note that roots of these polynomial also generate Ennola's cubic fields. In particular, if $\rho$ is a root of $x^3+(a-1)x^2-ax-1$, then $a\rho+\rho^2$ is a root of $x^3-(a+1)x^2-(a+2)x-1$.   

\subsection{First generating polynomial} \label{subsec:ennola1polynomial}
First of all, let us note that the proofs of the following propositions are technical and similar to the proofs for the simplest cubic fields. Therefore, we do not include them (some of the proofs can be found in \cite{Sg} or in Appendix A7 of the arxiv version of this paper \cite{KST}). 

We will start with the root $1<\rho<2$ of the polynomial $x^3+(a-1)x^2-ax-1$.

\begin{proposition} \label{prop:JPAennol1}
Let $a\geq 3$.
Then the Jacobi--Perron expansion of the vector $(1, \rho,\rho^2)$ is periodic. The preperiod of the JPA expansion of $(1, \rho,\rho^2)$ is 
\[
\begin{array}{ll}
\bm{\beta}^{(0)}=\left(1, \rho, \rho^2\right), & \bm{b}^{(0)}=(1,1), \\
\bm{\beta}^{(1)}=\left(-1+\rho, -1+\rho^2, 1 \right), & \bm{b}^{(1)}=(2,a+2), \\
\end{array}
\] 
and the period is 
\[
\begin{array}{ll}
\bm{\beta}^{(2)}=\left(1-2\rho+\rho^2, a+3-(a+2)\rho, -1+\rho\right), & \bm{b}^{(2)}=(0,a+2), \\
\bm{\beta}^{(3)}=\left(a+3-(a+2)\rho, -(a+3)+(2a+5)\rho-(a+2)\rho^2, 1-2\rho+\rho^2\right), & \bm{b}^{(3)}=(0,a+3).
\end{array}
\] 
\end{proposition}

We will proceed with the root $-1<\rho'<0$.

\begin{proposition} \label{prop:JPAennol2}
Let $a\geq 3$.
Then the Jacobi--Perron expansion of the vector $\left(1, -\rho',\rho'^2\right)$ is periodic. The preperiod of the JPA expansion of $\left(1,-\rho',\rho'^2\right)$ is 
\[
\begin{array}{ll}
\bm{\beta}^{(0)}=\left(1, - \rho', \rho'^2\right), & \bm{b}^{(0)}=(0,0),\\
\bm{\beta}^{(1)}=\left(-\rho', \rho'^2, 1\right), & \bm{b}^{(1)}=(0,a), \\
\bm{\beta}^{(2)}=\left(\rho'^2, 1+a\rho', -\rho'\right), & \bm{b}^{(2)}=(a-2,a).
\end{array}
\] 
and the period is 
\[
\begin{array}{ll}
\bm{\beta}^{(3)}=\left(1+a\rho'-(a-2)\rho'^2, -\rho'-a\rho'^2, \rho'^2\right), & \bm{b}^{(3)}=(0,1),  \\
\bm{\beta}^{(4)}=\left(-\rho'-a\rho'^2, -1-a\rho'+(a-1)\rho'^2, 1+a\rho'-(a-2)\rho'^2 \right), & \bm{b}^{(4)}=(0,1), \\
\bm{\beta}^{(5)}=\left(-1-a\rho'+(a-1)\rho'^2, 1+(a+1)\rho'+2\rho'^2, -\rho'-a\rho'^2\right), & \bm{b}^{(5)}=(0,a-2),  \\
\bm{\beta}^{(6)}=\big( 1+(a+1)\rho'+2\rho'^2, a-2+(a^2-2a-1)\rho'-(a^2-2a+2)\rho'^2,\\ \hspace{7cm} -1-a\rho'+(a-1)\rho'^2\big),  & \bm{b}^{(6)}=(0,1), \\
\bm{\beta}^{(7)}=\big(a-2+(a^2-2a-1)\rho'-(a^2-2a+2)\rho'^2, \\ \hspace{3cm} -2-(2a+1)\rho'+(a-3)\rho'^2, 1+(a+1)\rho'+2\rho'^2 \big), & \bm{b}^{(7)}=(0,1),\\
\bm{\beta}^{(8)}=\big(-2-(2a+1)\rho'+(a-3)\rho'^2,-a+3-(a^2-3a-2)\rho' \\ \hspace{1cm} +(a^2-2a+4)\rho'^2,a-2+(a^2-2a-1)\rho'-(a^2-2a+2)\rho'^2\big), & \bm{b}^{(8)}=(0,a), \\
\bm{\beta}^{(9)}=\big(-a+3-(a^2-3a-2)\rho'+(a^2-2a+4)\rho'^2, \\ \hspace{3cm} 3a-2+(3a^2-a-1)\rho'-(2a^2-5a+2)\rho'^2, \\ \hspace{6cm} -2-(2a+1)\rho'+(a-3)\rho'^2\big), & \bm{b}^{(9)}=(a-3,a).
\end{array}
\] 
\end{proposition}

And finally, we discuss the root $-a<\rho''<-a+1$, for which the corresponding JPA expansion is again periodic as in the previous cases.

\begin{proposition} \label{prop:JPAennol3}
Let $a\geq 3$.
Then the Jacobi--Perron expansion of the vector $\left(1,-\rho'',\rho''^2\right)$ is periodic. The preperiod of the JPA expansion of $\left(1,-\rho'',\rho''^2\right)$ is 
\[
\begin{array}{ll}
\bm{\beta}^{(0)}=\left(1, - \rho'', \rho''^2\right), \qquad \bm{b}^{(0)}=(a-1,a^2-1), \\
\end{array}
\] 
and the period is 
\[
\begin{array}{ll}
\bm{\beta}^{(1)}=\left(-a+1-\rho'', -a^2+1+\rho''^2, 1\right), & \bm{b}^{(1)}=(0,1), \\
\bm{\beta}^{(2)}=\left(-a^2+1+\rho''^2, a+\rho, -a+1-\rho''\right), & \bm{b}^{(2)}=(0,1), \\
\bm{\beta}^{(3)}=\left(a+\rho'', a^2-a-\rho''-\rho''^2, -a^2+1+\rho''^2 \right), & \bm{b}^{(3)}=(2a-2,a^2-a-1).
\end{array}
\] 
\end{proposition}

\subsection{Second generating polynomial} \label{subsec:ennola2polynomial}
In this subsection, we will focus on roots of the polynomial $x^3-(a-1)x^2-ax-1$ where $a\geq 5$. Note that these elements also generate Ennola's cubic fields. 
Let us denote the roots of the polynomial $x^3-(a-1)x^2-ax-1$ by $\psi, \psi'$ and $\psi''$ where
\[ a<a+\frac{a-1}{a^3}<\psi<a+\frac{a^2-1}{a^4}<a+\frac{1}{a^2},\]
\[-1 < -\frac{a-1}{a}<\psi' < -\frac{a-2}{a-1} < 0,\]
\[-1 < -\frac{1}{a-2}<\psi'' < -\frac{1}{a-1} < 0. \]
Note that we have $\rho=-(a-1)-(a-1)\psi+\psi^2$ where $1<\rho$ is a root of the polynomial $x^3+(a-3)x^2-(a-2)x-1$.

First of all, we will derive the JPA expansion of $(1,\psi,\psi^2)$.

\begin{proposition} \label{prop:JPAennol4}
Let $a\geq 5$.
Then the Jacobi--Perron expansion of the vector $\left(1,\psi,\psi^2\right)$ is periodic. The preperiod of the JPA expansion of $\left(1,\psi,\psi^2\right)$ is 
\[
\begin{array}{ll}
\bm{\beta}^{(0)}=\left(1, \psi, \psi^2\right), & \bm{b}^{(0)}=(a,a^2),\\
\bm{\beta}^{(1)}=\left(-a+\psi, -a^2+\psi^2, 1 \right), & \bm{b}^{(1)}=(2a,a^2+a),  
\end{array}
\] 
and the period is
\[
\begin{array}{ll}
\bm{\beta}^{(2)}=\left(a^2-2a\psi+\psi^2, a^3+a^2+1-(a^2+a)\psi, -a+\psi\right), & \bm{b}^{(2)}=(2a+1,a^2+a).  
\end{array}
\] 
\end{proposition}

Our next step is to determine the JPA expansion of $\left(1, -\psi',\psi'^2\right)$.

\begin{proposition} \label{prop:JPAennol5}
Let $a\geq 5$.
Then the Jacobi--Perron expansion of the vector $\left(1,-\psi',\psi'^2\right)$ is periodic. The preperiod of the JPA expansion of $\left(1, -\psi',\psi'^2\right)$ is 
\[
\begin{array}{ll}
\bm{\beta}^{(0)}=\left(1, - \psi', \psi'^2\right), & \bm{b}^{(0)}=(0,0),\\
\bm{\beta}^{(1)}=\left(-\psi', \psi'^2, 1 \right), & \bm{b}^{(1)}=(0,1), \\
\bm{\beta}^{(2)}=\left(\psi'^2, 1+\psi', -\psi'\right), & \bm{b}^{(2)}=(0,1), \\
\bm{\beta}^{(3)}=\left(1+\psi', -\psi'-\psi'^2, \psi'^2 \right), & \bm{b}^{(3)}=(0,a-3).
\end{array}
\] 
The period is
\[
\begin{array}{ll}
\bm{\beta}^{(4)}=\left(-\psi'-\psi'^2, -a+3 -(a-3)\psi' + \psi'^2, 1+\psi'\right),  & \bm{b}^{(4)}=(0,1),\\
\bm{\beta}^{(5)}=\left(-a+3 - (a-3)\psi' + \psi'^2, 1+2\psi'+\psi'^2, -\psi'-\psi'^2\right), & \bm{b}^{(5)}=(0,1), \\
\bm{\beta}^{(6)}=\left(1+2\psi'+\psi'^2, a-3+(a-4)\psi' -2\psi'^2, -a+3 -(a-3)\psi' + \psi'^2\right), & \bm{b}^{(6)}=(0,a-2), \\
\bm{\beta}^{(7)}=\big(a-3+(a-4)\psi' -2\psi'^2, -2a+5-(3a-7)\psi' -(a-3)\psi'^2,\\ \hspace{8cm} 1+2\psi'+\psi'^2 \big),  & \bm{b}^{(7)}=(a-5,a-2),\\
\bm{\beta}^{(8)}=\big(-a^2+6a-10-(a^2-6a+13)\psi' +(a-7)\psi'^2,\\ \hspace{3cm} -a^2+5a-5-(a^2-6a+6)\psi' +(2a-3)\psi'^2,\\ \hspace{6cm} a-3+(a-4)\psi' -2\psi'^2 \big), & \bm{b}^{(8)}=(0,1),\\
\bm{\beta}^{(9)}=\big(-a^2+5a-5-(a^2-6a+6)\psi' +(2a-3)\psi'^2,\\ \hspace{3cm} a^2-5a+7+(a^2-5a+9)\psi' -(a-5)\psi'^2,\\ \hspace{3cm} -a^2+6a-10-(a^2-6a+13)\psi' +(a-7)\psi'^2\big), & \bm{b}^{(9)}=(0,1), \\
\bm{\beta}^{(10)}=\big(a^2-5a+7+(a^2-5a+9)\psi' -(a-5)\psi'^2,\\ \hspace{3cm} a-5-7\psi' -(a+4)\psi'^2,\\ \hspace{3cm} -a^2+5a-5-(a^2-6a+6)\psi' +(2a-3)\psi'^2\big), & \bm{b}^{(10)}=(0,a-4).  
\end{array}
\] 
\end{proposition}

In the end, we will focus on the last root $\psi''$.

\begin{proposition} \label{prop:JPAennol6}
Let $a\geq 5$.
Then the Jacobi--Perron expansion of the vector $\left(1, -\psi'',\psi''^2\right)$ is periodic. The preperiod of the JPA expansion of $\left(1, -\psi'',\psi''^2\right)$ is 
\[
\begin{array}{ll}
\bm{\beta}^{(0)}=\left(1, - \psi'', \psi''^2\right), & \bm{b}^{(0)}=(0,0),\\
\bm{\beta}^{(1)}=\left(-\psi'', \psi''^2, 1\right), & \bm{b}^{(1)}=(0,a-2), \\
\bm{\beta}^{(2)}=\left(\psi''^2, 1+(a-2)\psi'', -\psi''\right), & \bm{b}^{(2)}=(a-2,a-2), 
\end{array}
\]
and the period is
\[
\begin{array}{ll}
\bm{\beta}^{(3)}=\left(1+(a-2)\psi'' - (a-2)\psi''^2, -\psi''-(a-2)\psi''^2, \psi''^2\right), & \bm{b}^{(3)}=(1,1),\\
\bm{\beta}^{(4)}=\big(-1-(a-1)\psi'', -1 - (a-2)\psi'' + (a-1)\psi''^2, \\ \hspace{5cm} 1+(a-2)\psi'' - (a-2)\psi''^2\big), & \bm{b}^{(4)}=(1,a-4), \\
\bm{\beta}^{(5)}=\big(\psi'' + (a-1)\psi''^2, a-3+ (a^2-4a+2)\psi''  -(a-2)\psi''^2, \\ \hspace{6cm} -1-(a-1)\psi''\big), & \bm{b}^{(5)}=(a-3,a-2). 
\end{array}
\] 
\end{proposition}

Summarizing results of the previous propositions, we can make the following conclusion which is Theorem \ref{thm:m3} from the introduction. 

\begin{theorem}
We have
\begin{enumerate}
    \item if $\tau$ is a root of the polynomial $x^3+(a-1)x^2-ax-1$ where $a\geq 3$, then the Jacobi--Perron expansion of the vector $(1, |\tau|,\tau^2)$ is periodic,
    \item if $\tau$ is a root of the polynomial $x^3-(a-1)x^2-ax-1$ where $a\geq 5$, then the Jacobi--Perron expansion of the vector $(1, |\tau|,\tau^2)$ is periodic.
\end{enumerate}

\end{theorem}

\begin{proof}
The first statement follows from Propositions \ref{prop:JPAennol1}, \ref{prop:JPAennol2} and \ref{prop:JPAennol3}, the second one from Propositions~\ref{prop:JPAennol4}, \ref{prop:JPAennol5} and \ref{prop:JPAennol6}.  
\end{proof}

\section{Indecomposability of (semi)convergents} \label{sec:semi}

In this section, we will discuss elements of the form (\ref{eq:semi}) and compare them with $\mathfrak{s}$-indecomposables in the corresponding fields. As we will see later, in some cases, we can obtain this way only $\mathfrak{s}$-indecomposables. However, some periodic JPA expansions can produce $\mathfrak{s}$-decomposables both as convergents and semiconvergents (see, e.g., Example \ref{ex:decomposable}).

Recall that in this paper, by semiconvergents of the Jacobi--Perron expansion of the vector $\bm{\beta}=\bm{\beta}^{(0)}$, we mean elements of the form
\begin{equation} 
\delta_{i,j}^{(k)}=\beta_i^{(k)}-j\beta_1^{(k)}
\end{equation}
where $2\leq i \leq n$, $0\leq j\leq \left\lfloor\frac{\beta_i^{(k)}}{\beta_1^{(k)}}\right\rfloor-1$ and $k\in\N_{0}$. Moreover, the elements with $j=0$ are called convergents.

\subsection{The simplest cubic fields} \label{subsec:semithesimplest}

Recall that in the simplest cubic fields generated by a root $\rho$ of the polynomial $x^3-ax^2-(a+3)x-1$, all $\mathfrak{s}$-indecomposables can be obtained as unit multiples of elements $1$, $1+\rho+\rho^2$ and elements of the form
\[
\theta_{v,W}=-v-(v(a+2)+1+W)\rho+(v+1)\rho^2
\] 
where $0\leq v\leq a$ and $0\leq W\leq a-v$ which form a triangle. In the following proposition, we show that only unit multiples of $\theta_{v,W}$ appear in the JPA expansion of the vector $(1,-\rho',\rho'^2)$ which was derived in Theorem \ref{thm:hjpasimplest}. Recall that $-2<\rho'<-1$ is a root of the polynomial $x^3-ax^2-(a+2)x-1$ where $a\geq -1$.

\begin{theorem} \label{thm:semicon}
The semiconvergents of the JPA expansion of $(1,-\rho',\rho'^2)$ are all $\mathfrak{s}$-indecomposable in their signatures.
\end{theorem}

\begin{proof}
By simple comparing and basic calculations, we can immediately decide that all the semiconvergents of $(1,-\rho',\rho'^2)$ are unit multiples of some element from Theorem \ref{thm:indesimplest}. The verification for $-1\leq a\leq 3$ is trivial and can be found in Appendix A8 of the arxiv version of this paper \cite{KST}. The results for  $a\geq 4$ are in Table \ref{tab:semi aeven} (for $a$ even) and Table \ref{tab:semi aodd} (for $a$ odd). Each of these tables contain the following information: norm of every semiconvergent $\delta_{i,j}^{(k)}$, and (totally positive) indecomposable integer $\alpha$ from Theorem \ref{thm:indesimplest} and a unit $\varepsilon\in\Z[\rho]$ (expressed using fundamental units $\rho$ and $\rho'$) such that $\delta_{i,j}^{(k)}=\alpha\varepsilon$.   

Moreover, we know that the JPA expansion is periodic, and so if $k\geq 5$, then $\delta_{i,j}^{(k+5)}= \varepsilon\delta_{i,j}^{(k)}$ for some unit $\varepsilon$. Thus, it suffices to only go through the iterations $k\leq 8$.
\end{proof}

\begin{table}[htbp]
\centering
\begin{tabular}{|c|c|c|c|}
\hline
 $\delta_{i,j}^{(k)}$ & Norm $N(\delta_{i,j}^{(k)})$ & Indecomposable integer $\alpha$ & Unit $\varepsilon$ \\
\hline
$\delta_{2,0}^{(1)}$ & $-(2a+3)$ & $\theta'_{0,0}$ & $-\rho^{-1}\rho'^{-1}$ \\
\hline
$\delta_{2,1}^{(1)}$ & $-1$ & $1$ & $-\rho^{-1}\rho'$ \\
\hline
$\delta_{3,j}^{(1)}$  &  & &   \\
$1\leq j\leq a$ & $-j^3+aj^2+(a+3)j+1$ & $\theta'_{a-j+1,0}$ & $\rho^{-2}\rho'^{-2}$ \\
\hline
$\delta_{2,0}^{(2)}$ & $2a+3$ & $\theta'_{0,0}$ & $\rho^{-2}\rho'^{-2}$ \\
\hline
$\delta_{3,0}^{(2)}$ & $1$ & $1$ & $\rho^{-1}$\\
\hline
$\delta_{3,j}^{(2)}$  &  &  &  \\
$1\leq j\leq a$ & $-j^3+aj^2+(a+3)j+1$ & $\theta'_{a-j+1,0}$& $\rho^{-3}\rho'^{-2}$\\
\hline
$\delta_{3,0}^{(3)}$ & $1$ & $1$ & $\rho^2$ \\ 
\hline
$\delta_{3,0}^{(4)}$ & $-1$ & $1$ & $-\rho^{-2}\rho'^{-1}$ \\
\hline
$\delta_{3,j}^{(4)}$  & $-(2a+3)j^3+a^2j^2$  & & \\
$1\leq j\leq A_0-1$ & $+(a^2+2a+3)j-1$ & $\theta'_{a-2j,j+1}$& $\rho^{-4}\rho'^{-2}$\\
\hline
$\delta_{3,0}^{(5)}$ & $2a+3$ & $\theta'_{0,0}$ & $\rho^{-3}\rho'^{-2}$\\
\hline
$\delta_{3,1}^{(5)}$ & $1$ & $1$ & $\rho^{-3}$ \\
\hline
$\delta_{3,0}^{(6)}$ & $-1$ & $1$ & $-\rho^{-3}\rho'^{-1}$ \\
\hline
$\delta_{2,j}^{(7)}$  & $j^3-(A_0-6)j^2-(6A_0-9)j$  & &  \\
$0\leq j\leq A_0-3$ & $-A_0^2 j+A_0^3-5A_0+3$ & $\theta'_{A_0-j-3,a-(A_0-j-3)}$& $\rho^{-5}\rho'^{-2}$\\
\hline
$\delta_{3,j}^{(7)}$  & $j^3-(A_0+6)j^2+(2A_0+9)j$  & &  \\
$0\leq j\leq A_0+1$ & $-A_0^2 j+A_0^3+4A_0^2+3A_0-1$ & $\theta'_{0,A_0+1-j}$ & $\rho^{-4}\rho'^{-2}$\\
\hline
$\delta_{3,A_0+2}^{(7)}$ & $1$ & $1$ & $\rho^{-4}$\\
\hline
$\delta_{2,0}^{(8)}$ & $-1$ & $1$ & $-\rho^{-5}\rho'^{-1}$\\
\hline
$\delta_{3,0}^{(8)}$ & $-1$ & $1$ & $-\rho^{-4}\rho'^{-1}$\\
\hline
$\delta_{3,1}^{(8)}$ & $1$ & $1$ & $\rho^{-4}\rho'^{-2}$\\
\hline
$\delta_{3,j}^{(8)}$  &  & & \\
$2\leq j\leq a+1$ & $-j^3+(a+3)j^2-aj-1$ & $\theta'_{a-j+2,0}$ &  $\rho^{-6}\rho'^{-4}$\\
\hline
\end{tabular}
\caption{\textsc{The simplest cubic fields}: Semiconvergents $\delta_{i,j}^{(k)}$ of the JPA expansion of $(1,-\rho',\rho'^2)$ for $a=2A_0\geq 4$, their norm, the corresponding indecomposable integer $\alpha$ and the unit $\varepsilon$ such that $\delta_{i,j}^{(k)}=\alpha\varepsilon$.} \label{tab:semi aeven}
\end{table}

\begin{table}[htbp]
\centering
{\fontsize{9.9}{12} \selectfont
\begin{tabular}{|c|c|c|c|}
\hline
 $\delta_{i,j}^{(k)}$ & Norm $N(\delta_{i,j}^{(k)})$ & Indecomposable integer $\alpha$ & Unit $\varepsilon$ \\
 \hline
$\delta_{2,0}^{(1)}$ & $-(2a+3)$ & $\theta'_{0,0}$ & $-\rho^{-1}\rho'^{-1}$ \\
\hline
$\delta_{2,1}^{(1)}$ & $-1$ & $1$ & $-\rho^{-1}\rho'$ \\
\hline
$\delta_{3,j}^{(1)}$  &  & &   \\
$1\leq j\leq a$ & $-j^3+aj^2+(a+3)j+1$ & $\theta'_{a-j+1,0}$ & $\rho^{-2}\rho'^{-2}$ \\
\hline
$\delta_{2,0}^{(2)}$ & $2a+3$ & $\theta'_{0,0}$ & $\rho^{-2}\rho'^{-2}$ \\
\hline
$\delta_{3,0}^{(2)}$ & $1$ & $1$ & $\rho^{-1}$\\
\hline
$\delta_{3,j}^{(2)}$  &  &  &  \\
$1\leq j\leq a$ & $-j^3+aj^2+(a+3)j+1$ & $\theta'_{a-j+1,0}$& $\rho^{-3}\rho'^{-2}$\\
\hline
$\delta_{3,0}^{(3)}$ & $1$ & $1$ & $\rho^2$ \\ 
\hline
$\delta_{3,0}^{(4)}$ & $-1$ & $1$ & $-\rho^{-2}\rho'^{-1}$ \\
\hline
$\delta_{3,j}^{(4)}$  & $-(2a+3)j^3+a^2j^2$  & &  \\
$1\leq j\leq A_0-1$ & $+(a^2+2a+3)j-1$ & $\theta'_{a-2j,j+1}$ & $\rho^{-4}\rho'^{-2}$\\
\hline
$\delta_{2,0}^{(5)}$ & $3A_0^3+9A_0^2+6A_0-1$ & $\theta'_{1,A_0+1}$ &$\rho^{-4}\rho'^{-2}$\\
\hline
$\delta_{3,0}^{(5)}$ & $2a+3$ & $\theta'_{0,0}$ & $\rho^{-3}\rho'^{-2}$\\
\hline
$\delta_{3,1}^{(5)}$ & $1$ & $1$ & $\rho^{-3}$\\
\hline
$\delta_{3,0}^{(6)}$ & $-1$ & $1$ & $-\rho^{-3}\rho'^{-1}$ \\
\hline
$\delta_{2,j}^{(7)}$  & $j^3-(A_0-7)j^2-(8A_0-12)j$  & &  \\
$0\leq j\leq A_0-3$ & $-A_0^2 j+A_0^3+A_0^2-8A_0+5$ & $\theta'_{A_0-j-3,a-(A_0-j-3)}$ & $\rho^{-5}\rho'^{-2}$\\
\hline
$\delta_{3,j}^{(7)}$  & $j^3-(A_0+8)j^2+(2A_0+17)j$  & &  \\
$0\leq j\leq A_0+2$ & $-A_0^2 j+A_0^3+6A_0^2+7A_0-5$ & $\theta'_{0,A_0+2-j}$ & $\rho^{-4}\rho'^{-2}$\\
\hline
$\delta_{3,A_0+3}^{(7)}$ & $1$ & $1$ & $\rho^{-4}$\\
\hline
$\delta_{2,0}^{(8)}$ & $-1$ & $1$ & $-\rho^{-5}\rho'^{-1}$\\
\hline
$\delta_{3,0}^{(8)}$ & $-1$ & $1$ & $-\rho^{-4}\rho'^{-1}$\\
\hline
$\delta_{3,1}^{(8)}$ & $1$ & $1$ & $\rho^{-4}\rho'^{-2}$\\
\hline
$\delta_{3,j}^{(8)}$  &  & &  \\
$2\leq j\leq a+1$ & $-j^3+(a+3)j^2-aj-1$ & $\theta'_{a-j+2,0}$ & $\rho^{-6}\rho'^{-4}$\\
\hline
\end{tabular}
}
\caption{\textsc{The simplest cubic fields}: Semiconvergents $\delta_{i,j}^{(k)}$ of the JPA expansion of $(1,-\rho',\rho'^2)$ for $a=2A_0+1\geq 4$, their norm, the corresponding indecomposable integer $\alpha$ and the unit $\varepsilon$ such that $\delta_{i,j}^{(k)}=\alpha\varepsilon$.} \label{tab:semi aodd}
\end{table}

The previous theorem could lead one to hope that the situation is so nice also in general. Unfortunately, this is not the case.
In the following two examples, we will show that even in the simplest cubic fields, the periodicity of JPA expansion does not guarantee that the convergents (or semiconvergents) are $\mathfrak{s}$-indecomposable.

\begin{example} \label{ex:decomposable}
Let us consider the JPA expansion of $(1,\rho,\rho^2)$ where $3<\rho<4$ is the largest root of the polynomial $x^3-2x^2-5x-1$, i.e., $a=2$. It can be easily computed that it is periodic with the preperiod $\bm{\beta}^{(0)}=(1,\rho,\rho^2)$ and $\bm{\beta}^{(1)}=(-3+\rho,-12+\rho^2,1)$, and the period 
\begin{align*}
\bm{\beta}^{(2)}&=(-12+\rho^2,4-\rho,-3+\rho^2),\\  
\bm{\beta}^{(3)}&=(16-\rho-\rho^2,9+\rho-\rho^2,-12+\rho^2),\\
\bm{\beta}^{(4)}&=(-7+2\rho, -28+\rho+2\rho^2,16-\rho-\rho^2),\\
\bm{\beta}^{(5)}&=(21-13\rho+2\rho^2,107-27\rho-\rho^2,-7+2\rho).
\end{align*}
However, the convergent $\beta_1^{(2)}=-12+\rho^2$ is $\mathfrak{s}$-decomposable. It follows from the fact that its norm is $37$, and the only absolute values of norms of $\mathfrak{s}$-indecomposables in $\O_K=\Z[\rho]$ are $1$, $7$, $11$ and $19$. The same is also true for the convergent $\beta_1^{(3)}=16-\rho-\rho^2$ with the norm $49$ which is, moreover, totally positive.  

In the second example, we will show that even the $\mathfrak{s}$-indecomposability of convergents does not imply that all the semiconvergents are also $\mathfrak{s}$-indecomposable. 
Let us consider the JPA expansion of the vector $(1,\rho,\rho^2)$ where $1<\rho<2$ is the largest root of the polynomial $x^3+x^2-2x-1$, i.e., $a=-1$. It can be easily computed that it is periodic with the preperiod $(1,\rho,\rho^2)$ and $(-1+\rho, -1+\rho^2,1)$, and the period $(1-2\rho+\rho^2, 5-4\rho,-1+\rho)$ and $(5-4\rho, -5+9\rho-4\rho^2,1-2\rho+\rho^2)$. In this case, all convergents are units. Therefore, they are $\mathfrak{s}$-indecomposable in their signatures.   
On the other hand, the semiconvergent $\delta_{3,2}^{(2)}=-3+5\rho-2\rho^2$ has the norm $13$. Since all the possible absolute values of norms of $\mathfrak{s}$-indecomposables are $1$ and $7$, this semiconvergent is $\mathfrak{s}$-decomposable in $\O_K$.   
\end{example}

\subsection{Ennola's cubic fields} \label{subsec:semiennola}

Recall that in Ennola's cubic fields which are fields generated by a root $\rho$ of the polynomial $x^3+(a-1)x^2-ax-1$ where $a \in \N, a\geq 3$, all $\mathfrak{s}$-indecomposables can be obtained as unit multiples of elements of the form
\begin{enumerate}
\item $\kappa_w= 1+w \rho + \rho^2$ where $1\leq w \leq a-1 $,
\item $\lambda_{v,w}=-v-(a(v-1)+w)\rho+(a(v-1)+w+1)\rho^2$ where $1\leq v \leq a-1$ and $\max\{1,v-1\}\leq w \leq a-1$,
\item $\mu_u=-1-u\rho+(u+2)\rho^2$ where $0\leq u \leq a-2$.
\end{enumerate}
As in Subsection \ref{subsec:semithesimplest}, we will look at the semiconvergents of the JPA expansion of some vectors of the form $(1,|\theta|, \theta^2)$ and discuss whether they are $\mathfrak{s}$-indecomposable. As we will see, in some cases, we again obtain only $\mathfrak{s}$-indecomposable integers in $\Z[\rho]$. However, sometimes, we get also $\mathfrak{s}$-decomposable integers despite the fact that the JPA expansion is periodic.

\begin{theorem}\label{th.8.2}
Let $\rho$, $\rho'$, $\rho''$, $\psi$, $\psi'$ and $\psi''$ be as in Subsections $\ref{subsec:ennola1polynomial}$ and $\ref{subsec:ennola2polynomial}$. Then
\begin{enumerate}
    \item all semiconvergents of JPA expansions of vectors $(1,-\rho',\rho'^2)$, $(1,-\psi',\psi'^2)$ and $(1,-\psi'',\psi''^2)$ are $\mathfrak{s}$-indecomposable in their signatures,
    \item some semiconvergent of JPA expansions of vectors $(1,\rho,\rho^2)$, $(1,-\rho'',\rho''^2)$ and $(1,\psi,\psi^2)$ is $\mathfrak{s}$-decomposable in its signature.
\end{enumerate}
\end{theorem}

\begin{proof}
In Tables \ref{tab:ennola2} and \ref{tab:ennola1}, we can find results for the JPA expansions of $(1,\rho,\rho^2)$ and $(1,-\rho'',\rho''^2)$.
In these cases, we can find some $\mathfrak{s}$-indecomposables as semiconvergents which is denoted by the symbol \xmark. Note that Table \ref{tab:ennola1} does not include $a=3$. However, the table for this concrete case is almost the same except for the convergent $\delta_{3,0}^{(3)}$, which is $\mathfrak{s}$-indecomposable for $a=3$ and associated with $\mu_0$. Tables for the remaining cases are in Section Appendix A8 of the arxiv version of this paper \cite{KST}.   
\end{proof}

\begin{table}[htbp]
\centering
{\fontsize{9.9}{12} \selectfont
\begin{tabular}{|c|c|c|c|}
\hline
 $\delta_{i,j}^{(k)}$ & Norm $N(\delta_{i,j}^{(k)})$ & $\mathfrak{s}$-indecomposable integer $\alpha$ & Unit $\varepsilon$ \\
 \hline
 $\delta_{2,0}^{(1)}$ &   $ - 2 a+3$ & $\lambda_{a-1,a-1}$  & $\rho^{-2}$  \\
 \hline
 $\delta_{2,1}^{(1)}$ &  $1$ & $1$  & $\rho(\rho-1)$ \\
 \hline
 $\delta_{3,j}^{(1)}$ &   &   &  \\
  $1\leq j \leq a-1$  & $-j^3+(a+1)j^2+(a+2)j+1$   & $\kappa_{a-j}$   &  $\rho^{-1}(\rho-1)$ \\
 \hline
  $\delta_{3,j}^{(1)}$ &   &   &  \\
  $a\leq j \leq a+1$  & $-j^3+(a+1)j^2+(a+2)j+1$   & \xmark   &   \\
 \hline
$\delta_{3,0}^{(2)}$ & $1$   & $1$  & $\rho-1 $ \\
 \hline
 $\delta_{3,j}^{(2)}$  &   &   &  \\
 $1\leq j \leq a-1$  & $-j^3+(a+1)j^2+(a+2)j+1$   & $\kappa_{a-j}$    & $\rho^{-1}(\rho-1)^2$   \\
 \hline
   $\delta_{3,j}^{(2)}$ &   &   &  \\
  $a\leq j \leq a+1$  & $-j^3+(a+1)j^2+(a+2)j+1$   & \xmark   &   \\
 \hline
$\delta_{3,0}^{(3)}$ & $1$  & $1$  & $(\rho-1)^2$ \\
 \hline
 $\delta_{3,1}^{(3)}$ & $1$  & $1$  & $\rho^{-1}(\rho-1)^2$ \\
 \hline
  $\delta_{3,j}^{(3)}$ &   &   &  \\
  $2\leq j \leq a$ & $-j^3+(a+4)j^2-(a+3)j+1$   & $\kappa_{a-j+1}$   & $\rho^{-2}(\rho-1)^3$   \\
 \hline
 $\delta_{3,j}^{(3)}$ &   &   &  \\
  $a+1\leq j \leq a+2$ & $-j^3+(a+4)j^2-(a+3)j+1$   & \xmark &    \\
 \hline
\end{tabular}
}
\caption{\textsc{Ennola's cubic fields}: Semiconvergents $\delta_{i,j}^{(k)}$ of the JPA expansion of $(1,\rho,\rho^2)$ where $1<\rho$ is the root of the polynomial $x^3+(a-1)x^2-ax-1$ where $a \geq 3$, their norm, the corresponding $\mathfrak{s}$-indecomposable integer $\alpha$ and the unit $\varepsilon$ such that $\delta_{i,j}^{(k)}=\alpha\varepsilon$.} \label{tab:ennola2}
\end{table}

\begin{table}[htbp]
\centering
{\fontsize{9.9}{12} \selectfont
\begin{tabular}{|c|c|c|c|}
\hline
 $\delta_{i,j}^{(k)}$ & Norm $N(\delta_{i,j}^{(k)})$ & $\mathfrak{s}$-indecomposable $\alpha$ & Unit $\varepsilon$ \\
\hline
$\delta_{2,j}^{(0)}$ &   &   &   \\
  $1\leq j\leq a-2$ & $ - j^3+(a-1)j^2+aj-1$   & $\lambda''_{a-j,a-1}$   & $-\rho''^{-2}(\rho''-1)^{-1}$  \\
\hline
$\delta_{3,1}^{(0)}$ & $ - 2 a+3$ &  $\lambda''_{a-1,a-1}$ & $\rho''^{-2}$  \\
\hline
$\delta_{3,2}^{(0)}$ & $ 2a^2-4a+1$ &  $\lambda''_{1,a-2}$ & $-\rho''^{-1}$  \\
\hline
$\delta_{3,3}^{(0)}$ & $ 6a^2-6a-11$ &  $\mu''_{a-3}$ & $-\rho''^{-1}$  \\
\hline
$\delta_{3,j}^{(0)}$  &   &   &   \\
  $4\leq j\leq a^2-2$ & $ - j^3+(a^2+1)j^2-(a^2+2a-2)j+1 $    & \xmark  &   \\
\hline
$\delta_{3,0}^{(2)}$ & $a^2-a-1$ & $\lambda''_{1,a-1}$  & $-\rho''^{-2}(\rho''-1)^{-1}$   \\  
\hline
$\delta_{2,j}^{(3)}$  &  $-j^3+(4a-4)j^2-(5a^2-11a+5)j $ &   &   \\
  $0\leq j \leq a-4$ &  $+2a^3-7a^2+7a-3$ & \xmark   &   \\
  \hline
$\delta_{2,a-3}^{(3)}$ & $2a+3$ & $\kappa''_{a-1}$  & $\rho''^{-2}$   \\
\hline
$\delta_{2,a-2}^{(3)}$ & $-1$ & $1$  & $-\rho''^{-1}$   \\
\hline
$\delta_{2,a-1}^{(3)}$ & $-1$ & $1$  & $-(\rho''-1)^{-1}$   \\
\hline
$\delta_{2,j}^{(3)}$  &  $-j^3+(4a-4)j^2-(5a^2-11a+5)j $ &   &   \\
  $a\leq j \leq 2a-3$ &  $+2a^3-7a^2+7a-3$ & $\lambda''_{2a-j-1,a-1}$   &  $-\rho''^{-3}(\rho''-1)^{-2}$  \\
  \hline
$\delta_{3,j}^{(3)}$ & $-j^3+(a^2-3a+1)j^2+(2a^3-4a^2+2)j$ &   &   \\
$0\leq j \leq a^2-a-2$ & $a^4-2a^3-a^2+2a+1$ & \xmark   &   \\
\hline
\end{tabular}
}
\caption{\textsc{Ennola's cubic fields}: Semiconvergents $\delta_{i,j}^{(k)}$ of the JPA expansion of $(1,-\rho'',\rho''^2)$ where $\rho''<-(a-1)$ is the root of the polynomial $x^3+(a-1)x^2-ax-1$ where $a \geq 4$, their norm, the corresponding $\mathfrak{s}$-indecomposable integer $\alpha$ and the unit $\varepsilon$ such that $\delta_{i,j}^{(k)}=\alpha\varepsilon$.} \label{tab:ennola1}
\end{table}

\section{Experiments on indecomposability of semiconvergents} \label{sec:exper}

For the simplest cubic and Ennola's cubic fields, we have found examples of Jacobi--Perron expansions whose convergents and semiconvergents are $\mathfrak{s}$-indecomposable. Moreover, as we will in the next section, some of these expansions can produce even more $\mathfrak{s}$-indecomposables. However, in other cases, some semiconvergents (or even convergents) are $\mathfrak{s}$-decomposable. In particular, the periodicity itself does not guarantee that we obtain only $\mathfrak{s}$-indecomposables. 

In this section, we provide some data for other totally real cubic fields. There are known some nice families of cubic irrationalities giving periodic Jacobi--Perron expansions. However, for example, those in Bernstein's papers \cite{Be1,Be2,Be3} are not suitable for our purposes since they do not belong to a totally real number field. For that reason, we decided to use a different approach.  
In this part, we discuss cases of fields of small discriminants obtained (together with a system of fundamental units) in \cite[pg. 149--153]{CS}. From the set from this paper, we have selected those fields which have units of all signatures and satisfy $\O_K=\Z[\rho]$ for one of the listed fundamental units.               The first condition follows from the fact that it is (slightly) easier to obtain totally positive indecomposables than the others. The second requirement is quite intuitive - it seems that units give more suitable expansions than other elements, and for the expansion of the vector of the form $(1,\rho,\rho^2)$, it is natural to consider $\mathfrak{s}$-indecomposables in $\Z[\rho]$.

For each field $K$ of this subset, we discuss several periodic expansions of vectors of the form $(1,\rho,\rho^2)$ where $\rho\in K$ is a suitable unit such that $0<\rho<2$. Again, this condition is intuitive; for example, in the case of the simplest cubic fields and Ennola's cubic fields, the expansions of such units give $\mathfrak{s}$-indecomposables rather than expansion corresponding to units $\rho>2$. Note that for some fields, we could not find any suitable expansion which would seem to be periodic, and, thus, these fields do not appear in our tables. However, we do not exclude the existence of such expansions.

Since we have units of all signatures in these fields, it suffices to determine only the representatives of totally positive indecomposables.
For that, we do not use the method described in Subsection~\ref{subsec:method}, but we proceed in the following way. For each of these fields, we find the set $S_T$ of all totally positive indecomposables up to some large trace $T$. This can be done by a computer program (for all these computations, we used Mathematica, and examples of used codes can be found at \url{https://sites.google.com/view/tinkovamagdalena/codes}) since the set of totally positive integers of a given trace is finite. Possibly, we do not get all representatives of $\mathfrak{s}$-indecomposables, but, in this task, it is not necessary to have all of them. Then, for several periodic JPA expansions, we check by a program if its (semi)convergents are associated with elements from the set $S_T$. If so, we have obtained an expansion giving only $\mathfrak{s}$-indecomposables. In the other case, we further check if the element which is not associated with elements from $S_T$, is indeed $\mathfrak{s}$-decomposable. In Table \ref{tab:exp1}, we show our results for fields of discriminant $\Delta_K$ up to $2101$; the results for $2177\leq\Delta_K\leq 6241$ are contained in Appendix A9 the arxiv version of this paper \cite{KST}.   

Moreover, these tables contain the following information: First of all, the discriminant $\Delta_K$ of a cubic field $K$. Note that we do not distinguish fields isomorphic by conjugation since, up to this conjugation, they have the same $\mathfrak{s}$-indecomposables. Then, there are several examples of polynomials such that some of their roots $\rho\in K$ satisfies that the vector $(1,\rho,\rho^2)$ has a periodic JPA expansion. Moreover, one can also find there approximate values of $\rho$ and of its conjugates $\rho'$ and $\rho''$, and the length of the preperiod and period of the expansion. The column titled as \textit{Conv.} summarizes information about convergents of the JPA expansion of $(1,\rho,\rho^2)$, and the column \textit{Semiconv.} about those semiconvergents which are not convergents. The symbol \cmark\, means that all the (semi)convergents of the expansion of $(1,\rho,\rho^2)$ are $\mathfrak{s}$-indecomposable in their signatures. The symbol \xmark\, stands for the situation that at least one of (semi)convergents is $\mathfrak{s}$-decomposable.

Moreover, the condition $\O_K=\Z[\rho]$ seems not to be true for many fields with units of all signatures listed in \cite{CS}. On the other hand, this is valid for some nice families of fields that do not possess all such units (this also includes Ennola's cubic fields which we discussed before in Section \ref{sec:jpaennola} or Subsection \ref{subsec:semiennola}). For that reason, we also study one more family of cubic fields generated by a root of the polynomial $x^3-(a+b)x^2+abx-1$ where $2\leq a\leq b-2$ \cite{Th}. For them, we know all totally positive indecomposables \cite[Proposition 5.3]{Ti2}. Moreover, for this family of fields, we have also constructed a program in Mathematica giving $\mathfrak{s}$-indecomposables which are not associated with totally positive indecomposables, This program is again based on traces (in this case, on traces after multiplication by elements of $\O_K^{\vee,\mathfrak{s}}$). The results for this family are in Table \ref{tab:expnotallsigns}.

To conclude, it seems that the behaviour of these expansions in general is rather unpredictable.

\begin{table}
\centering
\begin{tabular}{|c|c|c|c|c|c|c|c|c|}
\hline
$\Delta_K$ & Polynomial & $\rho$ & $\rho'$ & $\rho''$ & Preperiod & Period & Conv. & Semiconv.\\
\hline
\multirow{3}{*}{49} & $x^3-x^2-2x+1$ & $1.802$ & $0.445$ & $-1.247$ & $1$ & $2$ & \cmark & \cmark\\
\cline{2-9}
& $x^3-x^2-2x+1$ & $0.445$ & $1.802$  & $-1.247$ & $5$ & $2$ & \cmark & \cmark\\
\cline{2-9}
& $x^3+x^2-2x-1$ & $1.247$ & $-0.445$ & $-1.802$ & $2$ & $2$ & \cmark & \xmark\\
\hline
\multirow{3}{*}{81} & $x^3-3x+1$ & $1.532$ & $0.347$ & $-1.879$ & $2$ & $3$ & \cmark & \cmark\\
\cline{2-9}
& $x^3-3x+1$ & $0.347$ & $1.532$  & $-1.879$ & $3$ & $3$ & \cmark & \cmark\\
\cline{2-9}
& $x^3-3x-1$ & $1.879$ & $-1.532$ & $-0.347$ & $4$ & $6$ & \cmark & \xmark\\
\hline
\multirow{2}{*}{148} & $x^3+3x^2-x-1$ & $0.675$ & $-0.461$ & $-3.214$ & $2$ & $7$ & \cmark & \cmark\\
\cline{2-9}
& $x^3-3x^2-x+1$ & $0.461$ & $-0.675$  & $3.214$ & $7$ & $2$ & \cmark & \cmark\\
\hline
\multirow{2}{*}{169} & $x^3+x^2-4x+1$ & $1.377$ & $0.274$ & $-2.651$ & $2$ & $5$ & \cmark & \cmark\\
\cline{2-9}
& $x^3+x^2-4x+1$ & $0.274$ & $1.377$ & $-2.651$ & $3$ & $5$ & \cmark & \cmark\\
\hline
\multirow{2}{*}{321} & $x^3-x^2-4x+1$ & $0.239$ & $-1.7$ & $2.461$ & $5$ & $3$ & \cmark & \xmark\\
\cline{2-9}
& $x^3+x^2-4x-1$ & $1.7$ & $-0.239$  & $-2.461$ & $8$ & $3$ & \xmark & \xmark\\
\hline
\multirow{2}{*}{361} & $x^3+2x^2-5x+1$ & $1.285$ & $0.222$ & $-3.507$ & $2$ & $5$ & \cmark & \cmark\\
\cline{2-9}
& $x^3+2x^2-5x+1$ & $0.222$ & $1.285$ & $-3.507$ & $3$ & $5$ & \cmark & \cmark\\
\hline
\multirow{1}{*}{404} & $x^3+x^2-5x+1$ & $0.211$ & $1.655$ & $-2.866$ & $3$ & $3$ & \cmark & \cmark\\
\hline
\multirow{1}{*}{469} & $x^3+2x^2-4x-1$ & $1.391$ & $-0.227$ & $-3.164$ & $2$ & $5$ & \cmark & \xmark\\
\hline
\multirow{3}{*}{473} & $x^3-5x+1$ & $0.202$ & $2.128$ & $-2.33$ & $3$ & $3$ & \cmark & \cmark\\
\cline{2-9}
& $x^3-5x^2+1$ & $0.47$ & $4.959$ & $-0.429$ & $9$ & $2$ & \cmark & \cmark\\
\cline{2-9}
& $x^3+5x^2-1$ & $0.429$ & $-0.47$ & $-4.959$ & $2$ & $30$ & \xmark & \xmark\\
\hline
\multirow{1}{*}{564} & $x^3+7x^2+11x-1$ & $0.086$ & $-4.514$ & $-2.572$ & $2$ & $1$ & \xmark & \xmark\\
\hline
\multirow{1}{*}{621} & $x^3-6x^2+6x+1$ & $1.476$ & $-0.145$ & $4.669$ & $7$ & $9$ & \xmark & \xmark\\
\hline
\multirow{1}{*}{733} & $x^3+2x^2-6x+1$ & $0.178$ & $1.518$ & $-3.696$ & $3$ & $11$ & \xmark & \xmark\\
\hline
\multirow{2}{*}{785} & $x^3-2x^2-5x+1$ & $0.187$ & $-1.576$ & $3.388$ & $8$ & $14$ & \xmark & \xmark\\
\cline{2-9}
& $x^3+2x^2-5x-1$ & $1.576$ & $-0.187$ & $-3.388$ & $12$ & $4$ & \xmark & \xmark\\
\hline
\multirow{2}{*}{837} & $x^3-6x+1$ & $0.167$ & $-2.529$ & $2.361$ & $3$ & $3$ & \cmark & \cmark\\
\cline{2-9}
& $x^3-6x^2+1$ & $0.423$ & $-0.395$ & $5.972$ & $23$ & $10$ & \xmark & \xmark\\
\hline
\multirow{1}{*}{1300} & $x^3-7x^2+3x+1$ & $0.702$ & $-0.218$ & $6.516$ & $7$ & $5$ & \cmark & \cmark\\
\hline
\multirow{1}{*}{1345} & $x^3-7x+1$ & $0.143$ & $2.571$ & $-2.714$ & $3$ & $3$ & \cmark & \cmark\\
\hline
\multirow{2}{*}{1369} & $x^3+4x^2-7x+1$ & $1.187$ & $0.158$ & $-5.345$ & $2$ & $7$ & \cmark & \cmark\\
\cline{2-9}
& $x^3+4x^2-7x+1$ & $0.158$ & $1.187$ & $-5.345$ & $3$ & $7$ & \cmark & \cmark\\
\hline
\multirow{3}{*}{1425} & $x^3+7x^2+8x-1$ & $0.114$ & $-1.596$ & $-5.517$ & $2$ & $1$ & \xmark & \xmark\\
\cline{2-9}
& $x^3+8x^2-7x+1$ & $0.181$ & $0.626$ & $-8.808$ & $3$ & $13$ & \xmark & \cmark\\
\cline{2-9}
& $x^3+8x^2-7x+1$ & $0.626$ & $0.181$ & $-8.808$ & $6$ & $2$ & \xmark & \xmark\\
\hline
\multirow{2}{*}{1524} & $x^3-x^2-7x+1$ & $0.14$ & $3.133$ & $-2.273$ & $5$ & $3$ & \cmark & \xmark\\
\cline{2-9}
& $x^3-7x^2-x+1$ & $0.319$ & $7.121$ & $-0.44$ & $9$ & $6$ & \cmark & \cmark\\
\hline
\multirow{2}{*}{1573} & $x^3+8x^2+14x-1$ & $0.069$ & $-2.722$ & $-5.346$ & $2$ & $1$ & \xmark & \xmark\\
\cline{2-9}
& $x^3+14x^2-8x+1$ & $0.187$ & $0.367$ & $-14.554$ & $5$ & $9$ & \xmark & \xmark\\
\hline
\multirow{3}{*}{2101} & $x^3-4x^2-6x+1$ & $0.152$ & $-1.283$ & $5.131$ & $8$ & $11$ & \xmark & \xmark\\
\cline{2-9}
& $x^3+4x^2-6x-1$ & $1.283$ & $-0.152$ & $-5.131$ & $19$ & $7$ & \xmark & \xmark\\
\cline{2-9}
& $x^3-6x^2-4x+1$ & $0.195$ & $-0.779$ & $6.584$ & $5$ & $8$ & \xmark & \xmark\\
\hline
\end{tabular}
\caption{Indecomposability of (semi)convergents of JPA expansions of vectors of the form $(1,\rho,\rho^2)$ where $\rho$ is such that $\O_K=\Z[\rho]$; $K$ is an example of a totally real number field with the discriminant $\Delta_K\leq 2101$ and units of all signatures
} \label{tab:exp1}
\end{table}

\begin{table}
\centering
{\small
\begin{tabular}{|c|c|c|c|c|c|c|c|c|}
\hline
$\Delta_K$ & Polynomial & $\rho$ & $\rho'$ & $\rho''$ & Preperiod & Period & Conv. & Semiconv.\\
\hline
\multirow{4}{*}{229} & $x^3-4x+1$ & $1.861$ & $0.254$ & $-2.115$ & $6$ & $3$ & \cmark & \cmark\\
\cline{2-9}
& $x^3-4x+1$ & $0.254$ & $1.861$ & $-2.115$ & $3$ & $3$ & \cmark & \cmark\\
\cline{2-9}
& $x^3-8x^2+6x-1$ & $0.573$ & $0.243$ & $7.184$ & $5$ & $3$ & \xmark & \xmark\\
\cline{2-9}
& $x^3-8x^2+6x-1$ & $0.243$ & $0.573$ & $7.184$ & $2$ & $2$ & \cmark & \cmark\\
\hline
\multirow{4}{*}{761} & $x^3-7x^2+10x-1$ & $1.828$ & $0.108$ & $5.064$ & $18$ & $12$ & \xmark & \xmark\\
\cline{2-9}
& $x^3+x^2-6x+1$ & $1.892$ & $0.172$ & $-3.064$ & $6$ & $3$ & \cmark & \cmark\\
\cline{2-9}
& $x^3+x^2-6x+1$ & $0.172$ & $1.892$ & $-3.064$ & $3$ & $3$ & \cmark & \cmark\\
\cline{2-9}
& $x^3-10x^2+7x-1$ & $0.197$ & $0.547$ & $9.255$ & $2$ & $2$ & \cmark & \cmark\\
\hline
\multirow{4}{*}{985} & $x^3-x^2-6x+1$ & $0.163$ & $2.931$ & $-2.094$ & $5$ & $3$ & \cmark & \xmark\\
\cline{2-9}
& $x^3-15x^2+8x-1$ & $0.352$ & $0.196$ & $14.451$ & $12$ & $21$ & \xmark & \xmark\\
\cline{2-9}
& $x^3-15x^2+8x-1$ & $0.196$ & $0.352$ & $14.451$ & $2$ & $2$ & \cmark & \cmark\\
\cline{2-9}
& $x^3-6x^2-x+1$ & $0.341$ & $-0.478$ & $6.136$ & $3$ & $3$ & \cmark & \cmark\\
\hline
\multirow{4}{*}{1957} & $x^3-8x^2+12x-1$ & $0.088$ & $1.871$ & $6.041$ & $2$ & $16$ & \xmark & \xmark\\
\cline{2-9}
& $x^3+2x^2-8x+1$ & $1.912$ & $0.129$ & $-4.041$ & $6$ & $3$ & \cmark & \cmark\\
\cline{2-9}
& $x^3+2x^2-8x+1$ & $0.129$ & $1.912$ & $-4.041$ & $3$ & $3$ & \cmark & \cmark\\
\cline{2-9}
& $x^3-12x^2+8x-1$ & $0.166$ & $0.535$ & $11.3$ & $2$ & $2$ & \cmark & \cmark\\
\hline
\multirow{3}{*}{2597} & $x^3-24x^2+10x-1$ & $0.164$ & $0.258$ & $23.578$ & $2$ & $2$ & \cmark & \cmark\\
\cline{2-9}
& $x^3-8x^2-2x+1$ & $0.253$ & $-0.481$ & $8.228$ & $3$ & $3$ & \cmark & \cmark\\
\cline{2-9}
& $x^3+8x^2-2x-1$ & $0.481$ & $-0.253$ & $-8.228$ & $2$ & $2$ & \cmark & \cmark\\
\hline
\multirow{4}{*}{5521} & $x^3-3x^2-10x+1$ & $0.097$ & $4.971$ & $-2.068$ & $3$ & $13$ & \xmark & \xmark\\
\cline{2-9}
& $x^3-35x^2+12x-1$ & $0.204$ & $0.141$ & $34.655$ & $10$ & $13$ & \xmark & \xmark\\
\cline{2-9}
& $x^3-35x^2+12x-1$ & $0.141$ & $0.204$ & $34.655$ & $2$ & $2$ & \cmark & \cmark\\
\cline{2-9}
& $x^3-10x^2-3x+1$ & $0.201$ & $-0.483$ & $10.282$ & $3$ & $3$ & \cmark & \cmark\\
\hline
\multirow{3}{*}{6809} & $x^3+x^2-12x+1$ & $0.084$ & $2.951$ & $-4.035$ & $3$ & $3$ & \cmark & \cmark\\
\cline{2-9}
& $x^3-21x^2+10x-1$ & $0.142$ & $0.343$ & $20.515$ & $2$ & $2$ & \cmark & \cmark\\
\cline{2-9}
& $x^3-12x^2+x+1$ & $0.339$ & $-0.248$ & $11.909$ & $3$ & $3$ & \cmark & \cmark\\
\hline
\multirow{3}{*}{7249} & $x^3-x^2-12x+1$ & $0.083$ & $3.964$ & $-3.047$ & $5$ & $3$ & \cmark & \xmark\\
\cline{2-9}
& $x^3-28x^2+11x-1$ & $0.142$ & $0.255$ & $27.603$ & $2$ & $2$ & \cmark & \cmark\\
\cline{2-9}
& $x^3-12x^2-x+1$ & $0.252$ & $-0.328$ & $12.076$ & $3$ & $3$ & \cmark & \cmark\\
\hline
\multirow{3}{*}{8069} & $x^3+4x^2-12x+1$ & $1.935$ & $0.086$ & $-6.021$ & $6$ & $3$ & \cmark & \cmark\\
\cline{2-9}
& $x^3+4x^2-12x+1$ & $0.086$ & $1.935$ & $-6.021$ & $3$ & $3$ & \cmark & \cmark\\
\cline{2-9}
& $x^3-16x^2+10x-1$ & $0.125$ & $0.522$ & $15.353$ & $2$ & $2$ & \cmark & \cmark\\
\hline
\multirow{3}{*}{10309} & $x^3-4x^2-12x+1$ & $0.081$ & $5.979$ & $-2.06$ & $3$ & $7$ & \cmark & \cmark\\
\cline{2-9}
& $x^3-48x^2+14x-1$ & $0.124$ & $0.169$ & $47.707$ & $2$ & $2$ & \cmark & \cmark\\
\cline{2-9}
& $x^3-12x^2-4x+1$ & $0.167$ & $-0.485$ & $12.318$ & $3$ & $3$ & \cmark & \cmark\\
\hline
\multirow{3}{*}{13801} & $x^3+2x^2-15x+1$ & $0.067$ & $2.958$ & $-5.025$ & $3$ & $3$ & \cmark & \cmark\\
\cline{2-9}
& $x^3-24x^2+11x-1$ & $0.125$ & $0.341$ & $23.534$ & $2$ & $2$ & \cmark & \cmark\\
\cline{2-9}
& $x^3-15x^2+2x+1$ & $0.338$ & $-0.199$ & $14.861$ & $3$ & $3$ & \cmark & \cmark\\
\hline
\multirow{5}{*}{14089} & $x^3-11x^2+18x-1$ & $0.058$ & $1.927$ & $9.016$ & $2$ & $16$ & \xmark & \xmark\\
\cline{2-9}
& $x^3+5x^2-14x+1$ & $1.942$ & $0.073$ & $-7.016$ & $6$ & $3$ & \cmark & \cmark\\
\cline{2-9}
& $x^3+5x^2-14x+1$ & $0.073$ & $1.942$ & $-7.016$ & $3$ & $3$ & \cmark & \cmark\\
\cline{2-9}
& $x^3-18x^2+11x-1$ & $0.519$ & $0.111$ & $17.37$ & $5$ & $9$ & \xmark & \xmark\\
\cline{2-9}
& $x^3-18x^2+11x-1$ & $0.111$ & $0.519$ & $17.37$ & $2$ & $2$ & \cmark & \cmark\\
\hline
\end{tabular}
}
\caption{Indecomposability of (semi)convergents of JPA expansions of vectors of the form $(1,\rho,\rho^2)$ where $\rho$ is such that $\O_K=\Z[\rho]$; $K$ is an example of a totally real number field which does not have units of all signatures} \label{tab:expnotallsigns}
\end{table}

\section{More general semiconvergents} \label{sec:gensemi}

As we have seen in the previous section, for some expansions in our fields, convergents and semiconvergents can be $\mathfrak{s}$-indecomposable. However, even in the most favorable cases, we have obtained only a few elements from the set of representatives of $\mathfrak{s}$-indecomposables. For example, in the case of the simplest cubic fields, we have got some lines from the triangle of totally positive indecomposables $-v-(v(a+2)+1+W)\rho+(v+1)\rho^2$ where $0\leq v\leq a$ and $0\leq W\leq a-v$. Moreover, only one of these lines does not lie on the border of the triangle. In particular, they do not cover the whole triangle. And the exceptional indecomposable integer $1+\rho+\rho^2$ does not appear in the expansion at all. 

One way how to deal with this situation is to consider more general elements derived from convergents of expansions. Especially, semiconvergents are, in fact, some concrete linear combinations of convergents, and this could explain why we obtain only lines in the triangle. There is no obstacle to considering some other combinations to get other shapes like a triangle. Therefore, in this section, we will study in detail the expansions determined in the previous sections and try to find some other linear combinations giving $\mathfrak{s}$-indecomposables in our fields.

\subsection{The simplest cubic fields}

We will again start with the case of the simplest cubic fields, for which the convergents and semiconvergents of the expansion of $(1,-\rho',\rho'^2)$ are $\mathfrak{s}$-indecomposable. Recall that here, $-2<\rho'<-1$ is a root of the polynomial $x^3-ax^2-(a+3)x-1$ where $a\geq -1$. As we will see, in this case, it is possible to obtain all representatives of $\mathfrak{s}$-indecomposables, including the exceptional one. For that, we will use the following idea.

Let us consider some semiconvergent $\beta_2^{(k)}-i\beta_1^{(k)}$. From the construction, we know that its value is small, and the same is true for the convergent $\beta_3^{(k)}$. Therefore, if $\beta_2^{(k)}-i\beta_1^{(k)}-\beta_3^{(k)}$ is positive, it has even smaller value than $\beta_2^{(k)}-i\beta_1^{(k)}$. Thus, it could be a good candidate for an $\mathfrak{s}$-indecomposable integer. Moreover, if we want to obtain, e.g., a triangle of elements, these elements should have the same signature. Therefore, we can consider those elements of the form $\beta_2^{(k)}-i\beta_1^{(k)}-j\beta_3^{(k)}$ (and analogous elements for $\beta_3^{(k)}-i\beta_1^{(k)}$) which have the same signature as $\beta_2^{(k)}-i\beta_1^{(k)}$. Adding one more condition, we can obtain all representatives of $\mathfrak{s}$-indecomposables in the simplest cubic fields in the following way.  

\begin{proposition}
Let $\langle\bm{\beta}^{(k)}\rangle$ be the Jacobi--Perron expansion of the vector $(1,-\rho',\rho'^2)$ where $-2<\rho'<-1$ is a root of the polynomial $x^3-ax^2-(a+3)x-1$, and let $a\geq 15$. Suppose that the element $\beta_2^{(k)}-i\beta_1^{(k)}-j\beta_3^{(k)}$ satisfies the following conditions:
\begin{enumerate}
\item $1\leq i\leq \left\lfloor\frac{\beta_2^{(k)}}{\beta_1^{(k)}}\right\rfloor-1$,
\item $j\in \N_0$ is such that $\beta_2^{(k)}-i\beta_1^{(k)}-j\beta_3^{(k)}$ has the same signature as $\beta_2^{(k)}-i\beta_1^{(k)}$,
\item for all $i+1\leq l\leq \left\lfloor\frac{\beta_2^{(k)}}{\beta_1^{(k)}}\right\rfloor-1$, the element $(l-i)\beta_1^{(k)}-\beta_3^{(k)}$ does not have the same signature as $\beta_2^{(k)}-i\beta_1^{(k)}$.
\end{enumerate} 
Then $\beta_2^{(k)}-i\beta_1^{(k)}-j\beta_3^{(k)}$ is $\mathfrak{s}$-indecomposable in its signature. Similarly, suppose that the element $\beta_3^{(k)}-i\beta_1^{(k)}-j\beta_2^{(k)}$ satisfies the following conditions:
\begin{enumerate}
\item $1\leq i\leq \left\lfloor\frac{\beta_3^{(k)}}{\beta_1^{(k)}}\right\rfloor-1$,
\item $j\in \N_0$ is such that $\beta_3^{(k)}-i\beta_1^{(k)}-j\beta_2^{(k)}$ has the same signature as $\beta_3^{(k)}-i\beta_1^{(k)}$,
\item for all $i+1\leq l\leq \left\lfloor\frac{\beta_3^{(k)}}{\beta_1^{(k)}}\right\rfloor-1$, the element $(l-i)\beta_1^{(k)}-\beta_2^{(k)}$ does not have the same signature as $\beta_3^{(k)}-i\beta_1^{(k)}$.
\end{enumerate} 
Then $\beta_3^{(k)}-i\beta_1^{(k)}-j\beta_2^{(k)}$ is $\mathfrak{s}-$indecomposable in its signature $\mathfrak{s}$.

Moreover, this way, we obtain representatives of all indecomposable integers in $\Z[\rho]$ up to multiplication by units and conjugation.
\end{proposition}

\begin{proof}
For $k=0$, we do not get any additional elements as $\left\lfloor\frac{\beta_2^{(0)}}{\beta_1^{(0)}}\right\rfloor=\left\lfloor\frac{\beta_3^{(0)}}{\beta_1^{(0)}}\right\rfloor=1$.

For $k=1$, the element $\delta_{2,1}^{(0)}=-1+\rho'^2-(-1-\rho')<1$ since $-1+\rho'^2<1$, therefore $\beta_2^{(1)}-\beta_1^{(1)}-\beta_3^{(1)}=\delta_{2,1}^{(0)}-1<0$. Thus, we do not obtain more elements. Likewise, the elements $\beta_3^{(1)}-i\beta_1^{(1)}$ are totally positive for all $1\leq i\leq \left\lfloor\frac{\beta_3^{(1)}}{\beta_1^{(1)}}\right\rfloor-1=a$, and it can be proved that this is not true for any element $\beta_3^{(1)}-i\beta_1^{(1)}-\beta_2^{(1)}$.

On the other hand, for $k=2$, we indeed obtain additional elements, namely $\beta_3^{(2)}-i\beta_1^{(2)}-j\beta_2^{(2)}$ where $1\leq i\leq a$ and $1\leq j\leq \left\lfloor\frac{a-(i-1)}{2}\right\rfloor$. These elements are unit multiples of indecomposable integers $\theta'_{a+1-i-2j,j}$.  

For $k=3$, our integers parts are $0$ and $1$. Thus, we do not obtain any additional elements.

In the next step, $\left\lfloor\frac{\beta_2^{(4)}}{\beta_1^{(4)}}\right\rfloor=0$ and $\left\lfloor\frac{\beta_3^{(4)}}{\beta_1^{(4)}}\right\rfloor=A_0$. Here, we get the element $\beta_3^{(4)}-\beta_1^{(4)}-(a-1)\beta_2^{(4)}$ which is a unit, and elements $\beta_3^{(4)}-i\beta_1^{(4)}-j\beta_2^{(4)}$ where $1\leq i\leq A_0-1$ and $1\leq j\leq a-2i$. These integers are unit multiples of $\theta'_{a-2i-j,i+j+1}$. 

For $k=5$, we obtain one additional elements only if $a$ is even. It is equal to $\beta_3^{(5)}-\beta_1^{(5)}-\beta_2^{(5)}$ and is associated with the indecomposable integer $\theta'_{A_0-2,a-(A_0-2)}$.
The sixth iteration does not produce more elements. 

The same is true for the first part of the seventh iteration, where the obtained elements do not have the required signature. On the contrary, in the second part, we obtain some elements. If $a=2A0$ is even, we get elements $\beta_3^{(7)}-i\beta_1^{(7)}-\beta_2^{(7)}$ where $1\leq i\leq 5$. In this case, $(A_0-1)\beta_1^{(7)}-\beta_2^{(7)}$ has the same signature as $\beta_3^{(7)}-i\beta_1^{(7)}$ which excludes elements with $i=1,2,3$. If $i=5$, then our elements is a unit. If $i=4$, our element is associated to the indecomposable integer $1+\rho'+\rho'^2$. If $a=2A0+1$ is odd, we obtain elements $\beta_3^{(7)}-i\beta_1^{(7)}-\beta_2^{(7)}$ where $1\leq i\leq 6$. As before, $(A_0-1)\beta_1^{(7)}-\beta_2^{(7)}$ has the same signature as $\beta_3^{(7)}-i\beta_1^{(7)}$; thus we can exclude $i=1,2,3,4$. If $i=6$, our element is again a unit. For $i=5$, our element is associated to the exceptional indecomposable integer $1+\rho'+\rho'^2$.       

And finally, for $k=8$, all the obtained elements do not have the required signature.

Now we will discuss the final part of the statement, i.e., if we obtain all representatives of $\mathfrak{s}$-indecomposable integers. In the seventh iteration, we get a unit and an element associated with the exceptional indecomposable integer $1+\rho'+\rho'^2$. On the other hand, for $k=2$, we obtain unit multiples of elements $\theta'_{a+1-i-2j,j}$ where $1\leq i\leq a$ and $0\leq j\leq \left\lfloor\frac{a-(i-1)}{2}\right\rfloor$. These elements cover more than one third of the triangle of indecomposables, and in particular all elements $-(a-v-W)-((a+2)(a-v-W)+1+v)\rho'+(a-v-W+1)\rho'^2$, where $0\leq v\leq A$ and $v\leq W\leq a-2v-1$ with $a=3A+a_0$, $a_0\in\{0,1,2\}$. Moreover, they also cover the element $\theta'_{A,A}$ if $3|a$. Note that in \cite[Subsection 5.1]{KT}, we show that if the triangle of indecomposable integers contains $\alpha$, then it also contains $\alpha'\varepsilon_1$ and $\alpha''\varepsilon_2$ where $\varepsilon_1$ and $\varepsilon_2$ are two concrete units. That means that in some sense, it is enough to consider only one third of the triangle. Therefore, up to multiplication and conjugation, we get all indecomposables of the triangle.    
\end{proof}

\subsection{Ennola's cubic fields}

In this section, we will focus on more general semiconvergents, i.e., on the other linear combinations of convergents of the JPA expansion in the Ennola's cubic fields.  
We will not discuss all possible combinations but instead, 
we will look at semiconvergents of form $j \beta^{(i)}_1+ k \beta^{(i)}_2+ l \beta^{(i)}_3 $ where at least one of the coefficients $j,k,l$ is fixed. 

First of all, let is introduce the following notation:
$$S=\{\lambda_{v,w}=-v-(a(v-1)+w)\rho+(a(v-1)+w+1)\rho^2 \text{ where } 1\leq v \leq a-1, \max\{1,v-1\}\leq w \leq a-1\}.$$
After multiplication by some unit and index shifting, the above set can be expressed as
$$\tilde{S}=\{\tilde{\lambda}_{v,u}=1+v-u+au+(a-u+au)\psi-(u+1)\psi^2\text{ where }1\leq v \leq a-3\text{, } 0 \leq u \leq v\text{ and }(v,u)\neq (1,0)\}$$
By the symbol $\langle\gamma_1,\ldots,\gamma_r\rangle$ where $\gamma_i\in K$ for all $i=1,\ldots,r$ and $r\in\N$, we will mean the lattice generated by the elements $\gamma_1,\ldots,\gamma_r$.

\begin{theorem}
We have
\begin{enumerate} 
\item Let $\tau$ be one of the roots of the polynomial polynomial $x^3+(a-1)x^2-ax-1 $ where $a\geq 3$. Let $ \langle \bm{\beta} ^{(k)}\rangle _{k \geq 0 }$ be the  periodic JPA expansion of the vector $(1, |\tau|,\tau^2)$. 
Then  whenever for some iteration $k$, there exists a unit $\varepsilon$ such that 
\[ L=\langle \alpha -\beta|\alpha, \beta \in \sigma(S) \rangle = \langle \beta _1^{(k)}\varepsilon, \beta _2^{(k)}\varepsilon  \rangle \]
and  
\[ |N(\beta_3^{(k)})|\leq  a^2-5a, \]
then  
\[\sigma(S)\subset \beta_3^{(k)} \varepsilon + L,\]
where $\sigma$ is the embedding corresponding to the root $\tau$.
\item Let $\tau$ be one of the roots of the polynomial $x^3-(a-1)x^2-ax-1 $ where $a\geq 5$. Let $ \langle \bm{\beta} ^{(k)}\rangle _{k \geq 0 }$ be the  periodic JPA expansion of the vector $(1, |\tau|,\tau^2)$. 
Then  whenever for some iteration $k$, there exists a unit $\varepsilon$ such that 
\[ L=\langle \alpha -\beta|\alpha, \beta \in \sigma(\tilde{S}) \rangle = \langle \beta _1^{(k)}\varepsilon, \beta _2^{(k)}\varepsilon  \rangle \]
and  
\[ |N(\beta_3^{(k)})|\leq  a^2-5a, \]
then  
\[\sigma(\tilde{S})\subset \beta_3^{(k)} \varepsilon + L,\]
where $\sigma$ is the embedding corresponding to the root $\tau$.
\end{enumerate}
\end{theorem}

\begin{proof}
To prove this theorem, it is enough to check the statement for every expansion. For shortening, we will show just two interesting cases in the expansions of two of the roots of polynomial $x^3-(a-1)x^2-ax-1$.

In the tenth iteration of the expansion of $(1,-\psi',\psi'^2)$, the norm of $\beta_3^{(10)}$ is $-a^2+5a-5$ and 
$L=\langle \alpha -\beta|\alpha, \beta \in (S)' \rangle$. Let $\varepsilon=-a+5+7\psi'+(a+4)\psi'^2$.
Then
$$\beta_3^{(10)}-h\beta_1^{(10)}-j\beta_2^{(10)}=\tilde{\lambda}_{1+h+j,1+h}'\varepsilon.$$
Thus
\[(S)'\subset \beta_3^{(10)} \varepsilon^{-1} + L.\]

Second example is from the expansion $(1,-\psi'',\psi''^2)$ where in the fourth iteration, $N(\beta_3^{(4)})=2a^2-12a+17$. 
$L=\langle \alpha -\beta|\alpha, \beta \in (S)'' \rangle$, where $\varepsilon=\psi''+(a-1)\psi''^2$.
We have
\[ \beta_3^{(4)}-h\beta_1^{(4)}-j\beta_2^{(4)}=\tilde{\lambda}''_{a-4-h-2j,a-3-h-j} \varepsilon .\]
So again
\[(S)''\subset \beta_3^{(4)} \varepsilon^{-1} + L.\qedhere\]

\end{proof}

However, in this theorem, the condition $|N(\beta_3^{(k)})|\leq  a^2-5a $ is necessary. There is a counterexample in the JPA expansion of $(1, \rho', \rho'^2)$, where in the fifth iteration, the elements $\beta _1^{(5)}, \beta _2^{(5)}$ generate the basis of the lattice $L$ of the indecomposable elements $\lambda_{u,v}'$ and the norm of $\beta _3^{(5)}$ is $a^2-a-1$, but $\beta_2^{(k)}-h\beta_1^{(k)}-j\beta_3^{(k)}$ do not give all elements $\lambda_{u,v}'$.

We can look on the opposite implications of the theorem. There is easy to verify that if \[(\beta_3^{(k)}-h\beta_1^{(k)}-j\beta_2^{(k)}) \varepsilon\] cover $\sigma(S)$ for some unit $\varepsilon$ and the corresponding embedding $\sigma$, then  
\[ L=\langle \alpha -\beta|\alpha, \beta \in \sigma(S) \rangle \subseteq \langle \beta _1^{(k)}\varepsilon, \beta _2^{(k)}\varepsilon  \rangle. \]

This theorem is somewhat surprising, and the natural question is whether this is satisfied more generally.

\section{Other multidimensional continued fraction algorithms} \label{sec:othermulti}

As we have mentioned in the introduction, there exist many algorithms that generate multidimensional continued fractions. Our choice of the Jacobi--Perron algorithm was mostly influenced by its simplicity, amount of results, or by the existence of the Hasse--Bernstein unit. As we have seen, sometimes, it produces nice results, but so far, it has not given us the full answer to our question. Therefore, we have also looked at several other algorithms, and in this section, we summarize some of our experiments.

Brun's algorithm is one of the easiest ones \cite{Brun}. Let us show it for vectors of dimension 3. Let $\big(\beta_1^{(0)},\beta_2^{(0)},\beta_3^{(0)}\big)$ be the initial vector such that $\beta_i^{(0)}\geq 0$ for all $i$. Then in each iteration $k$, do the following:
\begin{enumerate}
\item $\big(\beta_1^{(k)},\beta_2^{(k)},\beta_3^{(k)}\big):=\big(\beta_{\pi(1)}^{(k)},\beta_{\pi(2)}^{(k)},\beta_{\pi(3)}^{(k)}\big)$ where $\pi$ is a permutation\ on set $\{1,2,3\}$ such that $\beta_{\pi(1)}^{(k)}\leq \beta_{\pi(2)}^{(k)}\leq \beta_{\pi(3)}^{(k)}$.
\item $\big(\beta_1^{(k+1)},\beta_2^{(k+1)},\beta_3^{(k+1)}\big):=\big(\beta_1^{(k)},\beta_2^{(k)},\beta_3^{k)}-\beta_2^{(k)}\big)$.
\end{enumerate}
We define periodicity, prepreriod and period in the same way as for Jacobi--Perron algorithm.

From our data, it seems that Brun's algorithm is not better than the Jacobi--Perron algorithm (rather worse). For example, when we take $\rho$ as a root of the polynomial $x^3-4x^2-7x-1$, i.e., the simplest cubic field with $a=4$, and consider the initial vector of the form $(1,|\rho|,\rho^2)$, it does not seem that the expansion is periodic for any root $\rho$. Note that the periodicity is crucial for us since up to multiplication by units, there are only finitely many $\mathfrak{s}$-indecomposables.

On the other hand, in Ennola's cubic fields, we can obtain a periodic expansion in the following way. Note that coordinates in each iteration are already ordered according to their sizes.

\begin{proposition}
Let $\rho>1$ be a root of the polynomial $x^3+(a-1)x^2-ax-1$ where $a\geq 3$. Then Brun's expansion of the vector $(1,\rho,\rho^2)$ is periodic with the period
\begin{align*}
\bm{\beta}^{(0)}&=(1,\rho,\rho^2),\\
\bm{\beta}^{(1)}&=(-\rho+\rho^2,1,\rho),\\
\bm{\beta}^{(2)}&=(-1+\rho,-\rho+\rho^2,1),\\
\bm{\beta}^{(k+2)}&=(-1+\rho,-\rho+\rho^2,1+k\rho-k\rho^2)\text{ where } 1\leq k \leq a-1.
\end{align*}
Moreover, for all $k\in\N_0$ and $1\leq i\leq 3$, the element $\beta_i^{(k)}$ is $\mathfrak{s}$-indecomposable in its signature.
\end{proposition}  

\begin{proof}
In this case, Brun's expansion $(1,\rho,\rho^2)$ is simple, so it is trivial to verify its form. Regarding the second part of the statement, almost all elements $\beta_i^{(k)}$ are units unless $\beta_3^{(k+2)}=1+k\rho-k\rho^2$ with $1\leq k \leq a-1$. In that case, we have
\[
1+k\rho-k\rho^2=\rho^{-1}\lambda_{k,a-1}
\]
where the elements $\lambda_{k,a-1}$ are specific integers from the set of $\mathfrak{s}$-indecomposables in Ennola's cubic fields.
\end{proof}
However, the same is not valid for the other conjugates of $\rho$, for which the analogous expansions seem to be not periodic. Moreover, even in this case, we obtain only a short line of $\mathfrak{s}$-indecomposables from the triangle, so at the moment, there is no convincing reason to prefer Brun's algorithm to the Jacobi--Perron algorithm.

Our next choice was the algorithm of Tamura and Yasutomi \cite{TY} which is derived from the Jacobi--Perron algorithm and is believed to be periodic for real cubic field. We have tried it for several initial vectors from Ennola's cubic fields, and it does not seem to be as good as the Jacobi--Perron algorithm. Moreover, we started to conduct experiments with new algorithms defined by Karpenkov \cite{Kar} which always give periodic expansions in totally real fields, but at this moment, our results are preliminary. Regarding periodicity, another promising algorithm was proposed by Garrity \cite{Ga} and further developed by Dasaratha, Flapan, Garrity, Lee, Mihaila, Neumann-Chun, Peluse, and Stoffregen  \cite{DF+} and would be worth trying in the future investigation of our problem.

Furthermore, \v Rada, Starosta and Kala  \cite{RSK} also considered a different approach. Instead of fixing some known algorithm, they ``choose an algorithm" in each iteration. In particular, every step in most of the algorithms above corresponds to the multiplication by a matrix. For example, in the Jacobi--Perron algorithm, this matrix is
\[
M^{(k)}=
\left(
\begin{matrix}
-a_1 & 1 & 0\\
-a_2 & 0 & 1\\
1 & 0 & 0
\end{matrix}
\right)
\]
where $a_i=\left\lfloor\frac{\beta_{i+1}^{(k)}}{\beta_1^{(k)}}\right\rfloor$ for $i=1,2$. It means that $\bm{\beta}^{(k+1)}=M^{(k)}(\bm{\beta}^{(k)})^{T}$. The product of these matrices over the period is the matrix of the period; let us simply denote it by $M$. If $\bm{\beta}=\bm{\beta}^{(0)}$ has a purely periodic expansion with a matrix $M$, then necessarily 
\begin{equation} \label{eq:eigen}
\varepsilon\bm{\beta}=M\bm{\beta}
\end{equation}
for some unit $\varepsilon$, i.e., $\beta$ is an eigenvector of $M$. In our follow-up research, we firstly fix some suitable chosen matrix $M$ satisfying (\ref{eq:eigen}) for some unit $\varepsilon$. Then, we try to decompose $M$ into a product of matrices of some specific forms to get
\begin{enumerate}
\item only $\mathfrak{s}$-indecomposables,
\item all $\mathfrak{s}$-indecomposables up to multiplication by units.
\end{enumerate}
So far, we have obtained promising results for the simplest cubic fields which will appear in a forthcoming article.

\section*{Acknowledgments}

All authors were supported by grant 21-00420M from Czech Science Foundation (GA\v{C}R). E. S. was further supported by project GA UK 742120 (from Charles University) and by SVV-2020-260589. M. T. was further supported by grant 22-11563O from Czech Science Foundation (GA\v{C}R).

We also thank the anonymous referee for helpful suggestions and corrections.

\section*{Statements and Declarations}

\subsection*{Conflict of interest statement}
Not applicable.

\subsection*{Code availability statement}
All codes are available at \\https://sites.google.com/view/tinkovamagdalena/codes.

\subsection*{Data availability statement}
Not applicable.

\addresseshere

\renewcommand{\thesection}{A\arabic{section}}

\setcounter{section}{4}
\section{Preliminaries on the Jacobi--Perron algorithm}

In this part, we show proof of a well-known relation between inhomogeneous and homogenous versions of the Jacobi--Perron algorithm.

\begin{proposition}\label{lm:iJPA-JPAarxiv}
Let $\langle\bm{\alpha} ^{(k)}\rangle$ be the iJPA expansion of the vector $(\theta_1, \theta_2, \dots, \theta_{n-1})\in\R^{n-1}$. Then the JPA expansion of $\bm{\beta}^{(0)}= (1, \theta_1, \theta_2, \dots, \theta_{n-1})$ is equal to
\[ \bm{\beta}^{(k)}=\left(\delta_{k}, \alpha_1^{(k)} \delta_{k}, \dots, \alpha_{n-2}^{(k)} \delta_{k},\alpha_{n-1}^{(k)} \delta_{k}\right),\]
where $\delta_k=\frac{1}{\alpha_{n-1}^{(1)} \cdot \dots \cdot \alpha_{n-1}^{(k)}}$ for $k \geq 0$, and $\delta_{0}=1$. Moreover for every $i,k$, we have
\[  \Bigg\lfloor \frac{\beta_{i+1}^{(k)}}{\beta_1^{(k)}} \Bigg\rfloor = \big\lfloor \alpha_i^{(k)}\big\rfloor.\]
\end{proposition}

\begin{proof}
We will prove this by induction on $k$. 
For $i=0$, we have 
$\beta_1^{(0)}= 1 = \delta_{0}$,
and for $1<i\leq n$, we obtain 
$$\beta_i^{(0)}=\theta_{i-1}=\alpha_{i-1}^{(0)}= \alpha_{i-1}^{(0)} \delta_{0}.$$ 
Therefore,
$$\bm{\beta}^{(0)}= \left(\delta_{0}, \alpha_1^{(0)} \delta_{0}, \dots, \alpha_{n-1}^{(0)} \delta_{0}\right).$$  

Let us now assume that the formula holds for $k$; we will prove it for $k+1$. We see that 
\[
\beta_{1}^{(k+1)} = \beta_2^{(k)}-\left\lfloor \frac{\beta_2^{(k)}}{\beta_1^{(k)}} \right\rfloor \beta_1^{(k)} = \alpha_1^{(k)} \delta_{k}-\left\lfloor \frac{\alpha_1^{(k)} \delta_{k}}{\delta_{k}} \right\rfloor \delta_{k} = (\alpha_1^{(k)}-\lfloor \alpha_1^{(k)} \rfloor) \delta_{k} 
=\frac{1}{\alpha_{n-1}^{(k+1)}} \delta_{k} = \delta_{k+1},
\]
\[
\beta_n^{(k+1)}=  \beta_1^{(k)} =\delta_{k}=\alpha_{n-1}^{(k+1)} \delta_{k+1},
\]
and for $1<i<n$, we get 
\begin{align*}
\beta_i^{(k+1)}&= \beta_{i+1}^{(k)}-\left\lfloor \frac{\beta_{i+1}^{(k)}}{\beta_1^{(k)}} \right\rfloor \beta_1^{(k)} = \alpha_i^{(k)} \delta_{k}-\left\lfloor \frac{\alpha_i^{(k)} \delta_{k}}{\delta_{k}} \right\rfloor \delta_{k}=  (\alpha_i^{(k)}-\lfloor \alpha_i^{(k)} \rfloor) \delta_{k}
\\
&=\frac{\left(\alpha_i^{(k)}-\lfloor \alpha_i^{(k)} \rfloor\right)}{\left(\alpha_1^{(k)}-\lfloor \alpha_1^{(k)} \rfloor\right)} \left(\alpha_1^{(k)}-\lfloor \alpha_1^{(k)} \rfloor\right) \delta_{k} = \alpha_{i-1}^{(k+1)} \frac{1}{\alpha_{n-1}^{(k+1)}} \delta_{k} = \alpha_{i-1}^{(k+1)} \delta_{k+1}.
\end{align*}  
 Hence, 
$$\bm{\beta}^{(k+1)}=\left(\delta_{k+1}, \alpha_1^{(k+1)} \delta_{k+1}, \dots, \alpha_{n-2}^{(k+1)} \delta_{k+1}, \alpha_{n-1}^{(k+1)} \delta_{k+1}\right),$$ 
and, from the equations above, it is easy to see that 
\[  \left\lfloor \frac{\beta_{i+1}^{(k)}}{\beta_1^{(k)}} \right\rfloor = \big\lfloor \alpha_i^{(k)}\big\rfloor.\qedhere\]
\end{proof}

Moreover, Proposition \ref{lm:iJPA-JPAarxiv} implies the following statement:

\begin{proposition} \label{lm:iJPA-JPA,perarxiv}
If the iJPA expansion of $(\theta_1, \theta_2, \dots, \theta_{n-1})\in\R^{n-1}$ is periodic with preperiod length $l_0$ and period length $l_1$, then the JPA expansion of $(1, \theta_1, \theta_2, \dots, \theta_{n-1})$ is also periodic with the same length of preperiod and period, and the unit $\varepsilon$ in the period of iJPA expansion is the inverse of the Hasse--Bernstein unit.
\end{proposition}

\begin{proof}
It follows from the previous proposition. It is easy to see that if $k\geq l_0$, then 
$$\delta_{k + l_1}=\frac{1}{\alpha_{n-1}^{(1)}\cdots \alpha_{n-1}^{(k)}\cdots\alpha_{n-1}^{(k + l_1)}}=\delta_{k} \frac{1}{\alpha_{n-1}^{(k+1)}\cdots\alpha_{n-1}^{(k + l_1)}}=\delta_{k}\frac{1}{\alpha_{n-1}^{(l_0)}\cdots \alpha_{n-1}^{(l_0+l_1-1)}}.$$ So, since $\bm{\alpha}^{k+l_1}=\bm{\alpha}^{(k)}$ for $k\geq l_0$, the JPA expansion of $(1, \theta_1, \theta_2, \dots, \theta_{n-1})$ is periodic, and $\varepsilon$ in its period is the inverse of Hasse--Bernstein unit.
\end{proof} 

\setcounter{section}{5}
\section{Jacobi-Perron algorithm in the simplest cubic fields}

In this part, we show several proofs of statements, which were left without a proof in Section \ref{sec:jpasimplest}. 

\subsection{Inhomogenous Jacobi-Perron expansion in the simplest cubic fields}
First of all, in Section \ref{sec:jpasimplest}, we discuss inhomogenous Jacobi-Perron expansion of the vector $(-\rho',\rho'^2)$ where $-2<\rho'<-1$ is a root of the polynomial $x^3-ax^2-(a+3)x-1$ where $a\geq -1$. We derive there its form for $a\geq 4$ (see Proposition \ref{prop:ijpasimplestperiodic4} or Proposition \ref{prop:ijpasimplestappendix} in this section). The expansion of the same vector for $-1\leq a\leq 3$ is also periodic. However, it has a different form, which we show in the following example.  

\begin{example}
For $-1\leq a\leq3$, we can use a computer program (the calculations were performed in Mathematica) to determine the inhomogenous Jacobi-Perron expansion of the vector $(-\rho',\rho'^2)$. We get the following results:
\begin{itemize}
\item For $a=-1$, the iJPA expansion of $(-\rho',\rho'^2)$ is periodic with the preperiod $\bm{\alpha}^{(0)}=(-\rho',\rho'^2)$, $\bm{a}^{(0)}=(1,3)$, and the period 
\[
\begin{array}{ll}
\bm{\alpha}^{(1)}=(5-\rho'-2\rho'^2,-2+\rho'^2), &\quad\bm{a}^{(1)}=(0,1),\\
\bm{\alpha}^{(2)}=(-1-\rho',\rho'^2), &\quad\bm{a}^{(2)}=(0,3).
\end{array}
\]
\item For $a=0$, our iJPA expansion is periodic with the preperiod 
\[
\begin{array}{ll}
\bm{\alpha}^{(0)}=(-\rho',\rho'^2), &\quad\bm{a}^{(0)}=(1,2), \\
\bm{\alpha}^{(1)}=(3-\rho'^2,-2-\rho'+\rho'^2), &\quad\bm{a}^{(1)}=(0,1),
\end{array}
\]
and the period 
\[
\begin{array}{ll}
\bm{\alpha}^{(2)}=(-1+\rho'^2,-\rho'),&\quad\bm{a}^{(2)}=(1,1),\\
\bm{\alpha}^{(3)}=(-\rho',-1-\rho'+\rho'^2),&\quad\bm{a}^{(3)}=(1,2),\\
\bm{\alpha}^{(4)}=(4-\rho'^2,-2-\rho'+\rho'^2), &\quad\bm{a}^{(4)}=(1,1).    
\end{array}
\]
\item For $a=1$, our iJPA expansion is periodic with the preperiod
\[
\begin{array}{ll} 
\bm{\alpha}^{(0)}=(-\rho',\rho'^2), &\quad\bm{a}^{(0)}=(1,1),\\
\bm{\alpha}^{(1)}=(1-\rho',-2-2\rho'+\rho'^2), &\quad\bm{a}^{(1)}=(2,2), 
\end{array}
\]
and the period 
\[
\begin{array}{ll} 
\bm{\alpha}^{(2)}=(5+\rho'-\rho'^2,-2-2\rho'+\rho'^2),&\quad\bm{a}^{(2)}=(1,2),\\
\bm{\alpha}^{(3)}=(-1+\rho'^2,-\rho'),&\quad\bm{a}^{(3)}=(0,1),\\
\bm{\alpha}^{(4)}=\big(\frac{4}{5}-\frac{1}{5}\rho'^2,-\frac{7}{5}-\rho'+\frac{3}{5}\rho'^2\big),&\quad\bm{a}^{(4)}=(0,1),\\
\bm{\alpha}^{(5)}=(-3-\rho'+\rho'^2,1-\rho'),&\quad\bm{a}^{(5)}=(0,2),\\
\bm{\alpha}^{(6)}=(-\rho',-1-2\rho'+\rho'^2),&\quad\bm{a}^{(6)}=(1,3).
\end{array}
\]
\item For $a=2$, our iJPA expansion is periodic with the preperiod 
\[
\begin{array}{ll}
\bm{\alpha}^{(0)}=(-\rho',\rho'^2),&\quad\bm{a}^{(0)}=(1,1),\\
\bm{\alpha}^{(1)}=(1-\rho',-2-3\rho'+\rho'^2),&\quad\bm{a}^{(1)}=(2,3),
\end{array}
\]
and the period 
\[
\begin{array}{ll}
\bm{\alpha}^{(2)}=(6+2\rho'-\rho'^2,-2-3\rho'+\rho'^2),&\quad\bm{a}^{(2)}=(1,3),\\
\bm{\alpha}^{(3)}=(-1+\rho'^2,-\rho'),&\quad\bm{a}^{(3)}=(0,1),\\
\bm{\alpha}^{(4)}=\big(\frac{6}{7}+\frac{1}{7}\rho'-\frac{1}{7}\rho'^2,-\frac{10}{7}-\frac{11}{7}\rho'+\frac{4}{7}\rho'^2\big),&\quad\bm{a}^{(4)}=(0,1),\\
\bm{\alpha}^{(5)}=(-3-2\rho'+\rho'^2,1-\rho'),&\quad\bm{a}^{(5)}=(1,2),\\
\bm{\alpha}^{(6)}=(-\rho',-1-3\rho'+\rho'^2),&\quad\bm{a}^{(6)}=(1,4).
\end{array}
\]
\item For $a=3$, our iJPA expansion is periodic with the preperiod 
\[
\begin{array}{ll}
\bm{\alpha}^{(0)}=(-\rho',\rho'^2),&\quad\bm{a}^{(0)}=(1,1),\\
\bm{\alpha}^{(1)}=(1-\rho',-2-4\rho'+\rho'^2),&\quad\bm{a}^{(1)}=(2,4),
\end{array}
\]
and the period 
\[
\begin{array}{ll}
\bm{\alpha}^{(2)}=(7+3\rho'-\rho'^2,-2-4\rho'+\rho'^2),&\quad\bm{a}^{(2)}=(1,4),\\
\bm{\alpha}^{(3)}=(-1+\rho'^2,-\rho'),&\quad\bm{a}^{(3)}=(0,1),\\
\bm{\alpha}^{(4)}=\big(\frac{8}{9}+\frac{2}{9}\rho'-\frac{1}{9}\rho'^2,-\frac{13}{9}-\frac{19}{9}\rho'+\frac{5}{9}\rho'^2\big),&\quad\bm{a}^{(4)}=(0,1),\\
\bm{\alpha}^{(5)}=(-3-3\rho'+\rho'^2,1-\rho'),&\quad\bm{a}^{(5)}=(2,2),\\
\bm{\alpha}^{(6)}=(-\rho',-1-4\rho'+\rho'^2),&\quad\bm{a}^{(6)}=(1,5).
\end{array}
\]   
\end{itemize}
\end{example}

Moreover, in Proposition \ref{prop:ijpasimplestperiodic4}, we show only a sketch of the proof for $a\geq 4$. Here, one can see it with all steps:

\begin{proposition} \label{prop:ijpasimplestappendix}
Let $a\geq 4$.
Then the Jacobi-Perron expansion of the vector $(-\rho',\rho'^2)$ is periodic. In particular, let $\lfloor\frac{a}{2}\rfloor=A_0$. Then the preperiod of the JPA expansion of $(-\rho',\rho'^2)$ is 
\[
\begin{array}{ll}
\bm{\alpha}^{(0)}=(-\rho',\rho'^2),&\quad\bm{a}^{(0)}=(1,1),\\
\bm{\alpha}^{(1)}=(1-\rho',-2-(a+1)\rho'+\rho'^2), &\quad\bm{a}^{(1)}=(2,a+1).
\end{array}
\] 
For $a$ even, the period is
\begin{align*}
&\bm{\alpha}^{(2)}=(a+4+a\rho'-\rho'^2,-2-(a+1)\rho'+\rho'^2),\\
&\bm{\alpha}^{(3)}=(-1+\rho'^2,-\rho'),\\
&\bm{\alpha}^{(4)}=\left(\frac{2a+2+(a-1)\rho'-\rho'^2}{2a+3},\frac{-3a-4-(a^2+3a+1)\rho'+(a+2)\rho'^2}{2a+3}\right),\\
&\bm{\alpha}^{(5)}= \left(-\left(A_0+2\right)-\left(A_0+1\right)\rho'+\rho'^2,1-\rho'\right),\\
&\bm{\alpha}^{(6)}=\Bigg(\frac{A_0^2+2A_0-1+(A_0-2)\rho'-\rho'^2}{A_0^3+4A_0^2+3A_0-1},\\&\hspace{3cm}\frac{-2A_0^2-3A_0+1-(2A_0^3+3A_0^2+A_0-1)\rho'+(A_0^2+A_0)\rho'^2}{A_0^3+4A_0^2+3A_0-1}\Bigg),\\
& \bm{\alpha}^{(7)}= (-(A_0+4)-2A_0\rho'+\rho'^2,A_0+2-\rho'),\\
& \bm{\alpha}^{(8)}=(-\rho',-1-(a+1)\rho'+\rho'^2),
\end{align*}
and the corresponding vectors of integers parts are $\bm{a}^{(2)}=(1,a+1)$, $\bm{a}^{(3)}=(0,1)$, $\bm{a}^{(4)}=(0,A_0)$, $\bm{a}^{(5)}=(0,2)$, $\bm{a}^{(6)}=(0,1)$, $\bm{a}^{(7)}=(A_0-2,A_0+3)$ and $\bm{a}^{(8)}=(1,a+2)$.

For $a$ odd, the period is
\begin{align*}
&\bm{\alpha}^{(2)}=(a+4+a\rho'-\rho'^2,-2-(a+1)\rho'+\rho'^2),\\
&\bm{\alpha}^{(3)}=(-1+\rho'^2,-\rho'),\\
&\bm{\alpha}^{(4)}=\left(\frac{2a+2+(a-1)\rho'-\rho'^2}{2a+3},\frac{-3a-4-(a^2+3a+1)\rho'+(a+2)\rho'^2}{2a+3}\right),\\
&\bm{\alpha}^{(5)}=\left(-\left(A_0+2\right)-\left(A_0+2\right)\rho'+\rho'^2,1-\rho'\right),\\
& \bm{\alpha}^{(6)}=\Bigg(\frac{A_0^2+3A_0-2+(A_0-2)\rho'-\rho'^2}{A_0^3+6A_0^2+7A_0-5},\\&\hspace{2cm}
\frac{-2A_0^2-5A_0+3-(2A_0^3+6A_0^2+2A_0-3)\rho'+(A_0^2+2A_0-1)\rho'^2}{A_0^3+6A_0^2+7A_0-5}\Bigg),\\
&\bm{\alpha}^{(7)}=(-(A_0+5)-(2A_0+1)\rho'+\rho'^2,A_0+3-\rho'),\\
& \bm{\alpha}^{(8)}=(-\rho',-1-(a+1)\rho'+\rho'^2),
\end{align*}
and the corresponding vectors of integers parts are $\bm{a}^{(2)}=(1,a+1)$, $\bm{a}^{(3)}=(0,1)$, $\bm{a}^{(4)}=(0,A_0)$, $\bm{a}^{(5)}=(1,2)$, $\bm{a}^{(6)}=(0,1)$, $\bm{a}^{(7)}=(A_0-2,A_0+4)$ and $\bm{a}^{(8)}=(1,a+2)$.
\end{proposition}

\begin{proof}
Recall that for $a\geq 7$, we have
\begin{equation} \label{eq:estimates2b}
 -1-\frac{1}{a+1}<\rho'<-1-\frac{1}{a+2}.  
\end{equation}
It is easy to verify that the statement holds in the concrete cases when $4\leq a\leq 6$. Therefore, in the rest of the proof, let us assume $a\geq7$.

\bigskip

\noindent
{\bf First iteration}
\nopagebreak

In the initial iteration, we have the vector $(-\rho',\rho'^2)$. We know that $\lfloor -\rho'\rfloor=1$, i.e., $a_1^{(0)}=1$. Using (\ref{eq:estimates2b}), we see that
\[
1<\rho'^2<\Big(1+\frac{1}{a+1}\Big)^2<1+\frac{2}{a+1}+\frac{1}{(a+1)^2}<2
\] 
for $a\geq 7$. It implies that $a_2^{(0)}=1$ in these cases of $a$. Then we have
\[
\alpha_1^{(1)}=\frac{\rho'^2-1}{-\rho'-1}=1-\rho'.
\]
The fact that
\[
\alpha_2^{(1)}=\frac{1}{-\rho'-1}=-2-(a+1)\rho'+\rho'^2
\] 
can be verified using $\rho'^3-a\rho'^2-(a+3)\rho'-1=0$.

\bigskip

\noindent
{\bf Second iteration}
\nopagebreak

We will proceed similarly with $\bm{\alpha}^{(1)}$. We can immediately conclude that $a_1^{(1)}=\lfloor 1-\rho'\rfloor=2$. On the other hand, 
the element $\alpha_2^{(1)}=-2-(a+1)\rho'+\rho'^2$ is a conjugate of $\rho'$, namely $\rho$, for which we have $a_2^{(1)}=\lfloor\rho\rfloor=a+1$. Then, it is easy to verify that 
\[
\frac{-2-(a+1)\rho'+\rho'^2-(a+1)}{1-\rho'-2}=a+4+a\rho'-\rho'^2 \quad\text{ and } \quad \frac{1}{1-\rho'-2}=-2-(a+1)\rho'+\rho'^2. 
\]
Thus, from now on, let us focus only on the determination of the integer parts, which is the most challenging part of each step.

\bigskip

\noindent
{\bf Third iteration}
\nopagebreak

Similarly as in the second iteration, we have $a_2^{(2)}=\lfloor \rho\rfloor=a+1$. For $a_1^{(2)}$, we can conclude that
\[
\alpha_1^{(2)}=a+4+a\rho'-\rho'^2>a+4-a\Big(1+\frac{1}{a+1}\Big)-\Big(1+\frac{1}{a+1}\Big)^2=1+\frac{a^2+a-1}{(a+1)^2}>1
\]
and
\[
\alpha_1^{(2)}=a+4+a\rho'-\rho'^2<a+4-a\Big(1+\frac{1}{a+2}\Big)-\Big(1+\frac{1}{a+2}\Big)^2=2-\frac{1}{(a+2)^2}<2.
\]
Therefore, $a_1^{(2)}=1$.

\bigskip

\noindent
{\bf Fourth iteration}
\nopagebreak

As before, we see that $a_1^{(3)}=\lfloor-1+\rho'^2\rfloor=0$ and $a_2^{(3)}=\lfloor-\rho'\rfloor=1$.

\bigskip

\noindent
{\bf Fifth iteration}
\nopagebreak

Let us know discuss the next iteration. First of all, we see that
\[
2a+2+(a-1)\rho'-\rho'^2<2a+2-(a-1)-1=a+2<2a+3,
\]
and thus
\[
a_1^{(4)}=\left\lfloor\frac{2a+2+(a-1)\rho'-\rho'^2}{2a+3}\right\rfloor=0.
\]

In the second part of this step, one can easily check that $\alpha_2^{(4)}$ is a root of the polynomial
\[
h(x)=(2a+3)x^3-a^2x^2-(a^2+2a+3)x+1,
\]
which has two positive and one negative root. In particular $h(0)=1>0$, $h(1)=1-2a^2<0$, $8h\big(\frac{a}{2}\big)=- a^3-8a^2-12a+8<0$ and $8h\big(\frac{a+1}{2}\big)=a^3+a^2-9a-1>0$. The element $\alpha_2^{(4)}$ is one of these two positive roots. Moreover, 
\[
-3a-4-(a^2+3a+1)\rho'+(a+2)\rho'^2>-3a-4+(a^2+3a+1)+(a+2)=a^2+a-1>2a+3,
\]
which implies $\alpha_2^{(4)}>1$. It follows that $\lfloor\alpha_2^{(4)}\rfloor=\lfloor\frac{a}{2}\rfloor=A_0$.

\bigskip

\noindent
{\bf Sixth iteration}
\nopagebreak

In the next step, for $a$ even, we can deduce that
\begin{multline*}
\alpha_1^{(5)}=-(A_0+2)-(A_0+1)\rho'+\rho'^2<-(A_0+2)+(A_0+1)\left(1+\frac{1}{2A_0+1}\right)+\left(1+\frac{1}{2A_0+1}\right)^2
\\=\frac{2A_0^2+7A_0+4}{(2A_0+1)^2}<1
\end{multline*}
for $A_0\geq 3$. This gives $a_1^{(5)}=0$ for $a$ even.

If $a$ is odd, we obtain
\begin{multline*}
\alpha_1^{(5)}=-(A_0+2)-(A_0+2)\rho'+\rho'^2<-(A_0+2)+(A_0+2)\left(1+\frac{1}{2A_0+2}\right)+\left(1+\frac{1}{2A_0+2}\right)^2
\\=\frac{6A_0^2+18A_0+13}{4(A_0+1)^2}<2
\end{multline*}
for $A_0\geq 3$. Moreover,
\[
\alpha_1^{(5)}>-(A_0+2)+(A_0+2)\left(1+\frac{1}{2A_0+3}\right)+\left(1+\frac{1}{2A_0+3}\right)^2=1+\frac{A_0+4}{2A_0+3}+\frac{1}{(2A_0+3)^2}>1.
\]
Therefore, $a_1^{(5)}=1$.

In the second part, we immediately obtain $a_2^{(5)}=\lfloor 1-\rho'\rfloor=2$.

\bigskip

\noindent
{\bf Seventh iteration}
\nopagebreak

For $a$ even, we have
\[
A_0^2+2A_0-1+(A_0-2)\rho'-\rho'^2<A_0^2+2A_0-1-(A_0-2)-1=A_0^2+A_0<A_0^3+4A_0^2+3A_0-1
\]
for $A_0\geq 1$. Similarly, for $a$ odd, we obtain
\[
A_0^2+3A_0-2+(A_0-2)\rho'-\rho'^2<A_0^2+3A_0-2-(A_0-2)-1=A_0^2+2A_0-1<A_0^3+6A_0^2+7A_0-5.
\]  
This gives $a_1^{(6)}=0$ for both cases of $a$.
 
Regarding the second coordinate, for $a$ even, we see that
\begin{multline*}
-2 A_0^2 - 3 A_0 + 1 - (2 A_0^3 + 3 A_0^2 + A_0 - 1)\rho' + (A_0^2 + 
    A_0) \rho'^2<\frac{8A_0^5+20A_0^4+18A_0^3+10A_0^2+A_0-1}{(2A_0+1)^2}\\<2(A_0^3+4A_0^2+3A_0-1)
\end{multline*}
and
\begin{multline*}
-2 A_0^2 - 3 A_0 + 1 - (2 A_0^3 + 3 A_0^2 + A_0 - 1)\rho' + (A_0^2 + 
    A_0) \rho'^2>\frac{8A_0^4+20A_0^3+14A_0^2+3A_0-2}{4(A_0+1)}\\>A_0^3+4A_0^2+3A_0-1
\end{multline*}
for $A_0\geq 3$.
Likewise, if $a$ is odd, we have
\begin{multline*}
-2 A_0^2 - 5 A_0 + 3 - (2 A_0^3 + 6 A_0^2 + 2 A_0 - 3) \rho' + (A_0^2 + 2 A_0 - 1) \rho'^2<\frac{8A_0^5+40A_0^4+64A_0^3+37A_0^2-8A_0-15}{4(A_0+1)^2}\\<2(A_0^3+6A_0^2+7A_0-5)
\end{multline*}
and
\begin{multline*}
-2 A_0^2 - 5 A_0 + 3 - (2 A_0^3 + 6 A_0^2 + 2 A_0 - 3) \rho' + (A_0^2 + 2 A_0 - 1) \rho'^2>\frac{8A_0^5+48A_0^4+96A_0^3+66A_0^2-11A_0-25}{(2A_0+3)^2}\\>A_0^3+6A_0^2+7A_0-5
\end{multline*}
for $A_0\geq 3$. This leads to $a_2^{(6)}=1$.

\bigskip

\noindent
{\bf Eighth iteration}
\nopagebreak

If $a$ is even, we can easily deduce that
\[
-(A_0 + 4) - 2 A_0 \rho' + \rho'^2<\frac{4A_0^3-4A_0^2-5A_0}{(2A_0+1)^2}<A_0-1
\]
and 
\[
-(A_0 + 4) - 2 A_0 \rho' + \rho'^2>\frac{4A_0^3-12A_0^2-7A_0}{4(A_0+1)^2}>A_0-2.
\]
In a similar manner, for $a$ odd, we have
\[
-(A_0 + 5) - (2 A_0 + 1) \rho' + \rho'^2<\frac{4A_0^3-10A_0-5}{4(A_0+1)^2}<A_0-1 
\]
and
\[
-(A_0 + 5) - (2 A_0 + 1) \rho' + \rho'^2>\frac{4A_0^3+4A_0^2-15A_0-17}{(2A_0+3)^2}>A_0-2.
\]
In both of these cases, we get $a_1^{(7)}=A_0-2$. 

In the second part, we immediately see that $a_2^{(7)}=\lfloor A_0+2-\rho'\rfloor=A_0+3$ for $a$ even, and $a_2^{(7)}=\lfloor A_0+3-\rho'\rfloor=A_0+4$ for $a$ odd.

\bigskip

\noindent
{\bf Ninth iteration}
\nopagebreak

In this step, we have $a_1^{(8)}=\lfloor-\rho'\rfloor=1$. Moreover,
$a_2^{(8)}=\lfloor-1-(a+1)\rho'+\rho'^2\rfloor=a+2$, which directly follows from the fact that $a_2^{(2)}=\lfloor-2-(a+1)\rho'+\rho'^2\rfloor=a+1$, which we have computed before. Then, it is trivial to check that $\frac{\alpha_2^{(8)}-(a+2)}{\alpha_1^{(8)}-1}=\alpha_1^{(2)}$ and $\frac{1}{\alpha_1^{(8)}-1}=\alpha_2^{(2)}$.
\end{proof}

\subsection{Homogenous Jacobi-Perron expansion in the simplest cubic fields}

Similarly as in the case of iJPA expansion, we can easily derive homogenous JPA expansion of the vector $(1,-\rho', \rho'^2)$ where $-2<\rho'<-1$ is a root of the polynomial $x^3-ax^2-(a+3)x-1$ with $-1\leq a\leq 3$. 

\begin{example} 
For $a=-1$, the expansion of $(1,-\rho', \rho'^2)$ is periodic with the preperiod $\bm{\beta}^{(0)}=(1,-\rho',\rho'^2)$, and the period
\begin{align*}
\bm{\beta}^{(1)}&=(-1-\rho',-3+\rho'^2,1),\\
\bm{\beta}^{(2)}&=(-3+\rho'^2,2+\rho',-1-\rho').
\end{align*}

For $a=0$, we get the expansion with the preperiod $\bm{\beta}^{(0)}=(1,-\rho',\rho'^2)$ and $\bm{\beta}^{(1)}=(-1-\rho',-2+\rho'^2,1)$, and the period
\begin{align*}
\bm{\beta}^{(2)}&=(-2+\rho'^2,2+\rho',-1-\rho'),\\
\bm{\beta}^{(3)}&=(4+\rho'-\rho'^2,1-\rho'-\rho'^2,-2+\rho'^2),\\
\bm{\beta}^{(4)}&=(-3-2\rho',-10-2\rho'+3\rho'^2,4+\rho'-\rho'^2).
\end{align*}

And finally, for $1\leq a\leq 3$, we obtain the expansion with the preperiod $\bm{\beta}^{(0)}=(1,-\rho',\rho'^2)$ and $\bm{\beta}^{(1)}=(-1-\rho',-1+\rho'^2,1)$, and the period
\begin{align*}
\bm{\beta}^{(2)}&=(1+2\rho'+\rho'^2,a+2+(a+1)\rho',-1-\rho'),\\
\bm{\beta}^{(3)}&=(a+1+(a-1)\rho'-\rho'^2,-(a+2)-(2a+3)\rho'-(a+1)\rho'^2,1+2\rho'+\rho'^2),\\
\bm{\beta}^{(4)}&=(-(a+2)-(2a+3)\rho'-(a+1)\rho'^2,-a-(a-3)\rho'+2\rho'^2,a+1+(a-1)\rho'-\rho'^2),\\
\bm{\beta}^{(5)}&=(-a-(a-3)\rho'+2\rho'^2,2a+3+(3a+2)\rho'+a\rho'^2,-(a+2)-(2a+3)\rho'-(a+1)\rho'^2),\\
\bm{\beta}^{(6)}&=(a^2+a+3+(a^2-a+5)\rho'-(a-2)\rho'^2,a-2-9\rho'-(a+5)\rho'^2,-a-(a-3)\rho'+2\rho'^2).
\end{align*}   
\end{example}

\section{Jacobi-Perron algorithm in  Ennola's cubic fields}
In this part, we will focus on both inhomogenous and homogenous JPA expansions for roots of the polynomial $x^3-(a-1)x^2-ax-1$ where $a\geq 5$ and provide proofs, which where omitted in Section \ref{sec:jpaennola}.  

\subsection{Inhomogeneou Jacobi-Perron algorithm in Ennola's cubic fields}

It this section, we will derive the iJPA expansions of the vector $(|\psi|, \psi^2)$, where $\psi$ is one of the roots of polynomial $x^3-(a-1)x^2-ax-1$ where $a\in \mathbb{N}$, $a \geq 5$. Let us denote this polynomial by $f(x)$ (or $f_a(x)$)  and the roots of the polynomial by $\psi, \psi', \psi''$, where $\psi'<\psi''<\psi$. 

We know that
\[ a<a+\frac{a-1}{a^3}<\psi<a+\frac{a^2-1}{a^4}<a+\frac{1}{a^2}.\]
\[-1 < -\frac{a-1}{a}<\psi'< -\frac{a-2}{a-1} < 0,\]
\[-1 < -\frac{1}{a-2}<\psi'' < -\frac{1}{a-1} < 0, \]

\subsubsection{The first root}
First of all, we will focus on
the iJPA expansion of the vector $(\psi, \psi^2)$. We show that it is periodic with preperiod length 2 and period length 1.

\begin{proposition}
Let $a\geq 5$.
Then the inhomogeneous Jacobi-Perron expansion of the vector $\left(\psi,\psi^2\right)$ is periodic. The preperiod of the iJPA expansion of $\left(\psi,\psi^2\right)$ is 
\[
\begin{array}{ll}
\bm{\alpha}^{(0)}=\left(\psi,\psi^2\right),&\quad\bm{a}^{(0)}=(a,a^2),\\
\bm{\alpha}^{(1)}=\left(a+ \psi,\psi+\psi^2\right), &\quad\bm{a}^{(1)}=(2a,a+a^2), 
\end{array}
\] 
and the period is
\[
\begin{array}{ll}
\bm{\alpha}^{(2)}=\left(a+1+ \psi,\psi+\psi^2\right), &\quad\bm{a}^{(2)}=(2a+1,a+a^2). 
\end{array}
\] 
\end{proposition}

\begin{proof}
We will use the same technique as for the simplest cubic fields. Here, we know that $a<\psi< \frac{a^3+1}{a^2}$, and all coefficients $a_i^{(k)}\geq 0$ from the definition of the iJPA algorithm. \\

\noindent
\textbf{First iteration}
\nopagebreak

First of all, we prove that $\bm{a}^{(0)}=(a,a^2)$. 
We know that $a<\psi< a+1$, so $a^{(0)}_1=\lfloor \psi \rfloor = a$. And $a^2<\psi^2< \frac{a^6+2a^3+1}{a^4}=a^2+\frac{2a^3+1}{a^4}<a^2+1$. Hence $a^{(0)}_2=\lfloor \psi^2 \rfloor = a^2$. From this, we get
\begin{eqnarray*}
\alpha_1^{(1)} = \frac{\alpha_2^{(0)}-a_2^{(0)}}{\alpha_1^{(0)}-a_1^{(0)}}  =  \frac{\psi^2-a^2}{\psi -a}  =  a+ \psi
\end{eqnarray*}
and
\begin{eqnarray*}
\alpha_2^{(1)} = \frac{1}{\alpha_1^{(0)}-a_1^{(0)}}  =  \frac{1}{\psi -a}  =  \psi+\psi^2.
\end{eqnarray*}
Thus, the first iteration is $(a+ \psi,\psi+\psi^2)$.\\

\noindent
\textbf{Second iteration}
\nopagebreak

Now, we will prove that $\bm{a}^{(1)}=(2a,a+a^2)$.
From the first iteration, we know that $\lfloor \psi \rfloor = a$, so $a^{(1)}_1=\lfloor a+ \psi \rfloor = a+ \lfloor \psi \rfloor = 2a $. For $a^{(1)}_2$, we can compute that
\[
\alpha^{(1)}_2 = \psi+\psi^2 > a^2+a
\]
and 
\[
\alpha^{(1)}_2 = \psi+\psi^2 < \frac{a^3+1}{a^2}+\frac{a^6+2a^3+1}{a^4}= \frac{a^6+a^5+2a^3+a^2+1}{a^4} < a^2+a+1.
\]
Thus, $a^{(1)}_2=a^2+a$. Now, it is easy to show that
\begin{eqnarray*}
\alpha_1^{(2)} = \frac{\alpha_2^{(1)}-a_2^{(1)}}{\alpha_1^{(1)}-a_1^{(1)}}  =  \frac{\psi+\psi^2-a^2-a}{a+ \psi -2a}  = a+1+ \psi
\end{eqnarray*}
and
\begin{eqnarray*}
\alpha_2^{(2)} = \frac{1}{\alpha_1^{(1)}-a_1^{(1)}}  =  \frac{1}{a+ \psi -2a}  =  \psi+\psi^2.
\end{eqnarray*}
Therefore, $\bm{\alpha}^{(2)}=(a+1+ \psi,\psi+\psi^2)$.\\

\noindent
\textbf{Third iteration}
\nopagebreak

From the second iteration, we know that $a^{(2)}_1=\lfloor a+1+ \psi \rfloor=1+\lfloor a+ \psi \rfloor=2a+1$ and $a^{(2)}_2=\lfloor \psi+\psi^2 \rfloor=a^2+a$. Hence $\bm{a}^{(2)}=(2a+1,a^2+a)$. Moreover, we can easily check that
\begin{eqnarray*}
\alpha_1^{(3)} =\frac{\alpha_2^{(2)}-a_2^{(2)}}{\alpha_1^{(2)}-a_1^{(2)}}  =  \frac{\psi+\psi^2-a^2-a}{a+1+ \psi -2a-1}  =  a+1+ \psi
\end{eqnarray*}
and
\begin{eqnarray*}
\alpha_2^{(3)} = \frac{1}{\alpha_1^{(2)}-a_1^{(2)}}  =  \frac{1}{a+1+ \psi -2a-1}  =  \psi+\psi^2.
\end{eqnarray*}
Thus, $\alpha^{(2)}=(a+1+ \psi,\psi+\psi^2)$. We can see that the second iteration is the same as the third iteration, so we have proved that the iJPA expansion of the vector $(\psi, \psi^2)$ is periodic with preperiod length 2 and period length 1.
\end{proof}

\subsubsection{The second root}
In the next part, we will prove that the iJPA expansion of the vector $(-\psi', \psi'^2)$ is periodic with preperiod length 4 and period length 7.

\begin{proposition}
Let $a\geq 5$.
Then the inhomogeneous Jacobi-Perron expansion of the vector $(-\psi',\psi'^2)$ is periodic. The preperiod of the iJPA expansion of $(-\psi',\psi'^2)$ is 
\[
\begin{array}{ll}
\bm{\alpha}^{(0)}=\left(-\psi',\psi'^2\right),\\
\bm{\alpha}^{(1)}=\left(-\psi',a+(a-1)\psi'-\psi'^2\right), \\
\bm{\alpha}^{(2)}=\left(a^2-2a+1 + (a^2-2a+2)\psi'-(a-1)\psi'^2,a+(a-1)\psi'-\psi'^2\right), \\
\bm{\alpha}^{(3)}=\left(-\psi',-1-(a-1)\psi'+\psi'^2\right),
\end{array}
\] 
and the corresponding vectors of integers parts are $\bm{a}^{(0)}=(0,0)$, $\bm{a}^{(1)}=(0,1)$, $\bm{a}^{(2)}=(0,1)$ and $\bm{a}^{(3)}=(0,a-3)$.
The period is
\[
\begin{array}{ll}
\bm{\alpha}^{(4)}=\left(-a^2+3a-1 - (a^2-3a+3)\psi'+(a-2)\psi'^2,a+(a-1)\psi'-\psi'^2\right),\\
\bm{\alpha}^{(5)}=\left(\frac{-2-3\psi'+\psi'^2}{2a-7},\frac{-2-(2a-4)\psi'+\psi'^2}{2a-7}\right), \\
\bm{\alpha}^{(6)}=\left(a-1+(a-1)\psi'-\psi'^2,-a-(2a-1)\psi'+2\psi'^2\right), \\
\bm{\alpha}^{(7)}=\left(-3-(a-1)\psi'+\psi'^2,-1-a\psi'+\psi'^2\right),\\
\bm{\alpha}^{(8)}=\left(\frac{2a-5+3\psi'-\psi'^2}{2a-7},\frac{-4-(2a-1)\psi'+2\psi'^2}{2a-7}\right),\\
\bm{\alpha}^{(9)}=\left(\frac{-(a-2)^2-(a^2-3a+3)\psi'+(a-2)\psi'^2}{a^2-5a+5}, \frac{a^2-6a+7-(2a-4)\psi'+\psi'^2}{a^2-5a+5}\right) \\
\bm{\alpha}^{(10)}=\left(- \psi',-2-(a-1)\psi'+\psi'^2\right).
\end{array}
\] 
and the corresponding vectors of integers parts are $\bm{a}^{(4)}=(0,1)$, $\bm{a}^{(5)}=(0,1)$, $\bm{a}^{(6)}=(0,a-2)$, $\bm{a}^{(7)}=(a-5,a-2)$, $\bm{a}^{(8)}=(0,1)$, $\bm{a}^{(9)}=(0,1)$ and $\bm{a}^{(10)}=(0,a-4)$.
\end{proposition}

\begin{proof}
We will proceed similarly as before.

\bigskip

\noindent
\textbf{First iteration}
\nopagebreak

From the fact that $-1<\psi'<0$, we directly get
$
a^{(0)}_1  =  \lfloor \alpha^{(0)}_1 \rfloor  =  \lfloor -\psi' \rfloor  =  0,
$ and
$
a^{(0)}_2  =  \lfloor \alpha^{(0)}_2 \rfloor  =  \lfloor \psi'^2 \rfloor  =  0.
$
Now, it is easy to see that 
\[ 
\alpha_1^{(1)}  =  \frac{\alpha_2^{(0)}-a_2^{(0)}}{\alpha_1^{(0)}-a_1^{(0)}}  =  \frac{\psi'^2}{-\psi'}  =  -\psi'
\]
and 
\[
\alpha_2^{(1)}  =  \frac{1}{\alpha_1^{(0)}-a_1^{(0)}}  =  \frac{1}{-\psi'}  =  a + (a-1)\psi' - \psi'^2.
\]

\bigskip

\noindent
\textbf{Second iteration}
\nopagebreak

From the first iteration, we know that $a^{(1)}_1= \lfloor -\psi' \rfloor=0$. So, we only need to compute $a^{(1)}_2=\lfloor \alpha^{(1)}_2 \rfloor$: 
\[
\alpha^{(1)}_2  =  a+(a-1)\psi'-\psi'^2   >   a-(a-1)\frac{a-1}{a}-\frac{a^2-2a+1}{a^2}  =  1+\frac{a-1}{a^2}  >  1,
\]
\begin{eqnarray*}
\alpha^{(1)}_2 = a+(a-1)\psi'-\psi'^2   <   a-(a-1)\frac{a-2}{a-1}-\frac{a^2-4a+4}{a^2-2a+1}  
\\
= 2-\frac{a^2-4a+4}{a^2-2a+1} <  2.
\end{eqnarray*}
Hence, we have $\bm{a}^{(1)} =  (0,1)$.

Moreover, we can simply compute that
\begin{eqnarray*}
\alpha_1^{(2)} = \frac{\alpha_2^{(1)}-a_2^{(1)}}{\alpha_1^{(1)}-a_1^{(1)}} =  \frac{a+(a-1)\psi'-\psi'^2-1}{-\psi'} 
= a^2-2a+1 + (a^2-2a+2)\psi'-(a-1)\psi'^2
\end{eqnarray*}
and
\[
\alpha_2^{(2)}  =  \frac{1}{\alpha_1^{(1)}-a_1^{(1)}}  =  \frac{1}{-\psi'}  =  a + (a-1)\psi' - \psi'^2.
\] 
This means that $\bm{\alpha}^{(2)}=(a^2-2a+1 + (a^2-2a+2)\psi'-(a-1)\psi'^2,a + (a-1)\psi' - \psi'^2)$.

\bigskip

\noindent
\textbf{Third iteration}
\nopagebreak

No, we will show that $\bm{a}^{(2)}=(0,1)$:
\begin{align*}
\alpha^{(2)}_1&=a^2-2a+1 + (a^2-2a+2)\psi'-(a-1)\psi'^2
\\
&\hspace{1cm}< a^2-2a+1 - (a^2-2a+2)\frac{a-2}{a-1}-(a-1)\frac{a^2-4a+4}{a^2-2a+1}
\\
&\hspace{1.5cm}= a^2-2a+1 - (a^2-2a+2)+\frac{a^2-2a+2}{a-1}-\frac{a^2-4a+4}{a-1}
\\
&\hspace{2cm}= -1+ \frac{2a-2}{a-1}=-1+2 = 1.
\end{align*}
Since the element $\alpha^{(2)}_1$ is non-negative, we have $a^{(2)}_1=0$. From the second iteration, we can see that $a^{(2)}_2=a^{(1)}_2=1$.
Hence
\[
\alpha_1^{(3)}  =  \frac{\alpha_2^{(2)}-a_2^{(2)}}{\alpha_1^{(2)}-a_1^{(2)}}  =  \frac{a + (a-1)\psi' - \psi'^2-1}{a^2-2a+1 + (a^2-2a+2)\psi'-(a-1)\psi'^2}  =  -\psi'
\]
and
\begin{eqnarray*}
\alpha_2^{(3)}=\frac{1}{\alpha_1^{(2)}-a_1^{(2)}}  =  \frac{1}{a^2-2a+1 + (a^2-2a+2)\psi'-(a-1)\psi'^2}
=-1-(a-1)\psi'+\psi'^2.
\end{eqnarray*}
By this, we have verified that $\bm{\alpha}^{(3)}=(-\psi',-1-(a-1)\psi'+\psi'^2)$.

\bigskip

\noindent
\textbf{Fourth iteration}
\nopagebreak

As before, we first calculate $\bm{a}^{(3)}$. Clearly, $a^{(3)}_1 = 0$, and so it remains to compute $a^{(3)}_2$. We have
\begin{eqnarray*}
\alpha^{(3)}_2= -1-(a-1)\psi'+\psi'^2 >  -1+(a-1)\frac{a-2}{a-1}+\frac{a^2-4a+4}{a^2-2a+1} 
\\
= -1+a-2+\frac{1}{a}+\frac{a^2-4a+4}{a^2-2a+1}>a-3
\end{eqnarray*}
and 
\begin{eqnarray*}
\alpha^{(3)}_2= -1-(a-1)\psi'+\psi'^2 <-1+(a-1)\frac{a-1}{a}+\frac{a^2-2a+1}{a^2}
\\
= -1+a-2+\frac{1}{a}+1+\frac{-2a+1}{a^2}=a-2-\frac{a-1}{a^2}<a-2.
\end{eqnarray*}
Thus, we get $\bm{a}^{(3)}=(0,a-3)$. A trivial verification shows that
\begin{eqnarray*}
\alpha^{(4)}_1  =  \frac{\alpha_2^{(3)}-a_2^{(3)}}{\alpha_1^{(3)}-a_1^{(3)}}  = \frac{-1-(a-1)\psi'+\psi'^2-a+3}{-\psi'}
=-a^2+3a-1 - (a^2-3a+3)\psi'+(a-2)\psi'^2
\end{eqnarray*}
and, as before,
$
\alpha_2^{(4)} = (-\psi')^{-1}=a+(a-1)\psi'-\psi'^2.
$
Therefore, we have \[\bm{\alpha}^{(4)}=(-a^2+3a-1 - (a^2-3a+3)\psi'+(a-2)\psi'^2,a+(a-1)\psi'-\psi'^2).\]

\bigskip

\noindent
\textbf{Fifth iteration}
\nopagebreak

Firstly, we will show that $ \bm{a}^{(4)}=(0,1)$. The element $\alpha^{(4)}_1$ is greater or equal to zero, so we just need to show that $\alpha^{(4)}_1<1$. Here we need to use the fact that $\frac{-a^4+a^3+2a+2}{a^4}<\psi'$:
\begin{align*}
\alpha^{(4)}_1&=-a^2+3a-1 - (a^2-3a+3)\psi'+(a-2)\psi'^2 
\\
&\hspace{1cm}< -a^2+3a-1 - (a^2-3a+3)\frac{-a^4+a^3+2a+2}{a^4}+(a-2)\left(\frac{-a^4+a^3+2a+2}{a^4}\right)^2
\\
&\hspace{1.5cm}= \frac{a^8- 2 a^6+ 8 a^5- 2 a^4- 4 a^3- 12 a-8  }{a^8} < 1.
\end{align*}
Hence, $a^{(4)}_1=0$, and from the second iteration, we know that $a^{(4)}_2=1$. Now, it is easy to verify that 
\begin{eqnarray*}
\alpha^{(5)}_1  =  \frac{\alpha_2^{(4)}-a_2^{(4)}}{\alpha_1^{(4)}-a_1^{(4)}}  = \frac{a+(a-1)\psi'-\psi'^2-1}{-a^2+3a-1 - (a^2-3a+3)\psi'+(a-2)\psi'^2}
\\
=\frac{1}{2a-7}(-2-3\psi'+\psi'^2)
\end{eqnarray*}
and
\begin{eqnarray*}
\alpha^{(5)}_2  =   \frac{1}{\alpha_1^{(4)}-a_1^{(4)}}  = \frac{1}{-a^2+3a-1 -(a^2-3a+3)\psi'+(a-2)\psi'^2}
\\
=\frac{1}{2a-7}(-2-(2a-4)\psi'+\psi'^2).
\end{eqnarray*}
Therefore, we get $\bm{\alpha}^{(5)}=\big(\frac{1}{2a-7}(-2-3\psi'+\psi'^2),\frac{1}{2a-7}(-2-(2a-4)\psi'+\psi'^2)\big)$.

\bigskip

\noindent
\textbf{Sixth iteration}
\nopagebreak

Now, we will compute $\bm{a}^{(5)}$. For $\alpha^{(5)}_1$, we can see that
\begin{eqnarray*}
\alpha^{(5)}_1= \frac{1}{2a-7}(-2-3\psi'+\psi'^2) < \frac{1}{2a-7}\left(-2+3\frac{a-1}{a}+\frac{a^2-2a+1}{a^2}\right)
\\
= \frac{1}{2a-7}\left( -2+3-\frac{3}{a}+1+\frac{-2a+1}{a^2} \right) = \frac{1}{2a-7}\frac{2a^2-5a+1}{a^2} <1.
\end{eqnarray*}
Since $\alpha^{(5)}_1 \geq 0$, we get $a_1^{(5)}=0$. For $\alpha^{(5)}_2$, it holds that
\begin{align*}
\alpha^{(5)}_2&= \frac{1}{2a-7}\left( -2-(2a-4)\psi'+\psi'^2 \right)
\\
&\hspace{1cm}> \frac{1}{2a-7}\left( -2+(2a-4)\frac{a-2}{a-1}+\frac{a^2-4a+4}{a^2-2a+1} \right)
\\
&\hspace{1.5cm}= \frac{1}{2a-7}\left( -2+2a-6+ \frac{2}{a-1} + 1+ \frac{-2a+3}{a^2-2a+1}\right)
\\
&\hspace{2cm}= 1+ \frac{1}{(2a-7)(a^2-2a+1)} >1
\end{align*}
and
\begin{align*}
\alpha^{(5)}_2&= \frac{1}{2a-7}\left( -2-(2a-4)\psi'+\psi'^2 \right)
\\
&\hspace{1cm}< \frac{1}{2a-7}\left( -2-(2a-4)\frac{a-1}{a}+\frac{a^2-2a+1}{a^2} \right)
\\
&\hspace{1.5cm}= \frac{1}{2a-7}\left(-2+2a-6+\frac{4}{a}+1+\frac{-2a+1}{a^2} \right)
\\
&\hspace{2cm}= 1+\frac{2a+1}{a^2(2a-7)}<2.
\end{align*}
Hence $\bm{a}^{(5)}=(0,1)$. Now it is easy to verify that
\begin{eqnarray*}
\alpha_1^{(6)}=\frac{\alpha_2^{(5)}-a_2^{(5)}}{\alpha_1^{(5)}-a_1^{(5)}}  =   \frac{\frac{1}{2a-7}\left( -2-(2a-4)\psi'+\psi'^2 \right)-1}{\frac{1}{2a-7}\left(-2-3\psi'+\psi'^2 \right)}
= a-1+(a-1)\psi'-\psi'^2
\end{eqnarray*}
and
\begin{eqnarray*}
\alpha^{(6)}_2  =  \frac{1}{\alpha_1^{(5)}-a_1^{(5)}}  =   \frac{1}{\frac{1}{2a-7}\left(-2-3\psi'+\psi'^2 \right)} 
= -a-(2a-1)\psi'+2\psi'^2.
\end{eqnarray*}
Therefore, we get $\bm{\alpha}^{(6)}=(a-1+(a-1)\psi'-\psi'^2,-a-(2a-1)\psi'+2\psi'^2)$.

\bigskip

\noindent
\textbf{Seventh iteration}
\nopagebreak

From the second iteration, we know that $a^{(1)}_2=1=\lfloor a+(a-1)\psi'-\psi'^2 \rfloor$. Hence, it is clear that $a^{(6)}_1=\lfloor a-1+(a-1)\psi'-\psi'^2 \rfloor = 0$. Thus, we only need to compute $a^{(6)}_2$. For $a>5$, it holds that 
\begin{multline*}
\alpha_2^{(6)} = -a-(2a-1)\psi'+2\psi'^2 > -a+(2a-1)\frac{a-2}{a-1}+2\frac{a^2-4a+4}{a^2-2a+1}
\\
= -a+2a-3-\frac{1}{a-1}+2+\frac{-4a+6}{a^2-2a+1} = a-1+\frac{-5a+7}{a^2-2a+1} 
> a-2.
\end{multline*}
For $a=5$, we have $f_5(x)=x^3-4x^2-5x-1$ and $\psi'<\frac{-31}{40}$. Then we get 
\begin{eqnarray*}
\alpha_2^{(6)} = -5-9\psi'+2\psi'^2 > -5+9\frac{31}{40}+2\frac{31^2}{40^2}
= \frac{2541}{800} > 3 = a-2.
\end{eqnarray*}
Therefore, $\alpha_2^{(6)}>a-2$ for all $a \geq 5$. For $\alpha_2^{(6)}$, it also holds that
 \begin{multline*}
\alpha_2^{(6)} = -a-(2a-1)\psi'+2\psi'^2 < -a+(2a-1)\frac{a-1}{a}+2\frac{a^2-2a+1}{a^2}
\\
= -a+2a-3+\frac{1}{a}+2+\frac{-4a+2}{a^2} =a-1+\frac{-3a+2}{a^2} < a-1.
\end{multline*}
Hence $\bm{a}^{(6)}=(0,a-2)$. 
Moreover, we can verify that
\begin{eqnarray*}
\alpha_1^{(7)}=\frac{\alpha_2^{(6)}-a_2^{(6)}}{\alpha_1^{(6)}-a_1^{(6)}}  =  \frac{-a-(2a-1)\psi'+2\psi'^2-a+2}{a-1+(a-1)\psi'-\psi'^2}
= -3-(a-1)\psi'+\psi'^2
\end{eqnarray*}
and
\begin{eqnarray*}
\alpha_2^{(7)}=\frac{1}{\alpha_1^{(6)}-a_1^{(6)}}  =  \frac{1}{a-1+(a-1)\psi'-\psi'^2}
= -1-a\psi'+\psi'^2.
\end{eqnarray*}
Therefore, $\bm{\alpha}^{(7)}=(-3-(a-1)\psi'+\psi'^2,-1-a\psi'+\psi'^2)$.

\bigskip

\noindent
\textbf{Eighth iteration}
\nopagebreak

In this part, we will show that $\bm{a}^{(7)}=(a-5,a-2)$. For $\alpha^{(7)}_1$, we have
\begin{multline*}
\alpha^{(7)}_1 = -3-(a-1)\psi'+\psi'^2 > -3+(a-1)\frac{a-2}{a-1}+\frac{a^2-4a+4}{a^2-2a+1}
\\
= -3+a-2+1+\frac{-2a+3}{a^2-2a+1} = a-4+\frac{-2a+3}{a^2-2a+1} > a-5
\end{multline*}
and
\begin{multline*}
\alpha^{(7)}_1 = -3-(a-1)\psi'+\psi'^2 < -3+(a-1)\frac{a-1}{a}+\frac{a^2-2a+1}{a^2}
\\
= -3+a-2+\frac{1}{a}+1+\frac{-2a+1}{a^2} = a-4+\frac{-a+1}{a^2} <a-4.
\end{multline*}
For $\alpha^{(7)}_2$, we can see that
\begin{multline*}
\alpha^{(7)}_2 = -1-a\psi'+\psi'^2 > -1+a\frac{a-2}{a-1}+\frac{a^2-4a+4}{a^2-2a+1} 
\\
= -1+a-1-\frac{1}{a-1} +1+\frac{-2a+3}{a^2-2a+1} = a-1+\frac{-3a+4}{a^2-2a+1} 
> a-2
\end{multline*}
and
\begin{multline*}
\alpha^{(7)}_2 = -1-a\psi'+\psi'^2 < -1+a\frac{a-1}{a}+\frac{a^2-2a+1}{a^2} 
= -1+a-1+1+\frac{-2a+1}{a^2} < a-1.
\end{multline*}
Now, 
it is easy to verify the following equalities:
\begin{eqnarray*}
\alpha_1^{(8)}=\frac{\alpha_2^{(7)}-a_2^{(7)}}{\alpha_1^{(7)}-a_1^{(7)}}  =  \frac{-1-a\psi'+\psi'^2-a+2}{-3-(a-1)\psi'+\psi'^2-a+5}
= \frac{1}{2a-7}(2a-5+3\psi'-\psi'^2),
\end{eqnarray*}
\begin{eqnarray*}
\alpha_2^{(8)}=\frac{1}{\alpha_1^{(7)}-a_1^{(7)}}  =  \frac{1}{-3-(a-1)\psi'+\psi'^2-a+5}
= \frac{1}{2a-7}(-4-(2a-1)\psi'+2\psi'^2).
\end{eqnarray*}
Therefore, $\bm{\alpha}^{(8)}=\left(\frac{1}{2a-7}(2a-5+3\psi'-\psi'^2), \frac{1}{2a-7}(-4-(2a-1)\psi'+2\psi'^2) \right)$.

\bigskip

\noindent
\textbf{Ninth iteration}
\nopagebreak

Firstly, we will compute $\bm{a}^{(8)}$ in the following way:
\begin{multline*}
\alpha^{(8)}_1=\frac{1}{2a-7}(2a-5+3\psi'-\psi'^2) < \frac{1}{2a-7}\left(2a-5-3\frac{a-2}{a-1}-\frac{a^2-4a+4}{a^2-2a+1}\right)
\\
= \frac{1}{2a-7}\left( 2a-5-3+\frac{3}{a-1}-\frac{a^2-4a+4}{a^2-2a+1} \right) 
=  1-\frac{2a^2-9a+8}{(2a-7)(a^2-2a+1)}<1. 
\end{multline*}
Since $\alpha_1^{(8)} \geq 0$, we get $a_1^{(8)}=0$. For $a_2^{(8)}$, we see that
\begin{multline*}
\alpha^{(8)}_2=\frac{1}{2a-7}(-4-(2a-1)\psi'+2\psi'^2) 
> \frac{1}{2a-7} \left( -4+(2a-1)\frac{a-2}{a-1}+2\frac{a^2-4a+4}{a^2-2a+1} \right) 
\\
= \frac{1}{2a-7}\left( -4+2a-3-\frac{1}{a-1}+2+\frac{-4a+6}{a^2-2a+1} \right) 
= 1+\frac{1}{2a-7}\left( 2 + \frac{-5a+7}{a^2-2a+1} \right)>1 
\end{multline*}
and
\begin{multline*}
\alpha^{(8)}_2=\frac{1}{2a-7}(-4-(2a-1)\psi'+2\psi'^2)  
< \frac{1}{2a-7}\left(-4+(2a-1)\frac{a-1}{a}+2\frac{a^2-2a+1}{a^2}\right) 
\\
= \frac{1}{2a-7}\left( -4+2a-3+\frac{1}{a}+2+\frac{-4a+2}{a^2} \right) 
= 1+ \frac{1}{2a-7}\left( 2 + \frac{-3a+2}{a^2} \right)<2.
\end{multline*}
Hence $\bm{a}^{(8)}=(0,1)$. Now, it is easy to find the value of  $\bm{\alpha}^{(9)}$:
\begin{align*}
\alpha_1^{(9)}&=\frac{\alpha_2^{(8)}-a_2^{(8)}}{\alpha_1^{(8)}-a_1^{(8)}}  =  \frac{\frac{1}{2a-7}(-4-(2a-1)\psi'+2\psi'^2)-1}{\frac{1}{2a-7}(2a-5+3\psi''--\psi'^2)}
\\
&= \frac{1}{a^2-5a+5}(-(a-2)^2-(a^2-3a+3)\psi'+(a-2)\psi'^2)
\end{align*}
and
\begin{eqnarray*}
\alpha_2^{(9)}=\frac{1}{\alpha_1^{(8)}-a_1^{(8)}}  =  \frac{1}{\frac{1}{2a-7}(2a-5+3\psi'-\psi'^2)}
= \frac{1}{a^2-5a+5}(a^2-6a+7-(2a-4)\psi'+\psi'^2).
\end{eqnarray*} 
Therefore, we get 
\begin{multline*}
\bm{\alpha}^{(9)}=\bigg(\frac{1}{a^2-5a+5}(-(a-2)^2-(a^2-3a+3)\psi'+(a-2)\psi'^2),\\ \frac{1}{a^2-5a+5}(a^2-6a+7-(2a-4)\psi'+\psi'^2)\bigg).
\end{multline*}

\bigskip

\noindent
\textbf{Tenth iteration}
\nopagebreak

Here, we will show that $\bm{a}^{(9)}=(0,1)$, and we will start with $a_1^{(9)}$:
\begin{align*}
\alpha_1^{(9)}&=\frac{1}{a^2-5a+5}(-(a-2)^2-(a^2-3a+3)\psi'+(a-2)\psi'^2)
\\
&\hspace{1cm}<\frac{1}{a^2-5a+5}\left( -(a-2)^2+(a^2-3a+3)\frac{a-1}{a}+(a-2)\frac{a^2-2a+1}{a^2} \right)
\\
&\hspace{1.5cm}=\frac{1}{a^2-5a+5}\left( -a^2+4a-4 +a^2-4a+6-\frac{3}{a}+a-4+\frac{5a-2}{a^2} \right)
\\
&\hspace{2cm}=\frac{1}{a^2-5a+5}\left( a-2 +\frac{2a-2}{a^2} \right) \;<\;1. 
\end{align*}
So, $a^{(9)}_1=0$ since $\alpha^{(9)}_1$ is non-negative. For $a_2^{(9)}$, we get the following inequalities:
\begin{align*}
\alpha_2^{(9)}&= \frac{1}{a^2-5a+5}(a^2-6a+7-(2a-4)\psi'+\psi'^2)
\\
&\hspace{1cm}>\frac{1}{a^2-5a+5}\left(a^2-6a+7+(2a-4)\frac{a-2}{a-1}+\frac{a^2-4a+4}{a^2-2a+1}\right)
\\
&\hspace{1.5cm}=\frac{1}{a^2-5a+5}\left( a^2-6a+7+2a-4-2+\frac{2}{a-1}+1+\frac{-2a+3}{a^2-2a+1} \right)
\\
&\hspace{2cm}=1+\frac{1}{a^2-5a+5}\left( a-3+\frac{1}{a^2-2a+1} \right)>1
\end{align*}
and
\begin{align*}
\alpha_2^{(9)}&= \frac{1}{a^2-5a+5}(a^2-6a+7-(2a-4)\psi'+\psi'^2)
\\
&\hspace{1cm}<\frac{1}{a^2-5a+5}\left( a^2-6a+7+(2a-4)\frac{a-1}{a}+\frac{a^2-2a+1}{a^2}\right)
\\
&\hspace{1.5cm}=\frac{1}{a^2-5a+5}\left( a^2-6a+7+2a-6+\frac{4}{a}+1+\frac{-2a+1}{a^2}\right)
\\
&\hspace{2cm}=1+\frac{1}{a^2-5a+5}\left( a-3+\frac{2a+1}{a^2} \right) <2.
\end{align*}
Thus, $\bm{a}^{(9)}=(0,1)$.

Furthermore, we can see that
\begin{eqnarray*}
\alpha_1^{(10)}=\frac{\alpha_2^{(9)}-a_2^{(9)}}{\alpha_1^{(9)}-a_1^{(9)}}  =  \frac{ \frac{1}{a^2-5a+5}(a^2-6a+7-(2a-4)\psi'+\psi'^2)-1}{\frac{1}{a^2-5a+5}(-(a-2)^2-(a^2-3a+3)\psi'+(a-2)\psi'^2)}
= - \psi',
\end{eqnarray*}
\begin{align*}
\alpha_2^{(10)}&=\frac{1}{\alpha_1^{(9)}-a_1^{(9)}}  =  \frac{ 1}{\frac{1}{a^2-5a+5}(-(a-2)^2-(a^2-3a+3)\psi'+(a-2)\psi'^2)}
\\
&= -2-(a-1)\psi'+\psi'^2.
\end{align*}
Hence, the tenth iteration is $\bm{\alpha}^{(10)}=(- \psi', -2-(a-1)\psi'+\psi'^2)$.

\bigskip

\noindent
\textbf{Eleventh iteration}
\nopagebreak

From the first iteration, we know that $a_1^{(10)}=\lfloor -\psi' \rfloor=0$. Moreover, the fourth iteration gives $\lfloor -1-(a-1)\psi'+\psi'^2 \rfloor = a-3$. Hence $a^{(10)}_2=\lfloor -2-(a-1)\psi'+\psi'^2 \rfloor = a-4$. Furthermore, a trivial verification shows that
\begin{align*}
\alpha_1^{(11)}&=\frac{\alpha_2^{(10)}-a_2^{(10)}}{\alpha_1^{(10)}-a_1^{(10)}}=\frac{-2-(a-1)\psi'+\psi'^2-a+4}{- \psi'}
\\
&= -a^2+3a-1 - (a^2-3a+3)\psi'+(a-2)\psi'^2,
\end{align*}
\begin{eqnarray*}
\alpha_2^{(11)}=\frac{1}{\alpha_1^{(10)}-a_1^{(10)}}=\frac{1}{- \psi'}
= a+(a-1)\psi'-\psi'^2.
\end{eqnarray*}
Therefore,  $\bm{\alpha}^{(11)}=(-a^2+3a-1 - (a^2-3a+3)\psi'+(a-2)\psi'^2,  a+(a-1)\psi'-\psi'^2)=\bm{\alpha}^{(4)}$. By this, we have shown that the iJPA expansion of the vector $(-\psi', \psi'^2)$ is periodic with preperiod length 4 and period length 7 and derived its form. 
\end{proof}

\subsubsection{The third root}
And finally, we will prove that the iJPA expansion of the vector $(-\psi'', \psi''^2)$ is periodic with preperiod length 3 and period length 3.

\begin{proposition}
Let $a\geq 5$.
Then the inhomogeneous Jacobi-Perron expansion of the vector $\left(-\psi'',\psi''^2\right)$ is periodic. The preperiod of the iJPA expansion of $\left(-\psi'',\psi''^2\right)$ is 
\[
\begin{array}{ll}
\bm{\alpha}^{(0)}=\left(-\psi'',\psi''^2\right),&\quad\bm{a}^{(0)}=(0,0),\\
\bm{\alpha}^{(1)}=\left(-\psi'',a+(a-1)\psi''-\psi''^2\right), &\quad\bm{a}^{(1)}=(0,a-2), \\
\bm{\alpha}^{(2)}=\left(a+1 + (2a-1)\psi''-2\psi''^2,a+(a-1)\psi''-\psi''^2\right), &\quad\bm{a}^{(2)}=(a-2,a-2), 
\end{array}
\] 
and the period is
\[
\begin{array}{ll}
\bm{\alpha}^{(3)}=\left(\frac{2a^2-9a+10+(a^2-2a-1)\psi''-(a-2)\psi''^2}{2a^2-12a+17}, \frac{2a^2-8a+8+(2a^2-7a+6)\psi''-(2a-5)\psi''^2}{2a^2-12a+17}\right),\\
\bm{\alpha}^{(4)}=\left(1- \psi'', a-2+a\psi''-\psi''^2\right), \\
\bm{\alpha}^{(5)}=\left(a + (2a-1)\psi''-2\psi''^2, a+(a-1)\psi''-\psi''^2 \right).
\end{array}
\] 
and the corresponding vectors of integers parts are $\bm{a}^{(3)}=(1,1)$, $\bm{a}^{(4)}=(1,a-4)$ and $\bm{a}^{(5)}=(a-3,a-2)$. 
\end{proposition}

\begin{proof}
We use the same technique as for the other root. For the third root, we know that $\frac{-1}{a-2}<\psi''< \frac{-1}{a-1}$. 

\bigskip

\noindent
\textbf{First iteration}
\nopagebreak

Similarly as before, $|\psi''|<1$, which gives $\bm{a}^{(0)}=(0,0)$ and 
$ \bm{\alpha} ^{(1)} = (- \psi'', a + (a-1) \psi''- \psi''^2) $.

\bigskip

\noindent
\textbf{Second iteration}
\nopagebreak

In this part, we will prove that $\bm{a}^{(1)}=(0,a-2)$. From the previous iteration, we know that $a^{(1)}_1=\lfloor \alpha^{(1)}_1 \rfloor=0$. Thus, it is enough to verify that $a^{(1)}_2=\lfloor \alpha^{(1)}_2 \rfloor=a-2$. We have
\begin{multline*}
\alpha_2^{(1)} = a + (a-1) \psi''- \psi''^2  >  a - (a-1) \frac{1}{a-2}- \frac{1}{a^2-4a+4}
\\
= a-1-\frac{1}{a-2}-\frac{1}{a^2-4a+4}  =  a-1-\frac{a-1}{a^2-4a+4}  >  a-2
\end{multline*}
and
\[
\alpha_2^{(1)} = a + (a-1) \psi''- \psi''^2 <  a - (a-1) \frac{1}{a-1}- \frac{1}{a^2-2a+1}
= a-1- \frac{1}{a^2-2a+1}  < a-1.
\]
Therefore, we get
\begin{eqnarray*}
\alpha_1^{(2)} = \frac{\alpha_2^{(1)}-a_2^{(1)}}{\alpha_1^{(1)}-a_1^{(1)}}  =  \frac{a+(a-1)\psi''-\psi''^2-a+2}{-\psi''} 
= a+1 + (2a-1)\psi''-2\psi''^2
\end{eqnarray*}
and
\[
\alpha_2^{(2)}  =  \frac{1}{\alpha_1^{(1)}-a_1^{(1)}}  =  \frac{1}{-\psi''}  =  a + (a-1)\psi'' - \psi''^2.
\]
Hence $\bm{\alpha}^{(2)}=(a+1 + (2a-1)\psi''+(-2)\psi''^2, a+(a-1)\psi''-\psi''^2)$.

\bigskip

\noindent
\textbf{Third iteration}
\nopagebreak

First, we will compute $\bm{a}^{(2)}=(a-2,a-2)$. From the previous iteration, we know that $a_2^{(2)}= a-2$. Thus, it is enough to compute $a_1^{(2)}$. We can deduce that
\begin{multline*}
\alpha_1^{(2)} = a+1 + (2a-1)\psi''-2\psi''^2  >  a+1 - (2a-1)\frac{1}{a-2}-2\frac{1}{a^2-4a+4}
\\  =  a-1+\frac{-3a+4}{a^2-4a+4}  > a-2 
\end{multline*}
for $a\geq 6$. If $a=5$, we can easily check that $-\frac{4}{13}<\psi''$ and
\[
6+9\psi''-2\psi''^2>6-9\frac{4}{13}-\Big(\frac{4}{13}\Big)^2=\frac{514}{169}>3=a-2.
\]
On the other hand,
\begin{multline*}
\alpha_1^{(2)} = a+1 + (2a-1)\psi''-2\psi''^2  <  a+1 - (2a-1)\frac{1}{a-1}-2\frac{1}{a^2-2a+1}
\\
 =  a-1-\frac{a+1}{a^2-2a+1}  < a-1.
\end{multline*}
Now, it is easy to verify that
\begin{align*}
\alpha_1^{(3)} &=\frac{\alpha_2^{(2)}-a_2^{(2)}}{\alpha_1^{(2)}-a_1^{(2)}}  =  \frac{ a+(a-1)\psi''-\psi''^2-a+2}{a+1 + (2a-1)\psi''-2\psi''^2-a+2} 
\\
&= \frac{1}{2a^2-12a+17}(2a^2-9a+10+(a^2-2a-1)\psi''-(a-2)\psi''^2)
\end{align*}
and
\begin{align*}
\alpha_2^{(3)} &= \frac{1}{\alpha_1^{(2)}-a_1^{(2)}}  =  \frac{1}{a+1 + (2a-1)\psi''-2\psi''^2-a+2} 
\\
 &=  \frac{1}{2a^2-12a+17}(2a^2-8a+8+(2a^2-7a+6)\psi''-(2a-5)\psi''^2)).
\end{align*}
Therefore, 
\begin{multline*}
\bm{\alpha}^{(3)}=\Big(\frac{1}{2a^2-12a+17}(2a^2-9a+10+(a^2-2a-1)\psi''-(a-2)\psi'^2),\\ \frac{1}{2a^2-12a+17}(2a^2-8a+8+(2a^2-7a+6)\psi''-(2a-5)\psi''^2)\Big).
\end{multline*}

\bigskip

\noindent
\textbf{Fourth iteration}
\nopagebreak

We will show that $\bm{a}^{(3)}=(1,1)$. For $\alpha_1^{(3)}$, we can conclude that
\begin{align*}
\alpha_1^{(3)} &= \frac{1}{2a^2-12a+17}(2a^2-9a+10+(a^2-2a-1)\psi''-(a-2)\psi''^2 
\\
& \hspace{1cm}>  \frac{1}{2a^2-12a+17}\left(2a^2-9a+10-(a^2-2a-1)\frac{1}{a-2}-(a-2)\frac{1}{a^2-4a+4}\right)
\\
&\hspace{1.5cm}= \frac{1}{2a^2-12a+17}\left(2a^2-9a+10-a+\frac{1}{a-2}-\frac{a-2}{a^2-4a+4}\right)
\\
 &\hspace{2cm}= 1+ \frac{1}{2a^2-12a+17}\left(2a-7\right)  > 1 
\end{align*}
and
\begin{align*}
\alpha_1^{(3)} &= \frac{1}{2a^2-12a+17}(2a^2-9a+10+(a^2-2a-1)\psi''-(a-2)\psi''^2 )
\\
& \hspace{1cm}<  \frac{1}{2a^2-12a+17}\left(2a^2-9a+10-(a^2-2a-1)\frac{1}{a-1}-(a-2)\frac{1}{a^2-2a+1}\right)
\\
&\hspace{1.5cm}= \frac{1}{2a^2-12a+17}\left(2a^2-9a+10-a+1+\frac{2}{a-1}-\frac{a-2}{a^2-2a+1}\right)
\\
 &\hspace{2cm}= 1+ \frac{1}{2a^2-12a+17}\left(2a-6+\frac{a}{a^2-2a+1}\right)  <  2.
\end{align*}
For $\alpha_2^{(3)}$, we get 
\begin{align*}
\alpha_2^{(3)} &= \frac{1}{2a^2-12a+17}(2a^2-8a+8+(2a^2-7a+6)\psi''-(2a-5)\psi''^2) 
\\
&\hspace{1cm} >  \frac{1}{2a^2-12a+17}\left(2a^2-8a+8-(2a^2-7a+6)\frac{1}{a-2}-(2a-5)\frac{1}{a^2-4a+4}\right)
\\
&\hspace{1.5cm}= \frac{1}{2a^2-12a+17}\left(2a^2-8a+8-2a+3-(2a-5)\frac{1}{a^2-4a+4}\right)
\\
 &\hspace{2cm}= 1+ \frac{1}{2a^2-12a+17}\left(2a-6-(2a-5)\frac{1}{a^2-4a+4} \right)  > 1 
\end{align*}
and
\begin{align*}
\alpha_2^{(3)} &= \frac{1}{2a^2-12a+17}(2a^2-8a+8+(2a^2-7a+6)\psi''-(2a-5)\psi''^2) 
\\
& \hspace{1cm}<  \frac{1}{2a^2-12a+17}\left(2a^2-8a+8-(2a^2-7a+6)\frac{1}{a-1}-(2a-5)\frac{1}{a^2-2a+1}\right)
\\
&\hspace{1.5cm}= \frac{1}{2a^2-12a+17}\left(2a^2-8a+8-2a+5-\frac{1}{a-1}-(2a-5)\frac{1}{a^2-2a+1}\right)
\\
 &\hspace{2cm}= 1+ \frac{1}{2a^2-12a+17}\left(2a-4-\frac{3a-6}{a^2-2a+1}\right)  < 2.
\end{align*}
Hence, $\bm{a}^{(3)}=(1,1)$. Now,  
it is easy to check that
\begin{align*}
\alpha_1^{(4)} &= \frac{\alpha_2^{(3)}-a_2^{(3)}}{\alpha_1^{(3)}-a_1^{(3)}}  =  \frac{ \frac{1}{2a^2-12a+17}(2a^2-8a+8+(2a^2-7a+6)\psi''-(2a-5)\psi''^2)-1}{\frac{1}{2a^2-12a+17}(2a^2-9a+10+(a^2-2a-1)\psi''-(a-2)\psi''^2)-1} 
\\
&= 1- \psi''
\end{align*}
and
\begin{align*}
\alpha_2^{(4)} &= \frac{1}{\alpha_1^{(3)}-a_1^{(3)}} =  \frac{1}{\frac{1}{2a^2-12a+17}(2a^2-9a+10+(a^2-2a-1)\psi''-(a-2)\psi''^2)-1} 
\\
 &=  a-2+a\psi''-\psi''^2.
\end{align*}
Thus, the fourth iteration is $(1- \psi'', a-2+a\psi''-\psi''^2)$.

\bigskip

\noindent
\textbf{Fifth iteration}
\nopagebreak

In this part, we will show that $\bm{a}^{(4)}=(1,a-4)$. From the first iteration, we know that $\lfloor -\psi'' \rfloor = 0$. Thus, it is clear that $a_1^{(4)}=\lfloor 1-\psi'' \rfloor = 1$. Therefore, it is enough to show that $a_2^{(4)}= a-4$. We see that
\begin{multline*}
\alpha_2^{(4)} = a-2+a\psi''-\psi''^2 > a-2-\frac{a}{a-2}-\frac{1}{a^2-4a+4}
\\
= a-2-1-\frac{2}{a-2}-\frac{1}{a^2-4a+4}  =  a-3-\frac{2a-3}{a^2-4a+4} >a-4
\end{multline*}
and
\begin{multline*}
\alpha_2^{(4)} = a-2+a\psi''-\psi''^2 < a-2-\frac{a}{a-1}-\frac{1}{a^2-2a+1}
\\
= a-2-1-\frac{1}{a-1}-\frac{1}{a^2-2a+1}  =  a-3-\frac{a}{a^2-2a+1} <a-3.
\end{multline*}
It is easy to verify that
\begin{eqnarray*}
\alpha_1^{(5)} = \frac{\alpha_2^{(4)}-a_2^{(4)}}{\alpha_1^{(4)}-a_1^{(4)}}  =  \frac{ a-2+a\psi''-\psi''^2-a+4}{1- \psi''-1} 
= a + (2a-1)\psi''-2\psi''^2
\end{eqnarray*}
and
\begin{eqnarray*}
\alpha_2^{(5)} = \frac{1}{\alpha_1^{(4)}-a_1^{(4)}}  =  \frac{1}{1- \psi''-1}
 =  a+(a-1)\psi''-\psi''^2.
\end{eqnarray*}
Therefore, $\bm{\alpha}^{(5)}=(a + (2a-1)\psi''-2\psi''^2, a+(a-1)\psi''-\psi''^2)$.

\bigskip

\noindent
\textbf{Sixth iteration}
\nopagebreak

From the third iteration, we know that $\lfloor a+1 + (2a-1)\psi'-2\psi'^2  \rfloor = a-2$. Thus, $a_1^{(5)}=\lfloor a + (2a-1)\psi''-2\psi''^2 \rfloor = a-3 $. Moreover, from the second iteration, we also know that $a_2^{(5)}=\lfloor a+(a-1)\psi''-\psi''^2 \rfloor = a-2$. Therefore, $\bm{a}^{(5)}=(a-3,a-2)$. 
Now it easy to check that
\begin{align*}
\alpha_1^{(6)} &= \frac{\alpha_2^{(5)}-a_2^{(5)}}{\alpha_1^{(5)}-a_1^{(5)}}  =  \frac{a+(a-1)\psi''-\psi''^2-a+2}{a + (2a-1)\psi''-2\psi''^2-a+3} 
\\
&= \frac{1}{2a^2-12a+17}(2a^2-9a+10+(a^2-2a-1)\psi''-(a-2)\psi''^2)
\end{align*}
and
\begin{align*}
\alpha_2^{(6)} &= \frac{1}{\alpha_1^{(5)}-a_1^{(5)}}  = \frac{1}{a + (2a-1)\psi''-2\psi''^2-a+2} 
\\
 &=  \frac{1}{2a^2-12a+17}(2a^2-8a+8+(2a^2-7a+6)\psi''-(2a-5)\psi''^2).
\end{align*}

Therefore, we have proved that 
\begin{multline*}
\bm{\alpha}^{(6)}= \Big(\frac{1}{2a^2-12a+17}(2a^2-9a+10+(a^2-2a-1)\psi''-(a-2)\psi''^2),\\ \frac{1}{2a^2-12a+17}(2a^2-8a+8+(2a^2-7a+6)\psi''-(2a-5)\psi''^2)\Big).
\end{multline*}
We can see that $\bm{\alpha}^{(6)}=\bm{\alpha}^{(3)}$, which means that the iJPA expansion of the vector $(-\psi'', \psi''^2)$ is periodic with period length 3 and preperiod length 3.
\end{proof}

\section{Indecomposability of (semi)convergents}

In this part, we show several results about semiconvergents and their indecomposability.

\subsection{The simplest cubic fields}

We again start with the simplest cubic fields. In Section \ref{sec:semi}, we did not describe in detail the cases when $-1\leq a\leq 3$. In particular, we have omitted the comparison of semiconvergents and $\mathfrak{s}$-indecomposables. In all these cases, we obtain only $\mathfrak{s}$-indecomposables. However, the JPA expansion is different for these a few values of $a$. Our results are contained in Table~\ref{tab:semi a=-1} for $a=-1$, in Table~\ref{tab:semi a=0} for $a=0$, and in Table~\ref{tab:semi 1<=a<=3} for $1\leq a\leq 3$.   

\begin{table}
\centering
\begin{tabular}{|c|c|c|c|}
\hline
 $\delta_{i,j}^{(k)}$ & Norm $N(\delta_{i,j}^{(k)})$ & Indecomposable integer $\alpha$ & Unit $\varepsilon$ \\
\hline
$\delta_{3,1}^{(0)}$ & $-1$ & $1$ & $-\rho\rho'$ \\
\hline 
$\delta_{3,2}^{(0)}$ & $1$ & $1$ & $\rho$ \\
\hline
$\delta_{3,0}^{(2)}$ & $1$ & $1$ & $\rho^{-1}$ \\
\hline
$\delta_{3,1}^{(2)}$ & $-1$ & $1$ & $-\rho'^{-1}$ \\
\hline
$\delta_{3,2}^{(2)}$ & $1$ & $1$ & $\rho'^{-2}$ \\
\hline
\end{tabular}
\caption{\textsc{The simplest cubic fields}: Semiconvergents $\delta_{i,j}^{(k)}$ of the JPA expansion of $(1,-\rho',\rho'^2)$ for $a=-1$, their norm, the corresponding indecomposable integer $\alpha$ and the unit $\varepsilon$ such that $\delta_{i,j}^{(k)}=\alpha\varepsilon$.} \label{tab:semi a=-1}
\end{table}

\begin{table}
\centering
\begin{tabular}{|c|c|c|c|}
\hline
 $\delta_{i,j}^{(k)}$ & Norm $N(\delta_{i,j}^{(k)})$ & Indecomposable integer $\alpha$ &  Unit $\varepsilon$ \\
\hline
$\delta_{3,1}^{(0)}$ & $-3$ & $\theta'_{0,0}$ & $-\rho^{-1}\rho'^{-1}$  \\
\hline 
$\delta_{2,0}^{(2)}$ & $3$ & $\theta'_{0,0}$ & $\rho^{-2}\rho'^{-2}$ \\
\hline
$\delta_{3,0}^{(2)}$ & $1$ & $1$ & $\rho^{-1}$\\ 
\hline
$\delta_{2,0}^{(3)}$ & $-1$ & $1$ & $-\rho^{-2}\rho'^{-1}$ \\
\hline
$\delta_{3,0}^{(3)}$ & $-1$ & $1$ & $-\rho^{-1}\rho'^{-1}$ \\
\hline
$\delta_{3,1}^{(3)}$ & $1$ & $1$ & $\rho^{-1}\rho'^{-2}$ \\
\hline
$\delta_{2,0}^{(4)}$ & $3$ & $\theta'_{0,0}$ & $\rho^{-3}\rho'^{-4}$\\
\hline
$\delta_{3,0}^{(4)}$ & $1$ & $1$ & $\rho^{-2}\rho'^{-2}$\\ 
\hline
\end{tabular}
\caption{\textsc{The simplest cubic fields}: Semiconvergents $\delta_{i,j}^{(k)}$ of the JPA expansion of $(1,-\rho',\rho'^2)$ for $a=0$, their norm, the corresponding indecomposable integer $\alpha$ and the unit $\varepsilon$ such that $\delta_{i,j}^{(k)}=\alpha\varepsilon$.} \label{tab:semi a=0}
\end{table} 

\begin{table}
\centering
\begin{tabular}{|c|c|c|c|}
\hline
 $\delta_{i,j}^{(k)}$ & Norm $N(\delta_{i,j}^{(k)})$ & Indecomposable $\alpha$ & Unit $\varepsilon$ \\
\hline
$\delta_{2,0}^{(1)}$ & $-(2a+3)$ & $\theta'_{0,0}$ & $-\rho^{-1}\rho'^{-1}$ \\
\hline 
$\delta_{2,1}^{(1)}$ & $-1$ & $1$ & $-\rho^{-1}\rho'$ \\
\hline
$\delta_{3,j}^{(1)}$  &  &  &  \\
$1\leq j\leq a$ & $-j^3+aj^2+(a+3)j+1$ & $\theta'_{a-j+1,0}$ & $\rho^{-2}\rho'^{-2}$\\
\hline
$\delta_{2,0}^{(2)}$ & $2a+3$ &  $\theta'_{0,0}$ & $\rho^{-2}\rho'^{-2}$\\
\hline
$\delta_{3,0}^{(2)}$ & $1$ & $1$ & $\rho^{-1}$\\
\hline
$\delta_{3,j}^{(2)}$  &  &  &  \\
$1\leq j\leq a$ & $-j^3+aj^2+(a+3)j+1$ & $\theta'_{a-j+1,0}$ & $\rho^{-3}\rho'^{-2}$ \\
\hline
$\delta_{3,0}^{(3)}$ & $1$ &  $1$ & $\rho^{-2}$\\
\hline
$\delta_{3,0}^{(4)}$ & $-1$ & $1$ & $-\rho^{-2}\rho'^{-1}$\\
\hline 
$\delta_{2,j}^{(5)}$  & $j^3-(2a-3)j^2$  &  & \\
$0\leq j\leq a-2$ & $+(a^2-5a)j+2a^2-1$ & $\theta'_{a-j-2,a-(a-j-2)}$ & $\rho^{-4}\rho'^{-2}$\\
\hline
$\delta_{3,0}^{(5)}$ & $2a+3$ &  $\theta'_{0,0}$ & $\rho^{-3}\rho'^{-2}$\\
\hline
$\delta_{3,1}^{(5)}$ & $1$ & $1$ & $\rho^{-3}$\\
\hline
$\delta_{2,0}^{(6)}$ & $-1$ & $1$ & $-\rho^{-4}\rho'^{-1}$\\
\hline
$\delta_{3,0}^{(6)}$ & $-1$ & $1$ & $-\rho^{-3}\rho'^{-1}$ \\
\hline
$\delta_{3,1}^{(6)}$ & $1$ & $1$ & $\rho^{-3}\rho'^{-2}$\\
\hline
$\delta_{3,j}^{(6)}$  &  & &  \\
$2\leq j\leq a+1$ & $-j^3+(a+3)j^2-aj-1$ & $\theta'_{a-j+2,0}$ & $\rho^{-5}\rho'^{-4}$\\
\hline
\end{tabular}
\caption{\textsc{The simplest cubic fields}: Semiconvergents $\delta_{i,j}^{(k)}$ of the JPA expansion of $(1,-\rho',\rho'^2)$ for $1\leq a\leq 3$, their norm and the corresponding indecomposable integer $\alpha$ and the unit $\varepsilon$ such that $\delta_{i,j}^{(k)}=\alpha\varepsilon$.} \label{tab:semi 1<=a<=3}
\end{table}

\subsection{Ennola's cubic fields}
Let $\rho''<\rho'<\rho$ be roots of the 
polynomial $x^3+(a-1)x^2-ax-1$ where $a\geq 3$. Then $\mathfrak{s}$-indecomposable integers in $\Z[\rho]$ are associated with elements of the form
\begin{enumerate}
\item $\kappa_w= 1+w \rho + \rho^2$ where $1\leq w \leq a-1 $,
\item $\lambda_{v,w}=-v-(a(v-1)+w)\rho+(a(v-1)+w+1)\rho^2$ where $1\leq v \leq a-1$ and $\max\{1,v-1\}\leq w \leq a-1$,
\item $\mu_u=-1-u\rho+(u+2)\rho^2$ where $0\leq u \leq a-2$.
\end{enumerate}

Moreover, in this part, we will consider roots of the polynomial $x^3-(a-1)x^2-ax-1$ where $a\geq 5$. Recall that we denote these roots as $\psi'<\psi''<\psi$. Then (after multiplication by some particular units and shifting of indeces) the $\mathfrak{s}$-indecomposables above can be rewritten as
 \begin{enumerate}
\item $\Tilde{\kappa}_w=1-w+aw+(1-w+aw)\psi - w \psi^2$ where $1\leq w \leq a-3$,
\item $\Tilde{\lambda}_{v,u}=1+v-u+au+(a-u+au)\psi-(u+1)\psi^2$ where $1\leq v \leq a-3$, $ 0 \leq u \leq v$ and $(v,u)\neq (1,0)$,
\item $\Tilde{\mu}_{z}=z+2+(z+4)\psi+\psi^2$ where $0\leq z \leq a-4$.
\end{enumerate}

Recall that in Section \ref{sec:semi}, we have discussed indecomposability of semiconvergents of JPA expansions of vectors corresponding to each of these roots and show the proof only for two of them. Now, we will also focus on the remaining cases.  

\begin{theorem}
We have
\begin{enumerate}
    \item all semiconvergents of JPA expansions of vectors $(1,-\rho',\rho'^2)$, $(1,-\psi',\psi'^2)$ and $(1,-\psi'',\psi''^2)$ are $\mathfrak{s}$-indecomposable in their signatures,
    \item some semiconvergent of JPA expansions of vectors $(1,\rho,\rho^2)$, $(1,-\rho'',\rho''^2)$ and $(1,\psi,\psi^2)$ is $\mathfrak{s}$-decomposable in its signature.
\end{enumerate}
\end{theorem}

\begin{proof}
The results for $(1,\rho,\rho^2)$ and $(1,-\rho'',\rho''^2)$ are in Tables \ref{tab:ennola2} and \ref{tab:ennola1} in Section \ref{sec:semi}. The semiconvergents of $(1,-\rho',\rho'^2)$ are discuss in Table \ref{tab:enn_firstpoly_2ndroot}, of $(1,\psi,\psi^2)$ in Table \ref{tab:enn_2ndpoly_1stroot}, of $(1,-\psi',\psi'^2)$ in Table \ref{tab:enn_2ndpoly_2ndroot}, and of $(1,-\psi'',\psi''^2)$ in Table \ref{tab:enn_2ndpoly_3rdroot}.      
\end{proof}

\begin{table}[htbp]
\centering
{\fontsize{9.9}{12} \selectfont
\begin{tabular}{|c|c|c|c|}
\hline
 $\delta_{i,j}^{(k)}$ & Norm $N(\delta_{i,j}^{(k)})$ & $\mathfrak{s}$-indecomposable integer $\alpha$ & Unit $\varepsilon$ \\
   \hline
 $\delta_{3,1}^{(1)}$ & $-2a+3$ & $\lambda'_{a-1,a-1}$ & $\rho'^{-2}(\rho'-1)^{-1}$  \\
 \hline
 $\delta_{3,j}^{(1)}$  & &   &  \\
   $2 \leq j \leq a-1$ &  $j^3-aj^2-(a-1)j+1$   & $\lambda'_{1,j-1}$   & $(\rho'-1)^{-1}$   \\
\hline
 $\delta_{2,j}^{(2)}$ & $-j^3+(2a-2)j^2-(a^2-3a+1)j$ &  &  \\
   $0\leq j \leq a-3$   & $-a^2+a+1$   & $\lambda'_{j+1,a-1}$   & $(\rho'-1)^{-1}$   \\
\hline
 $\delta_{3,1}^{(2)}$ & $2a-3$ & $\lambda'_{a-1,a-1}$ & $-\rho'^{-1}(\rho'-1)^{-1}$  \\
 \hline
 $\delta_{3,j}^{(2)}$  &  &  &  \\
   $2\leq j \leq a-1$  & $ - j^3+aj^2+(a-1)j-1$   & $\lambda'_{1,j-1}$   & $-\rho'(\rho'-1)^{-1}$   \\
  \hline
   $\delta_{3,0}^{(4)}$ & $ - 2 a+3$ & $\lambda'_{a-1,a-1}$  & $(\rho'-1)^{-1}$ \\
 \hline
   $\delta_{3,j}^{(5)}$  & $j^3-(2a-2)j^2+(a^2-3a+1)j$  &  &  \\
   $0\leq j \leq a-3$  & $+a^2-a-1$   & $\lambda'_{j+1,a-1}$   & $-\rho'(\rho'-1)^{-1}$   \\
  \hline
 $\delta_{3,0}^{(6)}$ & $-1$  & $1$ & $-\rho'^3$ \\
\hline
 $\delta_{3,0}^{(7)}$ & $1$ & $1$  &  $\rho'^3(\rho'-1)^{-1}$ \\
 \hline
 $\delta_{3,0}^{(8)}$ & $2a-3$ & $\lambda'_{a-1,a-1}$ & $-\rho'(\rho'-1)^{-1}$ \\
 \hline 
  $\delta_{3,1}^{(8)}$ & $-a^2+a+1$ & $\lambda'_{1,a-1}$ & $\rho'^2(\rho'-1)^{-1}$ \\
 \hline  
 $\delta_{3,j}^{(8)}$  & $j^3-j^2-(a^2+a-4)j$ &  &   \\
   $2\leq j \leq a-1$  & $+2a-3$    & $\lambda'_{2,j-1}$   & $\rho'^3(\rho'-1)^{-2}$   \\
 \hline
 $\delta_{2,j}^{(9)}$ & $-j^3+(2a-5)j^2-(a^2-7a+8)j $ &  &  \\
   $0\leq j \leq a-4$   & $-2a^2+6a-3$   & $\lambda'_{j+2,a-1}$   & $\rho'^3(\rho'-1)^{-2}$   \\
 \hline
 $\delta_{3,0}^{(9)}$ & $-1$ & $1$ &  $-\rho'^4(\rho'-1)^{-1}$ \\
 \hline
  $\delta_{3,1}^{(9)}$ & $2a-3$ & $\lambda'_{a-1,a-1}$ &  $-\rho'^2(\rho'-1)^{-2}$ \\
 \hline
 $\delta_{3,j}^{(9)}$  &  &  &  \\
  $2\leq j \leq a-1$  & $-j^3+aj^2+(a-1)j-1$   & $\lambda'_{1,j-1}$   &  $-\rho'^4(\rho'-1)^{-2}$ \\
 \hline
\end{tabular}
}
\caption{\textsc{Ennola's cubic fields}: Semiconvergents $\delta_{i,j}^{(k)}$ of the JPA expansion of $(1,-\rho',\rho'^2)$ where $-1<\rho'<0$ is a root of the polynomial $x^3+(a-1)x^2-ax-1$ where $a \geq 3$, their norm, the corresponding $\mathfrak{s}$-indecomposable integer $\alpha$ and the unit $\varepsilon$ such that $\delta_{i,j}^{(k)}=\alpha\varepsilon$.} \label{tab:enn_firstpoly_2ndroot}
\end{table}

\begin{table}[htbp]
\centering
{\fontsize{9.9}{12} \selectfont
\begin{tabular}{|c|c|c|c|}
\hline
 $\delta_{i,j}^{(k)}$ & Norm $N(\delta_{i,j}^{(k)})$ & $\mathfrak{s}$-indecomposable integer $\alpha$ & Unit $\varepsilon$ \\
  \hline
   $\delta_{2,j}^{(0)}$  &  &  &  \\
  $1\leq j \leq a-3$  & $-j^3+(a-1)j^2+aj+1$   & $\tilde{\kappa}_{j}$     &  $\rho^{-1}$ \\
   \hline
   $\delta_{2,j}^{(0)}$  &  &  &  \\
  $a-2\leq j \leq a-1$  & $-j^3+(a-1)j^2+aj+1$   & \xmark &   \\
 \hline
   $\delta_{3,1}^{(0)}$ & $2a-1$ & $\tilde{\kappa}_1$ & $(\rho-1)^{-1}$  \\  
   \hline
   $\delta_{3,j}^{(0)}$  &  &  &  \\
  $2\leq j \leq a^2-1$  & $-j^3+(a^2+1)j^2-(a^2-2a+2)j+1$   & \xmark     &   \\
 \hline
    $\delta_{2,j}^{(1)}$  & $-j^3+(4a-1)j^2-(5a^2-3a)j$  &  &  \\
  $0\leq j \leq a-2$  &  $+2a^3-2a^2+1$   & \xmark     &   \\
 \hline
 $\delta_{2,a-1}^{(1)}$ & $1$ & $1$ & $\rho-1$  \\
 \hline
 $\delta_{2,a}^{(1)}$ & $1$ & $1$ & $\rho^{-1}(\rho-1)$  \\
 \hline
 $\delta_{2,j}^{(1)}$  & $-j^3+(4a-1)j^2-(5a^2-3a)j$  &  &  \\
  $a+1\leq j \leq 2a-3$  &  $+2a^3-2a^2+1$   & $\tilde{\kappa}_{j-a}$     & $\rho^{-2}(\rho-1)^2$  \\
 \hline
  $\delta_{2,j}^{(1)}$  & $-j^3+(4a-1)j^2-(5a^2-3a)j$  &  &  \\
  $2a-2\leq j \leq 2a-1$  &  $+2a^3-2a^2+1$   & \xmark     &   \\
 \hline
  $\delta_{3,j}^{(1)}$  &   &  &  \\
  $1\leq j \leq a^2+a-1$  &  $-j^3+(a^2+a)j^2+(2a+1)j+1$   & \xmark     &   \\
 \hline
     $\delta_{2,j}^{(2)}$  & $-j^3+(4a+2)j^2-(5a^2+5a+1)j$  &  &  \\
  $0\leq j \leq a-1$  &  $+2a^3+3a^2+a+1$   & \xmark     &   \\
 \hline
 $\delta_{2,a}^{(2)}$ & $1$ & $1$ & $\rho^{-1}(\rho-1)^3$  \\
 \hline
 $\delta_{2,a+1}^{(2)}$ & $1$ & $1$ & $\rho^{-2}(\rho-1)^3$  \\
 \hline
 $\delta_{2,j}^{(2)}$  & $-j^3+(4a+2)j^2-(5a^2+5a+1)j$  &  &  \\
  $a+2\leq j \leq 2a-2$  &  $+2a^3+3a^2+a+1$   & $\tilde{\kappa}_{j-a-1}$     & $\rho^{-3}(\rho-1)^4$  \\
 \hline
 $\delta_{2,j}^{(2)}$  & $-j^3+(4a+2)j^2-(5a^2+5a+1)j$  &  &  \\
  $2a-1\leq j \leq 2a$  &  $+2a^3+3a^2+a+1$   & \xmark     &   \\
 \hline
  $\delta_{3,0}^{(2)}$ & $1$ & $1$ & $\rho^{-1}(\rho-1)^2$  \\
 \hline
  $\delta_{3,j}^{(2)}$  &   &  &  \\
  $1\leq j \leq a^2+a-1$  &  $-j^3+(a^2+a)j^2+(2a+1)j+1$   & \xmark     &   \\
 \hline 
\end{tabular}
}
\caption{\textsc{Ennola's cubic fields}: Semiconvergents $\delta_{i,j}^{(k)}$ of the JPA expansion of $(1,\psi,\psi^2)$ where $a<\psi$ is a root of the polynomial $x^3-(a-1)x^2-ax-1$ where $a \geq 5$, their norm, the corresponding $\mathfrak{s}$-indecomposable integer $\alpha$ and the unit $\varepsilon$ such that $\delta_{i,j}^{(k)}=\alpha\varepsilon$.} \label{tab:enn_2ndpoly_1stroot}
\end{table}

\begin{table}[htbp]
\centering
{\fontsize{9.9}{12} \selectfont
\begin{tabular}{|c|c|c|c|}
\hline
 $\delta_{i,j}^{(k)}$ & Norm $N(\delta_{i,j}^{(k)})$ & $\mathfrak{s}$-indecomposable integer $\alpha$ & Unit $\varepsilon$ \\
 \hline
$\delta_{3,j}^{(3)}$  &   &  &  \\
  $1\leq j \leq a-4$  &  $-j^3+(2a-3)j^2-(a^2-3a+2)j+1$   & $\tilde{\lambda}'_{j,j}$ & $\rho'(\rho'-1)^{-2}$   \\
 \hline 
  $\delta_{3,0}^{(4)}$ & $1 $ & $1 $ & $\rho'(\rho'-1)^{-1}$ \\
\hline
$\delta_{3,0}^{(5)}$ & $-1$ & $1 $ & $-\rho'(\rho'-1)^{-2}$ \\
\hline
$\delta_{3,0}^{(6)}$ & $-2a+7 $ & $\tilde{\lambda}'_{a-3,a-3} $ & $\rho'(\rho'-1)^{-2}$ \\
 \hline
   $\delta_{3,j}^{(6)}$ &  &  &  \\
   $1\leq j \leq a-3$  & $-j^3+j^2+(a^2-3a-2)j-2a+7$   & $\tilde{\lambda}'_{j,1}$   & $-\rho'^2(\rho'-1)^{-2}$   \\
    \hline
 $\delta_{2,j}^{(7)}$  & $j^3-(2a-9)j^2+(a^2-11a+26)j$ & &   \\
   $0\leq j \leq a-6$ & $+2a^2-14a+23$   & $\tilde{\lambda}'_{j+2,j+2}$   & $-\rho'^3(\rho'-1)^{-3}$  \\
 \hline
   $\delta_{3,0}^{(7)}$ & $1 $ & $1 $ & $\rho'^{2}(\rho'-1)^{-2}$ \\
 \hline
  $\delta_{3,1}^{(7)}$ & $-2a+7 $ & $\tilde{\lambda}'_{a-3,a-3} $ & $\rho'^{2}(\rho'-1)^{-3}$ \\
 \hline
   $\delta_{3,j}^{(7)}$  &  &  &   \\
  $2\leq j \leq a-3$ & $j^3-(a-2)j^2-(a-3)j+1$ & $\tilde{\lambda}'_{j,0}$   &  $\rho'^{3}(\rho'-1)^{-2} $ \\
 \hline 
 $\delta_{3,0}^{(8)}$ & $-1 $ & $1 $ & $-\rho'^{3}(\rho'-1)^{-2}$ \\
 \hline
 $\delta_{3,0}^{(9)}$ & $2a-7$ & $\tilde{\lambda}'_{a-3,a-3} $ & $-\rho'^3(\rho'-1)^{-3}$ \\
 \hline 
   $\delta_{3,j}^{(10)}$ & $-j^3+(2a-6)j^2-(a^2-7a+11)j$  &   & \\
   $0\leq j \leq a-5$  & $-a^2+5a-5$   & $\tilde{\lambda}'_{j+1,j+1}$   & $\rho'^4(\rho'-1)^{-3}$   \\
  \hline
 \end{tabular}
}
\caption{\textsc{Ennola's cubic fields}: Semiconvergents $\delta_{i,j}^{(k)}$ of the JPA expansion of $(1,-\psi',\psi'^2)$ where $-1<\psi'<-\frac{1}{2}$ is a root of the polynomial $x^3-(a-1)x^2-ax-1$ where $a \geq 5$, their norm, the corresponding $\mathfrak{s}$-indecomposable integer $\alpha$ and the unit $\varepsilon$ such that $\delta_{i,j}^{(k)}=\alpha\varepsilon$.} \label{tab:enn_2ndpoly_2ndroot}
\end{table}

\begin{table}[htbp]
\centering
{\fontsize{9.9}{12} \selectfont
\begin{tabular}{|c|c|c|c|}
\hline
 $\delta_{i,j}^{(k)}$ & Norm $N(\delta_{i,j}^{(k)})$ & $\mathfrak{s}$-indecomposable integer $\alpha$ & Unit $\varepsilon$ \\
  \hline
   $\delta_{3,1}^{(1)}$ & $1 $ & $1$ & $\rho''(\rho''-1)^{-1}$ \\
 \hline
  $\delta_{3,j}^{(1)}$  &  &  &  \\
  $2\leq j \leq a-3$ &  $j^3-aj^2+(a-1)j+1$  & $\tilde{\lambda}''_{a-j-1,a-j-1} $   &  $(\rho''-1)^{-2}$  \\
 \hline
   $\delta_{2,0}^{(2)}$ & $-a^2+5a-5$ & $\tilde{\lambda}''_{1,1}$ & $(\rho''-1)^{-2}$ \\
 \hline
   $\delta_{2,j}^{(2)}$  & $-j^3+2j^2+(a^2-3a-3)j $ &  &  \\
   $1\leq j \leq a-3$  & $-a^2+5a-5$   & $\tilde{\lambda}''_{a-3,a-2-j}$   & $-\rho''^{-1}(\rho''-1)^{-2}$   \\
   \hline
   $\delta_{3,1}^{(2)}$ & $-1$ & $1$ & $-\rho''(\rho''-1)^{-2}$ \\
 \hline
    $\delta_{3,j}^{(2)}$ &  &  &   \\
    $2\leq j \leq a-3$  & $-j^3+aj^2-(a-1)j-1$   & $\tilde{\lambda}''_{a-j-1,a-j-1}$   & $-(\rho''-1)^{-3}$   \\
 \hline
  $\delta_{2,0}^{(3)}$ & $a^2-5a+5$ & $\tilde{\lambda}''_{1,1}$ & $-(\rho''-1)^{-3}$ \\
 \hline 
   $\delta_{2,0}^{(4)}$ & $2a-1$ & $\tilde{\kappa}''_{1}$ & $\rho''^{-2}(\rho''-1)^{-2}$ \\
 \hline 
   $\delta_{3,j}^{(4)}$ & $j^3-(2a-7)j^2+(a^2-9a+18)j$ &  &  \\
   $0\leq j \leq a-5$   & $+2a^2-12a+17$   & $\tilde{\lambda}''_{a-3-j,0}$   & $-\rho''^{-1}(\rho''-1)^{-2}$   \\
 \hline
  $\delta_{2,j}^{(5)}$  & $-j^3-j^2+(a^2-3a-2)j $ &  &  \\
   $0\leq j \leq a-4$ & $+2a-7$   & $\tilde{\lambda}''_{a-3,a-3-j}$   &   $-\rho''^{-2}(\rho''-1)^{-3}$ \\
 \hline
    $\delta_{3,0}^{(5)}$  & $-1 $  & $1$ & $-\rho''^{-1}(\rho''-1)^{-2}$ \\
 \hline
  $\delta_{3,1}^{(5)}$  & $-1 $  & $1$ & $-(\rho''-1)^{-3}$ \\
 \hline
   $\delta_{3,j}^{(5)}$  &   &   &   \\
   $2\leq j \leq a-3$  & $-j^3+aj^2-(a-1)j-1$   & $\tilde{\lambda}''_{a-j-1,a-j-1}$   & $-\rho''^{-1}(\rho''-1)^{-4}$   \\
 \hline
\end{tabular}
}
\caption{\textsc{Ennola's cubic fields}: Semiconvergents $\delta_{i,j}^{(k)}$ of the JPA expansion of $(1,-\psi'',\psi''^2)$ where $-\frac{1}{2}<\psi''<0$ is a root of the polynomial $x^3-(a-1)x^2-ax-1$ where $a \geq 5$, their norm, the corresponding $\mathfrak{s}$-indecomposable integer $\alpha$ and the unit $\varepsilon$ such that $\delta_{i,j}^{(k)}=\alpha\varepsilon$.} \label{tab:enn_2ndpoly_3rdroot}
\end{table}

\section{Experiments on indecomposability of semiconvergents}

In this section, we show more results on indecomposability of semiconvergents for cubic number fields. In particular, Table \ref{tab:exp2} is a continuation of Table \ref{tab:exp1} from Section \ref{sec:exper}.

\begin{table}
\centering
\begin{tabular}{|c|c|c|c|c|c|c|c|c|}
\hline
$\Delta_K$ & Polynomial & $\rho$ & $\rho'$ & $\rho''$ & Preperiod & Period & Conv. & Semiconv.\\
\hline
\multirow{1}{*}{2177} & $x^3+8x^2+13x-1$ & $0.074$ & $-2.393$ & $-5.68$ & $2$ & $1$ & \xmark & \xmark\\
\hline
\multirow{3}{*}{2233} & $x^3-x^2-8x+1$ & $0.123$ & $3.319$ & $-2.443$ & $5$ & $3$ & \cmark & \xmark\\
\cline{2-9}
& $x^3-8x^2-x+1$ & $0.301$ & $8.108$ & $-0.409$ & $4$ & $2$ & \xmark & \xmark\\
\cline{2-9}
& $x^3+8x^2-x-1$ & $0.409$ & $-0.301$ & $-8.108$ & $4$ & $7$ & \cmark & \xmark\\
\hline
\multirow{1}{*}{2804} & $x^3+x^2-9x+1$ & $0.113$ & $2.474$ & $-3.587$ & $3$ & $3$ & \cmark & \cmark\\
\hline
\multirow{4}{*}{2857} & $x^3+2x^2-9x+1$ & $0.114$ & $2.086$ & $-4.2$ & $3$ & $3$ & \cmark & \cmark\\
\cline{2-9}
& $x^3-9x^2+2x+1$ & $0.479$ & $8.759$ & $-0.238$ & $11$ & $2$ & \cmark & \cmark\\
\cline{2-9}
& $x^3+9x^2+2x-1$ & $0.238$ & $-0.479$ & $-8.759$ & $2$ & $30$ & \xmark & \xmark\\
\cline{2-9}
& $x^3+8x^2+11x-1$ & $0.085$ & $-1.886$ & $-6.2$ & $2$ & $1$ & \xmark & \xmark\\
\hline
\multirow{2}{*}{2993} & $x^3-8x^2+9x+1$ & $1.485$ & $-0.102$ & $6.617$ & $9$ & $12$ & \cmark & \cmark\\
\cline{2-9}
& $x^3+8x^2+9x-1$ & $0.102$ & $-1.485$ & $-6.617$ & $2$ & $1$ & \xmark & \xmark\\
\hline
\multirow{3}{*}{3137} & $x^3-3x^2-8x+1$ & $0.12$ & $-1.788$ & $4.668$ & $5$ & $17$ & \xmark & \xmark\\
\cline{2-9}
& $x^3-8x^2-3x+1$ & $0.214$ & $-0.559$ & $8.345$ & $20$ & $8$ & \xmark & \xmark\\
\cline{2-9}
& $x^3+8x^2-3x-1$ & $0.559$ & $-0.214$ & $-8.345$ & $2$ & $10$ & \cmark & \xmark\\
\hline
\multirow{1}{*}{3281} & $x^3+4x^2-9x+1$ & $0.117$ & $1.513$ & $-5.63$ & $3$ & $11$ & \xmark & \xmark\\
\hline
\multirow{2}{*}{3569} & $x^3-9x^2-2x+1$ & $0.243$ & $-0.448$ & $9.205$ & $7$ & $5$ & \xmark & \xmark\\
\cline{2-9}
& $x^3+9x^2-2x-1$ & $0.448$ & $-0.243$ & $-9.205$ & $6$ & $6$ & \xmark & \xmark\\
\hline
\multirow{3}{*}{3604} & $x^3+5x^2-9x+1$ & $1.306$ & $0.119$ & $-6.425$ & $2$ & $21$ & \cmark & \cmark\\
\cline{2-9}
&$x^3+5x^2-9x+1$ & $0.119$ & $1.306$ & $-6.425$ & $9$ & $30$ & \xmark & \xmark\\
\cline{2-9}
&$x^3+9x^2+5x-1$ & $0.156$ & $-0.766$ & $-8.39$ & $2$ & $13$ & \cmark & \cmark\\
\hline
\multirow{1}{*}{3892} & $x^3+9x^2+17x-1$ & $0.057$ & $-2.799$ & $-6.258$ & $2$ & $1$ & \xmark & \xmark\\
\hline
\multirow{2}{*}{3969} & $x^3+6x^2-9x+1$ & $1.138$ & $0.121$ & $-7.259$ & $2$ & $7$ & \cmark & \cmark\\
\cline{2-9}
& $x^3+6x^2-9x+1$ & $0.121$ & $1.138$ & $-7.259$ & $3$ & $7$ & \cmark & \cmark\\
\hline
\multirow{1}{*}{4065} & $x^3+10x^2+23x-1$ & $0.043$ & $-3.685$ & $-6.357$ & $2$ & $1$ & \xmark & \xmark\\
\hline
\multirow{1}{*}{5297} & $x^3-11x+1$ & $0.091$ & $3.27$ & $-3.361$ & $3$ & $3$ & \cmark & \cmark\\
\hline
\multirow{2}{*}{5369} & $x^3-9x^2+10x+1$ & $1.411$ & $-0.092$ & $7.681$ & $4$ & $4$ & \xmark & \xmark\\
\cline{2-9}
& $x^3+9x^2+10x-1$ & $0.092$ & $-1.411$ & $-7.681$ & $2$ & $1$ & \xmark & \xmark\\
\hline
\multirow{4}{*}{5684} & $x^3+3x^2-11x+1$ & $0.093$ & $2.073$ & $-5.167$ & $3$ & $3$ & \cmark & \cmark\\
\cline{2-9}
& $x^3+9x^2+13x-1$ & $0.073$ & $-1.907$ & $-7.167$ & $2$ & $1$ & \xmark & \xmark\\
\cline{2-9}
& $x^3-11x^2+3x+1$ & $0.482$ & $-0.194$ & $10.711$ & $11$ & $2$ & \cmark & \cmark\\
\cline{2-9}
& $x^3+11x^2+3x-1$ & $0.194$ & $-0.482$ & $-10.711$ & $2$ & $30$ & \xmark & \xmark\\
\hline
\multirow{2}{*}{5697} & $x^3+9x^2+12x-1$ & $0.079$ & $-1.73$ & $-7.348$ & $2$ & $1$ & \xmark & \xmark\\
\cline{2-9}
& $x^3+12x^2-9x+1$ & $0.136$ & $0.578$ & $-12.714$ & $3$ & $8$ & \xmark & \xmark\\
\hline
\multirow{1}{*}{6185} & $x^3+4x^2-11x+1$ & $0.094$ & $1.801$ & $-5.895$ & $3$ & $3$ & \cmark & \cmark\\
\hline
\multirow{2}{*}{6241} & $x^3+7x^2-10x+1$ & $1.122$ & $0.108$ & $-8.23$ & $2$ & $7$ & \cmark & \cmark\\
\cline{2-9}
& $x^3+7x^2-10x+1$ & $0.108$ & $1.122$ & $-8.23$ & $3$ & $7$ & \cmark & \cmark\\
\hline
\end{tabular}
\caption{Indecomposability of (semi)convergents of JPA expansions of vectors of the form $(1,\rho,\rho^2)$ where $\rho$ is such that $\O_K=\Z[\rho]$; $K$ is an example of a totally real number field with the discriminant $2177\leq\Delta_K\leq 6241$ and units of all signatures} \label{tab:exp2}
\end{table}

\end{document}